\PassOptionsToPackage{unicode}{hyperref}
\PassOptionsToPackage{hyphens}{url}
\PassOptionsToPackage{dvipsnames,svgnames,x11names}{xcolor}
\documentclass[
  12pt]{article}

\usepackage{xr}

\usepackage{amsmath,amssymb,amsfonts,amsthm}
\usepackage{iftex,dsfont}
\ifPDFTeX
  \usepackage[T1]{fontenc}
  \usepackage[utf8]{inputenc}
  \usepackage{textcomp} 
\else 
  \usepackage{unicode-math}
  \defaultfontfeatures{Scale=MatchLowercase}
  \defaultfontfeatures[\rmfamily]{Ligatures=TeX,Scale=1}
\fi
\usepackage{lmodern}
\ifPDFTeX\else  
\fi
\IfFileExists{upquote.sty}{\usepackage{upquote}}{}
\IfFileExists{microtype.sty}{
  \usepackage[]{microtype}
  \UseMicrotypeSet[protrusion]{basicmath} 
}{}
\makeatletter
\@ifundefined{KOMAClassName}{
  \IfFileExists{parskip.sty}{%
    \usepackage{parskip}
  }{
    \setlength{\parindent}{0pt}
    \setlength{\parskip}{6pt plus 2pt minus 1pt}}
}{
  \KOMAoptions{parskip=half}}
\makeatother
\usepackage{xcolor}
\makeatletter
\ifx\paragraph\undefined\else
  \let\oldparagraph\paragraph
  \renewcommand{\paragraph}{
    \@ifstar
      \xxxParagraphStar
      \xxxParagraphNoStar
  }
  \newcommand{\xxxParagraphStar}[1]{\oldparagraph*{#1}\mbox{}}
  \newcommand{\xxxParagraphNoStar}[1]{\oldparagraph{#1}\mbox{}}
\fi
\ifx\subparagraph\undefined\else
  \let\oldsubparagraph\subparagraph
  \renewcommand{\subparagraph}{
    \@ifstar
      \xxxSubParagraphStar
      \xxxSubParagraphNoStar
  }
  \newcommand{\xxxSubParagraphStar}[1]{\oldsubparagraph*{#1}\mbox{}}
  \newcommand{\xxxSubParagraphNoStar}[1]{\oldsubparagraph{#1}\mbox{}}
\fi
\makeatother

\usepackage{longtable,booktabs,array}
\usepackage{calc} 
\usepackage{etoolbox}
\makeatletter
\patchcmd\longtable{\par}{\if@noskipsec\mbox{}\fi\par}{}{}
\makeatother
\IfFileExists{footnotehyper.sty}{\usepackage{footnotehyper}}{\usepackage{footnote}}
\makesavenoteenv{longtable}
\usepackage{graphicx}
\makeatletter
\def\maxwidth{\ifdim\Gin@nat@width>\linewidth\linewidth\else\Gin@nat@width\fi}
\def\maxheight{\ifdim\Gin@nat@height>\textheight\textheight\else\Gin@nat@height\fi}
\makeatother
\setkeys{Gin}{width=\maxwidth,height=\maxheight,keepaspectratio}
\makeatletter
\def\fps@figure{htbp}
\makeatother

\makeatletter
\@ifpackageloaded{caption}{}{\usepackage{caption}}
\AtBeginDocument{%
\ifdefined\contentsname
  \renewcommand*\contentsname{Table of contents}
\else
  \newcommand\contentsname{Table of contents}
\fi
\ifdefined\listfigurename
  \renewcommand*\listfigurename{List of Figures}
\else
  \newcommand\listfigurename{List of Figures}
\fi
\ifdefined\listtablename
  \renewcommand*\listtablename{List of Tables}
\else
  \newcommand\listtablename{List of Tables}
\fi
\ifdefined\figurename
  \renewcommand*\figurename{Figure}
\else
  \newcommand\figurename{Figure}
\fi
\ifdefined\tablename
  \renewcommand*\tablename{Table}
\else
  \newcommand\tablename{Table}
\fi
}
\@ifpackageloaded{float}{}{\usepackage{float}}
\floatstyle{ruled}
\@ifundefined{c@chapter}{\newfloat{codelisting}{h}{lop}}{\newfloat{codelisting}{h}{lop}[chapter]}
\floatname{codelisting}{Listing}

\makeatother
\makeatletter
\@ifpackageloaded{caption}{}{\usepackage{caption}}
\@ifpackageloaded{subcaption}{}{\usepackage{subcaption}}
\makeatother

\ifLuaTeX
  \usepackage{selnolig}  
\fi
\usepackage[]{natbib}
\usepackage{bookmark}
\usepackage{bibunits}
\IfFileExists{xurl.sty}{\usepackage{xurl}}{} 
\hypersetup{
  pdftitle={Title},
  pdfauthor={Author 1; Author 2},
  pdfkeywords={3 to 6 keywords, that do not appear in the title},
  colorlinks=true,
  linkcolor={blue},
  filecolor={Maroon},
  citecolor={Blue},
  urlcolor={Blue},
  pdfcreator={LaTeX via pandoc}}

\newcommand*\samethanks[1][\value{footnote}]{\footnotemark[#1]}

\usepackage{algorithm}
\usepackage{algpseudocode}
\usepackage{rotating}

\newcommand{\RR}{\mathbb{R}}
\newcommand{\EE}{\mathbb{E}}
\newcommand{\PP}{\mathbb{P}}
\newcommand{\NN}{\mathbb{N}}

\newtheorem{lemma}{Lemma}
\newtheorem{thm}{Theorem}
\newtheorem{corollary}{Corollary}
\newtheorem{proposition}{Proposition}

\theoremstyle{definition}

\newtheorem{model}{Model} 

\newtheorem{definition}{Definition}

\newcommand{\R}{\mathbb{R}}

\newcommand{\bA}{\mathbf{A}}
\newcommand{\bB}{\mathbf{B}}

\newcommand{\bP}{\mathbf{P}}

\newcommand{\bU}{\mathbf{U}}
\newcommand{\bV}{\mathbf{V}}
\newcommand{\bW}{\mathbf{W}}
\newcommand{\bX}{\mathbf{X}}
\newcommand{\bY}{\mathbf{Y}}

\newcommand{\Xhat}{\hat{\bX}}

\newcommand{\Yhat}{\hat{\bY}}

\newcommand{\tr}{\operatorname{tr}}

\newcommand{\Var}{\operatorname{Var}}
\newcommand{\Cov}{\operatorname{Cov}}

\newcommand{\ind}{\perp\!\!\!\!\perp}

\newcommand{\inddmv}[1][]{\mathop{\mbox{ind-$d_{MV}^{#1}$}}}
\newcommand{\procwp}[1][2]{\mathop{\mbox{proc-$W_#1$}}}
\newcommand{\procwphat}[1][2]{\mathop{\mbox{proc-$\hat{W}_#1$}}}

\makeatletter
\let\@fnsymbol\@arabic
\makeatother

\newcommand{\anon}{1}

\begin{document}

\begin{bibunit}
\def\spacingset#1{\renewcommand{\baselinestretch}%
{#1}\small\normalsize} \spacingset{1}


\if1\anon
{
    \title{\bf Vertex misalignment and changepoint localization in network time series}
    \author{
    Tianyi Chen$^{a,}$\thanks{Co-first author}\,\,\,\thanks{The authors gratefully acknowledge support from the following sources:  TC from the The Acheson J. Duncan Fund for the Advancement of Research in Statistics;  TC and CEP from the ONR Science of Autonomy award number N00014-24-1-2278; VL from AFOSR Complex Networks award number FA9550-25-1-0128 and The JHU HLT COE; CEP from ONR Applied Physics Lab (APL) Grant 151530; ZL and VL from the Brin MRC at UMD.}, Mohammad Sharifi Kiasari$^{b,}$\samethanks[1], Avanti Athreya$^{a}$, Sijing Yu$^{b}$,\\
    Youngser Park$^{c}$, Carey Priebe$^{a}$, Vince Lyzinski$^{b}$, and Zachary Lubberts$^{d}$ \\
    $^{a}$\small{Department of Applied Mathematics and Statistics, Johns Hopkins University}\\
    $^{b}$\small{Department of Mathematics, University of Maryland College Park}\\
    $^{c}$\small{Center for Imaging Science, Johns Hopkins University}\\
    $^{d}$\small{Department of Statistics, University of Virginia}
    }
    \maketitle
} \fi

\if0\anon
{
  \bigskip
  \bigskip
  \bigskip
  \begin{center}
    {\LARGE\bf Vertex misalignment and changepoint localization in network time series}
\end{center}
  \medskip
} \fi

\bigskip
\begin{abstract}
 Inference for time series of networks often relies on accurate vertex correspondence between network realizations at different times. In practice, such vertex alignments can be misspecified or unknown. 
We study the impact of errorful vertex alignment on changepoint localization for dynamic networks through two illustrative models, with the key distinction being whether changepoint information is contained in marginal or joint distributions of the time-varying latent positions. 
We compare localization techniques ranging from the simple network statistic of average degree to the modern procedure of Euclidean mirrors.  
Our results demonstrate that understanding the interplay of marginal and joint information in the observed network time series is key to robust network inference.
\end{abstract}

\noindent%
\vfill

\newpage
\spacingset{1.8} 

\section{Introduction}

Changepoint localization and detection in network time series is a rapidly-developing subfield of network inference. As in the classical formulation of these problems \citep{shumway2000time,box2015time,kedem2005regression, verzelen2023optimal, padilla2021optimal}, the goal is to identify if and when a time-varying random network undergoes a distributional change. 
Statistical changepoint analysis for networks must contend with the inherently non-Euclidean structure of network data, the dependencies between underlying network features, and \emph{vertex alignment} (the correspondence between vertices over time), but this last issue has largely been overlooked in the literature. In practice, the common assumption of aligned vertices may be incorrect or even inadmissible \cite{chen2023discovering}. As such, robust network inference requires understanding and mitigating vertex misalignment. 

Theoretical facets of the analysis of time-varying networks encompass random matrix theory and eigenvector fluctuations \citep{cape2019signal,agterberg2022entrywise,tang2025eigenvector}, tensor analysis 
\citep{macdonald2022latent,pensky2025signed,chen2022analysis}, minimax bounds and thresholds of detectability \citep{cai2021optimal,zhang2018tensor,lei2024computational}, and optimization and graph matching \citep{fishkind2019seeded,arroyo2021graph,saad2021graph, BergmannHerzogJasa:2024}. On a practical level, the problem of identifying shifts in network dynamics has broad applicability, for example in organizational networks for corporations \citep{zuzul2021dynamic}; neural development for organoids \citep{chen_worm}; and transportation networks \citep{agterberg2022joint}. 

The non-Euclidean nature of network data can be addressed by considering lower-dimensional spectral embeddings of multiple networks or a network adjacency tensor \citep{paul2020spectral,jing2021community,pantazis2022importance,macdonald2022latent,wang2023multilayer,agterberg2022joint,padilla2022change,gallagher2021spectral}, but this requires accurate alignment between vertices across networks.
In \cite{athreya2025euclidean}, the authors develop the methodology of Euclidean mirrors, which involves the construction of a dissimilarity measure between time-indexed pairs of adjacency matrices $\bA_{t}$, $\bA_{t'}$ of a network time series. 
For certain types of dissimilarity measures and under certain model assumptions, a low-dimensional embedding of the matrix of dissimilarities---known as {\em Euclidean mirror} or {\em iso-mirror}---can help identify when the sequence of networks undergoes a significant change \citep{chen2023discovering, chen2024euclidean}. Existing approaches to dynamic networks, including that of the Euclidean mirror, often presume a fixed vertex set and a known correspondence across vertices. In practice, both of these assumptions can fail. 
As networks evolve, vertex sets may themselves change (see, for example, the cases in \cite{kivela2014multilayer} that do not have equal size across layers); and even when vertex sets remain the same, correct alignments may be unknown, errorfully observed, or corrupted \citep{priebe2015statistical, lyzinski2020matchability, arroyo2021maximum}.
For example, in the analysis of neural organoid time series, the vertex set of neural cells expands as the organoid grows, while constant agitation due to in-vitro culturing makes tracking vertex correspondence over time impossible.
In \cite{chen2023discovering}, the authors mitigate the lack of vertex correspondence in organoids by using graph matching on the padded adjacency matrices and applying Euclidean mirrors for exploratory data analysis on organoid time series. 
A natural step toward a deeper understanding of the Euclidean mirrors methodology is to explore the impact of this vertex misalignment and matching on changepoint localization.

There is a large body of work devoted to recovering an unknown vertex correspondence across networks---this is the well-studied graph matching problem \citep{gms1,gms2,gms3}---and the effect of label errors has been studied in the context of graph and vertex classification
\citep{vogelstein2015shuffled,chen_worm}, joint graph clustering \citep{racz2021correlated,racz2,chai2024efficient}, and two-sample graph hypothesis testing \citep{saxena2025lost,qi2025asymptotically}.
The present work is, to our knowledge, the first to systematically consider the effect of label error in time-series analysis and change-point localization.

To formalize the impact of misalignment on graph time-series analysis, we consider a collection of random networks on the same vertex set indexed by time. 
For this network time series, we assume that associated to each vertex $i$ in each network at time $t$ is a (typically unobserved) low-dimensional random vector, called the {\em latent position}, and that the probability of connections between nodes in the network depends on these latent positions. 
A central question is if the associated underlying stochastic process---called here a {\em latent position process} (LPP)---changes in some significant way, can we can detect this by appropriate analysis of the observed network data itself? 
Moreover, how is this localization impacted by vertex misalignment?

To that end, we prove several results on the impact of vertex shuffling on latent position process networks, specifically describing how this weakens the cross-graph dependence across time. 
We prove key results relating graph matching (of the adjacency matrices) and optimal transport (of the estimated point processes) in this setting, and the impact of these two procedures for rectifying the information loss from vertex mislabeling. 
We then  describe two distinct latent position processes, dubbed {\em London} and {\em Atlanta}. In both latent position processes, there is a concrete distributional change at a specific time point.
In each case, the changepoint in the underlying respective LPP gives rise to a corresponding  changepoint in the associated time series of networks. 
In the London model, the vertex labels are effectively uninformative and this changepoint can be localized even when vertex correspondence is lost. 
Moreover, several network statistics or functionals---ranging from simple ones such as average degree to more complicated ones involving Euclidean mirrors---can localize the changepoint. 
In the Atlanta model, by contrast, a shuffled vertex correspondence can severely degrade changepoint localization, sometimes irrecoverably.
Here, the Euclidean mirror methodology can correctly localize the changepoint under a known correspondence, and retains some signal even under partial vertex shuffling. 
However, in the presence of substantial shuffling, we show that these localization approaches will generally fail. 
Further, efforts to reconstruct the latent alignment, such as via graph matching or optimal transport, provide scant improvement.  

While both models are simplifications, they present novel edge cases. 
In London, the vertex labels do not carry information relevant to changepoint localization, while in Atlanta, changepoint recovery cannot proceed without the vertex correspondence. This distinction can be characterized in terms of the marginal and joint distributions across time, for each network time series and its underlying LPP. 
We prove that vertex relabeling can weaken correlation between latent position processes at different times; thus, for cases in which changepoint localization depends on accurately gauging correlation, and hence the joint structure, a loss of vertex correspondence can degrade downstream inference. 
In the London model, changepoint localization can proceed via accurate estimation of marginal properties of the underlying latent position process, and there is no information loss from vertex mislabelling. 
In the Atlanta model, changepoint localization depends on understanding joint distributions, and the correlation-weakening effect of vertex mislabelling ranges from navigable, when a small number of vertices are mislabelled, to insurmountable once a nontrivial fraction of vertices are incorrectly matched.
The importance of vertex correspondence in real data lies, of course, between these two extremes. Understanding the nuances of these models lends insight into representations for real data and changepoint detection methodology in the presence of vertex mislabeling.


\textbf{Notation} For $n,m \in \NN$ with $n<m$, let $[n] : = \{1,2,\cdots,n\}$ and $[n,m]=\{n,n+1,\cdots,m\}$. Let $e_i$ denote the $i$th standard basis vector in $\mathbb{R}^d$, and let $v=\left(v_1, v_2, \cdots, v_d\right)^\top$ denote the column vector $v=\sum_{i=1}^d v_i e_i$.  The vector Euclidean norm is denoted $\|\cdot\|$. For a matrix $\bA \in \R^{k \times l}$, denote its $i, j$th entry by $A_{ij}$; its $i$th row by $A^{i.}$; and its $j$th column by $A^{.j}$. Where there is no ambiguity, we will use $A^i$ or $A^j$ to denote the $i$th row or $j$th column of $\bA$. 
The spectral norm of $\bA$ is denoted by $\|\bA\|_2$, the Frobenius norm by $\|\bA\|_F$, the maximum Euclidean row norm by $\|\bA\|_{2 \to \infty}$, and the maximum absolute row sum by $\|\bA\|_{\infty}$. For a square matrix $\bA \in \R^{n \times n}$, $\lambda_1(\bA), \cdots, \lambda_n(\bA)$ denote the eigenvalues in decreasing order. We denote the set of matrices $\bW\in \RR^{n\times d}$ with orthonormal columns by $\mathcal{O}^{n\times d}$, and the set of permutation matrices $\bP\in \RR^{n\times n}$ by $\mathcal{P}_n$. The set of all permutations of $\{1,2,3,\cdots, n\}$ is denoted by $S_n$.
In our asymptotic analysis, the relevant terminology and definitions for asymptotic order  is found in \cite{chen2024euclidean}.

\section{Vertex misalignment: matching and transport}
\label{sec:Main_results_GM_OT}
In this section, we consider consequences of vertex misalignment on network time series, specifically for the dissimilarity measures considered by Euclidean mirrors, and we explore the relationship between graph matching and optimal transport, which are two approaches for mitigating vertex misalignment.

In what follows, we will model the connection between latent positions and network distributions using the Random Dot Product Graph (RDPG) formalism \citep{young2007random,athreya_survey}.
Namely, given the latent position matrices $(\bX_t)_{t=1}^T$, where each row  $i$ of $\bX_t$ is a $d$-dimensional {\em latent position vector} $X_t^{i}$ associated to a specific node, we model the adjacency matrix $\bA_t$ of our network at time $t$ via $\bA_t\stackrel{ind.}{\sim}$RDPG($\bX_t$). This guarantees that the edges in $\bA_t$ are conditionally independent given $\bX_t$, and $\mathbb{E}(\bA_t|\bX_t)=\bX_t\bX_t^T$. The time-varying collection of networks so generated is called a {\em latent position process time series of graphs} (LPPTSG, sometimes further shortened to TSG, for time series of graphs, when the latent position process is understood from context).
In the RDPG and the LPPTSG, the latent positions at time $t$ can be estimated via the Adjacency Spectral Embedding (ASE), essentially a thin singular value decomposition of the adjacency matrix $\bA_t$. For more on RDPGs and ASE, see \cite{STFP-2011, athreya_survey, rubin-delanchy_tang_priebe_grdpg}.

\subsection{Vertex shuffling, dependence, and dissimilarity}\label{sec:D^3}

In \cite{athreya2025euclidean}, the authors define a distance, the $d_{MV}$ distance, which captures dependence between latent positions across time by computing an expectation with respect to their joint distribution. 
Considering two time points in our LPP $t$ and $t'$, and two associated $d$-dimensional latent position vectors $X_t$ and $X_{t'}$, the $d_{MV}$ distance is defined via
$$d_{MV}(X_t,X_{t'}) := \min_{\bW \in\mathcal{O}^{d\times d}} \left\|\EE[(X_t-\bW X_{t'})(X_t-\bW X_{t'})^\top]\right\|^{1/2},$$
where $\bW$ accounts for the rotational nonidentifiability in the RDPG framework. 
As shown in \cite{athreya2025euclidean}, this distance can be estimated using realized matrices of latent positions via the {\em estimated} $d_{MV}$ distance
$\hat{d}_{MV}(\bX_t,\bX_s):=\min_{\bW \in\mathcal{O}^{d\times d}}\frac{1}{n}\left\|\sum_{i=1}^n(X_t^i-\bW X_{s}^i)(X_t^i-\bW X_{s}^i)^T\right\|^{1/2}$.
The Euclidean mirror, denoted $\psi$, is effectively the low-dimensional Euclidean configuration obtained from applying canonical multidimensional scaling (CMDS) to the distance matrix defined by these $d_{MV}(X_t,X_{t'})$ for all times $t,t'$.
Computing or estimating these distances requires the vertex correspondence to be known across graphs.
When it is not, the measured dependence between the random vectors $X_t,\,X_{t'}$ is weakened.
In the extreme, the dependence is lost and $X_t$ and $X_{t'}$ are independent, which we capture in the definition below.  

\begin{definition}[Independent-$d_{MV}$ Distance]
\label{def:independent_d_MV}
Let $X,Y\in L^2(\Omega)$ be latent position random vectors. The independent-$d_{MV}$ distance is defined as
$$\inddmv(X,Y) := \min_{\bW\in\mathcal{O}^{d\times d}} \left\|\EE_{(X',Y')\sim \mu_X \otimes \mu_Y}[(X'-\bW Y')(X'-\bW Y')^\top]\right\|^{1/2},$$
where $\mu_{X}\otimes \mu_{Y}$ is the product of the marginal measures of $X$ and $Y$.
\end{definition}
We model the impact of vertex shuffling on a network via a uniformly-distributed permutation of the graph vertex labels, by which the matrix $\bA_t$  is mapped to  $\bP\bA_t \bP^\top$ for an appropriate permutation $P$. 
In the RDPG framework, shuffling the graph is equivalent to shuffling the latent positions, and the estimated $d_{MV}$ distance becomes $\hat{d}_{MV}(\bX_t,P\bX_s).
$

The theorem below shows that under sufficient vertex shuffling, the information across graphs is lost in the computation of $d_{MV}$, and each time point in the time series can be thought of as an independent realization from the LPP.
\begin{thm}\label{thm:independentdMV_convergence}
Consider a positive scalar LPP $\{X_t ; t\in [m]\}$. 
Let $\{\bX_{t}\}_{t\in[m]}$ be latent position matrices with $n$ rows drawn i.i.d from this LPP.
Let 
$\bP_n$ be a permutation matrix uniformly selected from the class of $n \times n$ permutation matrices $\mathcal{P}_n$. 
Then for any $\epsilon > 0$,  
$$ \PP \left( \left| \hat{d}^2_{MV}(\mathbf{X}_{t}, \bP_{n} \mathbf{X}_{s})-  \inddmv[2](X_{t}, X_{s}) \right| \geq \epsilon \right) \leq  \frac{12}{n \epsilon^2} + \frac{4}{n \epsilon}.$$
\end{thm}
We can also consider the case where only a portion of the vertices are misaligned, or where only a portion of the vertex alignments are known, as in the following.
\begin{definition}[$\alpha$-shuffling]\label{def:alpha-shuffling}
Let $\alpha\in(0,1]$, and define $\alpha_n=\lfloor \alpha n\rfloor$. Let $\sigma_{\alpha_n}$ be a uniformly random perumtation of $[n-\alpha_n+1,n]$ with corresponding permutation matrix $\bP_{\alpha_n}$. Define the {\em $\alpha$-shuffling} by $\sigma\in S_n$ satisfying $\sigma(i)=i$ for $i=1,\ldots, n-\alpha_n$ and $\sigma(j)=\sigma_{\alpha_n}(j)$ for $j=n-\alpha_n+1,\ldots, n$. The corresponding permutation matrix of such an $\alpha$-shuffling is given by 
$\bP_{n,\alpha}:=\mathbf{I}_{n-\alpha_n}\oplus \bP_{\alpha_n}$. When $\alpha=1$, we call this a \emph{shuffling}.
\end{definition}

\begin{definition}[$\alpha$-shuffled TSG]
\label{def:alpha-shuffled-tsg}
Let $\bA_1,\ldots,\bA_T$ be the adjacency matrices of a TSG on a fixed, aligned vertex set with $n$ vertices. Given $\alpha\in(0,1]$, draw $T$ independent $\alpha$-shufflings as in Definition \ref{def:alpha-shuffling}, with permutation matrices $\{\bP^1_{n, \alpha},...,\bP^T_{n, \alpha}\}$.
We call the TSG with adjacency matrices given by 
$\bP^1_{n, \alpha} \bA_1 (\bP^1_{n, \alpha})^\top ,\ldots,  \bP^T_{n, \alpha} \bA_T (\bP^T_{n, \alpha})^\top$ an $\alpha$\emph{-shuffled TSG}; and when $\alpha = 1$, a \emph{shuffled TSG}.
\end{definition}
It is reasonable to ask 
for positive scalar latent positions, we can characterize the portion of the dependence  preserved by $\hat{d}_{MV}$ when the vertices are $\alpha$-shuffled for $\alpha$<1.  We have
\begin{align}\label{equ:decompose_partial_shuffling}
\hat{d}^2_{MV}(\bX_t,\bP_{n,\alpha}\bX_s)
&= \frac{n - \alpha_n}{n} \hat{d}^2_{MV}(\bX_t^{(1)},\bX^{(1)}_s) + \frac{\alpha_n}{n}\hat{d}^2_{MV}( \bX_t^{(2)} , \bP_{\alpha_n} \bX_s^{(2)} ),
\end{align}
where $\bX_t = [(\bX_t^{(1)})^\top, (\bX_t^{(2)})^\top]^\top$; $\bX_t^{(1)}$ denotes the first $n-\alpha_n$ rows of $\bX_t$ which are aligned with $\bX_s^{(1)}$, and $\bX_t^{(2)}$ and $\bX_s^{(2)}$ contain the remaining $\alpha_n$ rows, which are shuffled. 
\begin{corollary}\label{Cor:partial shuffling}
Let $\alpha\in(0,1]$ and consider a positive scalar LPP. 
For any $\epsilon > 0$, we have 
\begin{align*}
    \mathbb{P} &\left( \left| \hat{d}^2_{MV}(\mathbf{X}_{t}, \bP_{n,\alpha} \mathbf{X}_{s})-  
\left( \left(1-\frac{\alpha_n}{n}\right) d^2_{MV}(X_{t},X_{s}) +  \frac{\alpha_n}{n}\, \inddmv[2](X_{t}, X_{s})\right)
\right| \geq \epsilon \right)\\
&\leq  \frac{1}{(n-\alpha_n) \epsilon^2} + \frac{12}{\alpha_n \epsilon^2} + \frac{4}{\alpha_n \epsilon}.
\end{align*}
\end{corollary}
Between the two poles of perfect vertex alignment and totally random shuffling, it is not uncommon to have a set of seed vertices where the alignment is known or can be correctly estimated. 
In such settings, 
Corollary \ref{Cor:partial shuffling} motivates improving the alignment even when it cannot be perfectly estimated. We further
 define $\alpha$-$d_{MV}$ to be a convex combination of $d_{MV}$ and ind-$d_{MV}$; namely
$
\alpha\text{-}d^2_{MV}(X_t,X_s) := (1-\alpha) ~ d^2_{MV}(X_t,X_s) + \alpha\cdot \inddmv[2](X_t,X_s), 
$ 
with $\alpha\text{-}\hat{d}^2_{MV}$ defined analogously.

\subsection{Graph matching and optimal transport}\label{sec:Mirrors_Matching_Transport}
We focus on two methods for addressing the case of unknown vertex alignments: optimal transport, which aligns the marginal distributions of the latent positions at two different times, and graph matching, which directly aligns the adjacency matrices. 
The objective of optimal transport is to find an alignment that minimizes the Wasserstein-$p$ distance $W_p$ between latent position distributions.
\begin{definition}\label{def:Wasserstein_distance_measures}
Let $\mu, \nu$ be two measures with finite $p$-moments on $\RR^d$. Let $\gamma(\mu, \nu)$ be a \emph{coupling} between these measures, namely a measure on $\RR^d\times \RR^d$ with the product $\sigma$-algebra, with marginals $\mu$ and $\nu$. 
Let $\Gamma(\mu, \nu)$ be the set of all such couplings.  
The Wasserstein $p$-distance between $\mu$ and $\nu$ is defined by
$W_p(\mu, \nu):=\inf_{\gamma \in \Gamma(\mu, \nu)
} \left[\mathbb{E}_{(x,y) \sim \gamma} \|x-y\|_p^p\right]^{1/p}.$ 
When we have random vectors $X,Y$ supported on $\RR^d$, we will also write $W_p(X,Y)$ to denote the Wasserstein $p$-distance between their induced measures, $W_p(\mu_X,\mu_Y)$.
\end{definition}

Because of orthogonal nonidentifiability of the RDPG, we minimize the Wasserstein $p$ distance over orthogonally-transformed latent positions, denoted by $\procwp[p]$. 

\begin{definition}\label{def:procwp-distance}
Consider two random vectors $X \in \RR^d$, $X \sim \mu$ and $Y\in \RR^d$, $Y \sim \nu$. Let $\bW \in \mathcal{O}^{d \times d}$ be given, and let $\nu_{\bW}$ denote the probability measure of $Z=\bW Y$. We define 
\begin{align*}
\procwp[p] (X,Y) : = \min_{\bW \in \mathcal{O}^{d\times d} } W_p(X,\bW Y) &= \min_{\bW\in\mathcal{O}^{d\times d}} \inf_{\gamma\in\Gamma(\mu,\nu_\bW)}[\EE_{(x,z)\sim\gamma}\|x-z
\|_p^p]^{1/p}\\
&= \min_{\bW\in\mathcal{O}^{d\times d}}\inf_{\gamma\in\Gamma(\mu,\nu)}[\EE_{(x,y)\sim\gamma}\|x-\bW y\|_p^p]^{1/p}.
\end{align*}
\end{definition}

When we have two point clouds of vectors in $\RR^d$, $\{x_1,\ldots,x_n\}$ and $\{y_1,\ldots,y_n\}$, we can arrange these as the rows of the matrices $\bX,\bY\in\RR^{n\times d}$. The empirical measure for one of these point clouds is given by $\mu_{\bX}=\frac{1}{n}\sum_{i=1}^n \delta_{x_i},$ where $\delta_x$ is the point mass at $x$. 
The Wasserstein distance between these empirical measures gives us a distance between matrices via 
$\hat{W}_p(\bX,\bY):= W_p(\mu_\bX,\mu_\bY).$ 
In this setting, minimization over couplings is always achieved by a permutation matrix, prompting the following definition.


\begin{definition}\label{def:procwp_matrix_distance}
Given two latent position matrices $\bX,\bY \in \RR^{n \times d}$, we define 
$$
\procwphat[p] (\bX,\bY) : = \min_{\bW \in \mathcal{O}^{d\times d}}\min_{\sigma \in S_n} \left( \frac{1}{n} \sum^n_{i=1}\| (\bX)_i -  (\bY \bW)_{\sigma(i)} \|_p^p \right)^{1/p}.
$$
When $p = 2$, this is simply 
\begin{align*}
\text{proc-}\hat{W}_2^2 (\bX,\bY) &\!=\!  \frac{1}{n} \min_{\bW \in \mathcal{O}^{d\times d}}\min_{\sigma \in S_n} \sum^n_{i=1}\| (\bX)_i \!-\!  (\bY \bW)_{\sigma(i)} \|^2\!=\! \frac{1}{\sqrt{n}} \min_{\bW \in \mathcal{O}^{d\times d}}\min_{\bP \in \mathcal{P}_n} \|\bX-\bP\bY \bW\|_F^2  .
\end{align*}
\end{definition}

 
 We generally do not observe the matrices of \emph{latent} positions, and we can consider $\procwphat$ evaluated at the estimated latent position matrices $\Xhat$ and $\Yhat$. 
 The minimization over permutations in $\procwphat[p](\Xhat, \Yhat)$ amounts to aligning the estimated latent positions directly, and hence can be used to fix any vertex shuffling observed in the process. 
 
 Alternatively, we can determine a vertex relabeling by {\em graph matching} \citep{gms1, foggia2014graph}; given two $n \times n$ adjacency matrices $\bA$ and $\bB$, the goal of graph matching is to find a permutation matrix $\bP_{GM}$ that minimizes 
$\|\bA-\bP\bB\bP^{\top}\|_F$.
Suppose we have adjacency matrices $\bA_s$ and $\bA_t$ from an LPP TSG at times $s$ and $t$, and suppose $\bP_{GM}$ is an optimal alignment as above. 
To understand pairwise dissimilarities for the latent position process while simultaneously mitigating vertex misalignments, we can use graph matching between adjacency matrices to induce an alignment between the estimated latent positions, and then compute the resulting $d_{MV}$ dissimilarity. For a positive scalar LPP, this reduces to computing
$
\frac{1}{\sqrt{n}} \| \hat{\bX}_t - \bP_{GM}\hat{\bX}_s\|_2 .
$
If we minimize this quantity with the estimated latent positions, we recover  
$
\min_{\bP \in \mathcal{P}_n}\frac{1}{\sqrt{n}}\| \hat{\bX}_t - \bP\hat{\bX}_s\|_F=W_2(\mu_{\hat{\bX}_t}, \mu_{\hat{\bX}_s} )$.

Unlike the graph matching setting where the adjacency matrices are observed, we do not observe the latent positions directly. 
The following two theorems show that Wasserstein-type distances between latent position distributions can be well-approximated by Wasserstein-type distances on the estimated latent positions.
\begin{thm}\label{thm:W1_continuous}
Let $\bX,\bY\in \RR^{n\times d},$ and suppose that $\bU,\bV\in \RR^{n\times d}$ satisfy the following: $ \min_{\bW\in \mathcal{O}^{d\times d}}\|\bU-\bX \bW\|_{2\rightarrow\infty}\leq \epsilon,$ and $\min_{\bW\in\mathcal{O}^{d\times d}}\|\bV-\bY \bW\|_{2\rightarrow\infty}\leq \epsilon.$ Then we have that
$ |\procwphat[1](\bU,\bV)-\procwphat[1](\bX,\bY)|\leq 2\epsilon. $
\end{thm}

Note that Theorem~\ref{thm:W1_continuous} is a deterministic statement, and we make no assumption about the joint distribution of the two samples comprising the rows of $\bX$ and $\bY$ or the relationship between these matrices and their estimates $\hat{\bX}$ and $\hat{\bY}$.  The next theorem provides the analogous result for pairs of latent positions (and their estimates) for RDPGs.
\begin{thm}\label{thm:W1_as_converge}
Let $\mu_X,\mu_Y$ be distributions on $B_d^+$, the subset of the unit ball in $\RR^d$ with all nonnegative entries. Let $\bX,\bY\in\RR^{n\times d}$ be matrices whose rows come from iid samples from $\mu_X,\mu_Y$, respectively. Suppose that $\EE_{\mu_X}[xx^\top], \EE_{\mu_Y}[yy^\top]$ are positive definite. There is a constant $C_d$ such that for sufficiently large $n$,
$$\PP[|\procwphat[1](\bX,\bY)-\procwp[1](\mu,\nu)|>C_d n^{-1/(d+2)}] \leq \exp(-n^{-d/(d+2)}).$$
Furthermore, let $\bA_\bX\sim\mathrm{RDPG}(\bX), \bA_\bY\sim\mathrm{RDPG}(\bY)$. Let $\hat{\bX}, \hat{\bY}$ be the ASEs for these $A_X,A_Y$, respectively. For any $a>0$, there is a constant $C_{a,d}>0$ such that for sufficiently large $n$, with probability at least $1-2n^{-a}$,
$|\procwphat[1](\hat{\bX},\hat{\bY})-\procwp[1](\mu,\nu)| \leq C_{a,d} n^{-1/(d+2)}.$
\end{thm}

Theorem \ref{thm:W1_as_converge} establishes consistent estimation of $\procwp[1]$ distances between the time-varying LPP-induced distributions, based on observations of networks at the corresponding times. 
While the rate of convergence appears to scale poorly with the dimension of the latent positions, this can be improved when the distributions admit a density or additional derivatives; see, e.g., \cite{chewi2025statistical}.

\section{A tale of two models: London and Atlanta}\label{sec:T2M}

\begin{figure}[t!]
    \centering
    \includegraphics[width=0.8\textwidth]{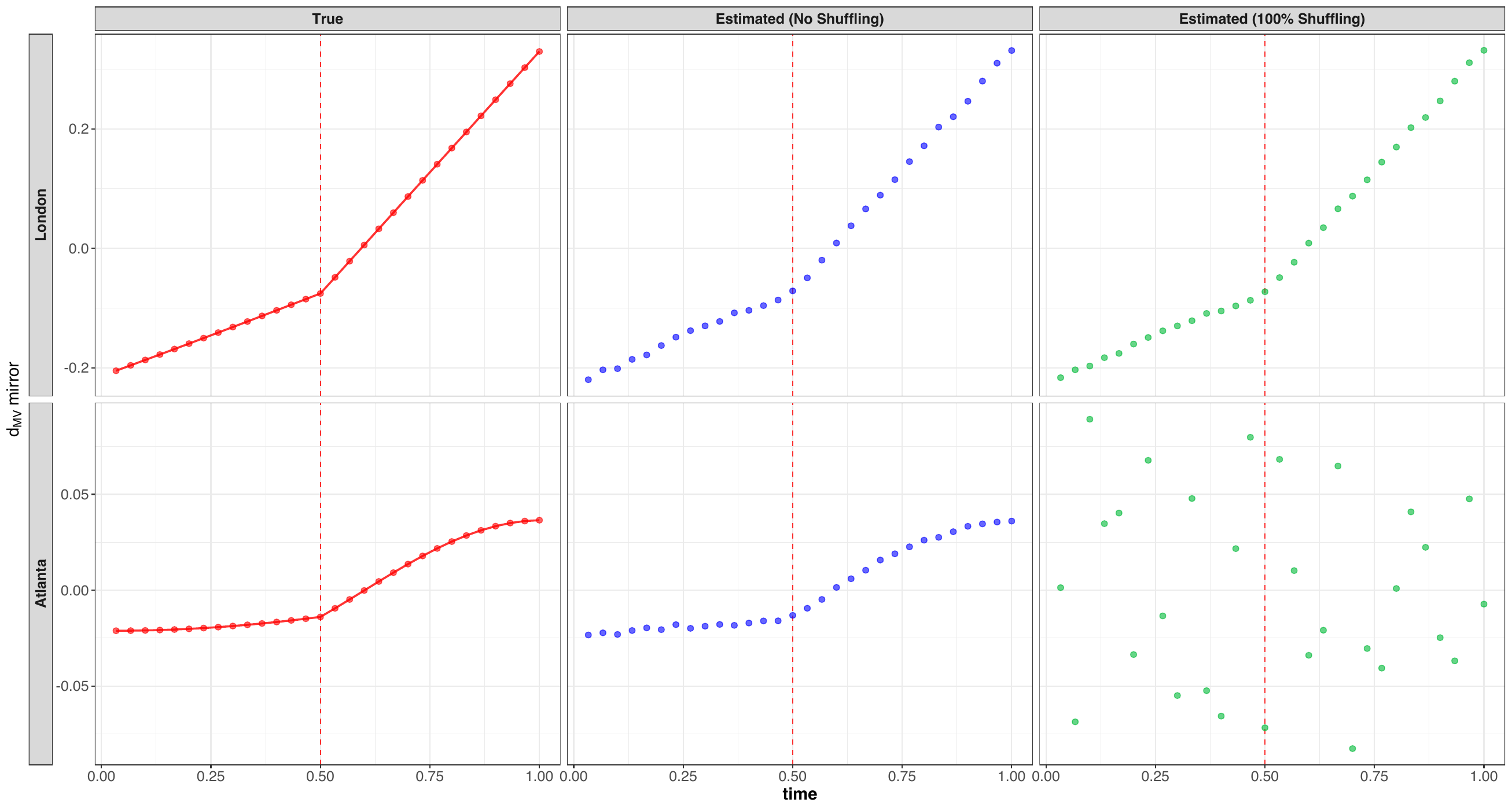}
    \caption{Estimated mirrors for London and Atlanta models after vertex shuffling. Both models exhibit a clear, estimable changepoint in the setting without shuffling. After shuffling the vertices, the London model shows little change in the estimated mirror, while the Atlanta model loses all relevant signal.}
    \label{fig:londonvsatlanta}
\end{figure}

To investigate the effect of vertex misalignment on the recovery of the network dynamics, we introduce two positive scalar LPPs with related but distinct behavior: the \emph{London model}, previously studied in \cite{chen2023discovering}, and the new \emph{Atlanta model}\footnote{These models are so named not just in homage to Charles Dickens, but because of the conference locations in which they were first defined.}. 
Associated to these latent position processes is the corresponding London or Atlanta LPPTSG, whose time-varying adjacency matrices can be observed.

Both LPPs---and hence both 
LPPTSGs---have
a first-order changepoint $t^*$ which we wish to estimate from observations of the network adjacencies. 
To localize the common changepoint $t^*$ in both models, we consider different functions of network observations, ranging from simple statistics such as average degree to more sophisticated localization techniques based on Euclidean mirrors. 
Moreover, in each scenario we consider the effect of vertex shuffling on changepoint localization. Figure \ref{fig:londonvsatlanta} is a pictorial representation of our central point: the London model has a Euclidean mirror whose qualitative features preserve the changepoint even after vertex shuffling. In the Atlanta model, however, signal is lost after vertex shuffling.
Moreover, we also show (see Figure \ref{fig:London-mirrors}) that in the Atlanta model, the Euclidean mirror, in contrast to label-agnostic measures, still captures relevant signal after mild shuffling.

Both models have one-dimensional latent position processes that are discrete-time, finite-state space random walks taking values in the unit interval. 
In the London model, the latent position process begins at a fixed point, and at each time, the latent position $X_t$ can either stay fixed or move one step to the right with probability $p_t$. 
To model the changepoint, $p_t$ is fixed at $p \in (0,1)$ for all times $t<t^{*}$, and thereafter $p_t$ is fixed at a different value $q$. 
The Atlanta model is similar, but the initial distribution of latent positions is uniform over the state space, and at each time there is an equal probability of moving to the state to the right or to the left with probability $p_t$ (only allowing one direction at the boundaries).
One of the key differences between these two models is that the marginal distribution of the latent positions is different from one time to the next in the London model, whereas the marginal distribution of the latent positions is fixed for all times in the Atlanta model. 
While the marginal distributions are informative for the changepoint in the London model (in fact, they are sufficient, as we show in Theorem~\ref{thm:London_MLE}), in the Atlanta model, the marginal distributions contain no information about the changepoint. 
This crucial differences impacts the estimability of the changepoint after vertex shuffling.

\emph{With or without shuffling}, the London model has an approximating piecewise linear Euclidean mirror that exhibits a slope change at $t^*$. 
The Atlanta model has an approximating Euclidean mirror that exhibits a similar piecewise linear structure and slope change at $t^*$ {\em when the vertices are correctly aligned}, but as Figure \ref{fig:londonvsatlanta} demonstrates, this structure and slope change are rapidly corrupted as the fraction of shuffled vertices increases. In both cases, the true 1-dimensional piecewise-linear mirror that we aim to approximate, which captures the changepoint at $t^*$ as well as the different slopes
$p$ and $q$, is given by
$
\psi_Z(t):=pt \mathds{1}\{t<t^*\}+(pt^*+ (t-t^*)q)\mathds{1}\{t\geq t^*\}.
$
Recall that for a distance $d$, $\psi_d$ is the configuration of points obtained by applying CMDS to the matrix of pairwise inter-node distances.

The London model is formally defined as follows.
\begin{model}[London Model LPP and TSG]\label{model:London} Let $m\geq 2$ be an integer. Let $c_L\geq0, \delta_m>0$ be constants satisfying $c+\delta_m m\leq 1$. Let $t^*\in(0,1)$. Denote $t^*_m = \lfloor t^*m\rfloor $, and put $t_i = i/m$ for $i\in\{1,\ldots,m\}$ and define the London model LPP as
\begin{align*}
X_0  = c_L \quad \text{with probability 1}; \quad\quad 
X_{t_i}=\begin{cases}
X_{t_{i-1}}+ \delta_m \quad \text{with probability $p_{t_i}$} \\
X_{t_{i-1}} \quad \text{with probability $1-p_{t_i}$},
\end{cases}
\end{align*}
where $p_{t_i} = p$ for $i < t_m^*$ and $p_{t_i} = q$ for $i \geq t_m^*$.

\end{model}

One of the key properties of the London model is that no information about the changepoint is lost from shuffling. 
The marginal counts for each time are sufficient for the LPP, so that no information about $p,q,$ and $t^*$ is lost when the vertex correspondence is removed. 
This also leads to a maximum likelihood estimator for the changepoint.

\begin{thm}\label{thm:London_MLE}
Let $\{X^i\}$ be an iid sample from the London model.  For each $t\in \{1,\ldots, m\}$,~let 
$
c_t(k):= \frac{1}{n}\#\{i: X^i_t=c+\delta k\},
$
$0\leq k\leq m$. 
Then there exists a function $h$ such that the likelihood $\mathrm{lik}(\{X_t^i\}_{i,t}|p,q,t^*)=h(\{c_t(k)\}_{t,k}|p,q,t^*)$, so the marginal counts $c_t(k)$ are sufficient statistics for the London Model. For each $t\in\{1,\ldots, m\}$, the MLE estimates for $\hat{p}$ and $\hat{q}$ assuming that the changepoint occurs at time $t$ are given by
$$
\hat{p}(t):=\frac{1}{t-1}\sum_{j=0}^{t-1} j \frac{c_{t-1}(j)}{n},\quad \hat{q}(t):= \frac{1}{m-t+1}\sum_{k=0}^m k\frac{c_m(k)-c_{t-1}(k)}{n}.
$$
\end{thm}

We can interpret the estimator $\hat{p}(t)$ as the proportion of all opportunities for the $i$th latent position to increase up to time $t-1$ at which it actually increased, averaged over all $n$ latent positions. The estimator $\hat{q}(t)$ is similar, but only counts increases after time $t-1$. 


The Atlanta model is similar to the London model, but allows for movement of the latent positions in both directions. 
\begin{model}[Atlanta Model]\label{mod:Atlanta} 
Let $N \geq 2$ be a fixed number of states, and let $\delta_N = \frac{c_A}{N-1}$, where $c_A \in (0,1)$ is a constant. Let $m \geq 2$ be the number of time points, and set $t_i = i/m$ for $i \in [m]$. Let $t^* \in (0,1)$, and let $t^*_m = \lfloor t^*m \rfloor$. 
We initialize via $X_{t_0} \sim \mathrm{Uniform}(\{0,\delta_N,\ldots,(N-1)\delta_N=c_A\})$.  In the interior of the state-space, we have
\begin{align*}
 X_{t_{i-1}}\in (0, c_A)\Rightarrow   X_{t_i} =\begin{cases}
X_{t_{i-1}} + \delta_N & \text{w.p. $p_{t_i}$}, \\
X_{t_{i-1}} - \delta_N & \text{w.p. $p_{t_i}$}, \\
X_{t_{i-1}} & \text{w.p. $1 - 2p_{t_i}$}.
\end{cases} 
\end{align*}
where w.p. represents ``with probability.''
At the boundaries, 
\begin{align*}
\text{$X_{t_{i-1}} \!=\! 0$}\Rightarrow X_{t_i} \!=\!\begin{cases}
X_{t_{i-1}} \!+\! \delta_N & \text{w.p. $p_{t_i}$}, \\
X_{t_{i-1}} & \text{w.p. $1 \!-\! p_{t_i}$};
\end{cases}\quad
\text{$X_{t_{i-1}} \!=\! c_A$}\Rightarrow X_{t_i} \!=\!\begin{cases}
X_{t_{i-1}} \!-\! \delta_N & \text{w.p. $p_{t_i}$}, \\
X_{t_{i-1}} & \text{w.p. $1 \!-\! p_{t_i}$},
\end{cases}
\end{align*}
where $p_{t_i} = p$ for $i < t^*_m$, and $p_{t_i} = q$ for $i \geq t^*_m$, with $p \neq q$.
\end{model}

\begin{figure}[t!]
    \centering
        \includegraphics[width=\linewidth]{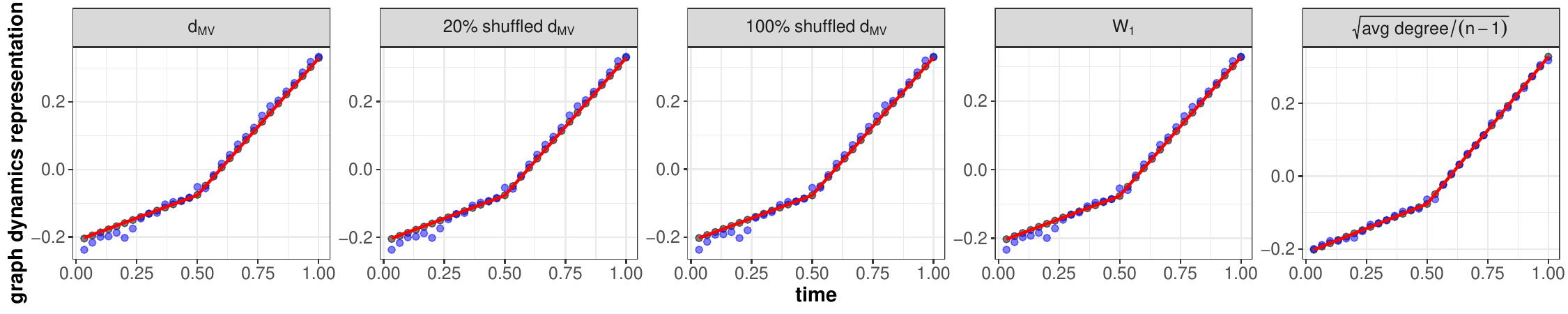}\\
        \vspace{2mm}
        \includegraphics[width=\linewidth]{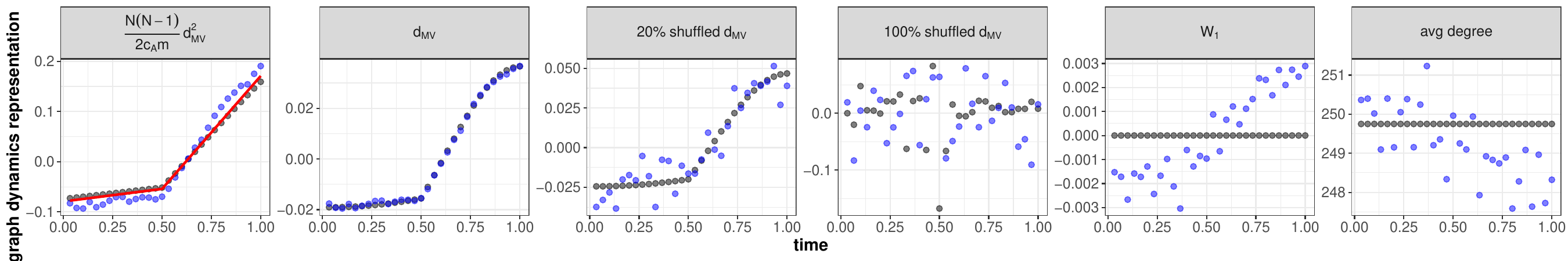}
    \caption{Comparison of the true mirror $\psi_{Z}$ (in red), along with $\psi_{d_{MV}}$, $\psi_{0.2-d_{MV}}$, $\psi_{\text{ind-}d_{MV}}$, $\psi_{W_1}$, and a suitably transformed $\psi_{deg}$ (in black), and their estimates (in blue) based on one realized TSG generated from London (top) and Atlanta (bottom) with
    $n = 100$, 
    $m = 30$ and $t^* = 0.5$.
    In the London model, we have
    $p = 0.3$, $q = 0.9$,  $c_L = 0.1$, $\delta_m = 0.9/30$.  
    In Atlanta, $p = 0.05$, $q = 0.45$, $m = 30$ and $t^* = 0.5$, $c_A = 0.8$, $N = 50$.
    }
    \label{fig:London-mirrors}
\end{figure}

The Atlanta model is parameterized by $N$, $m$, $p$, $q$, $t^*$, and $c_A$. 
It is a Markov process with the uniform distribution as a stationary distribution. 
Since the process starts in its stationarity distribution, the marginal distribution remains unchanged for all times. Namely, $X_{t_i} \overset{\mathcal{\text{dist}}}{=} X_{t_0}$ for any $i \in [m]$.
This already suggests that any method that utilizes only marginal information---namely the $W_1$ mirror, the trivial average-degree-mirror (defined via
$
\hat{\psi}_{\text{deg}}(t_k)= \frac{1}{n} \sum^n_{i =1} \sum^n_{j \neq i} \left(\bA_{t_k}\right)_{i,j}$ for all $ k \in[m] 
$) , and the fully shuffled mirror---will not be effective in localizing $t^*$, as can be seen in Figure \ref{fig:London-mirrors}.
We formalize this below, with a proof in Appendix \ref{sec:AtlPf}).  The analogous theorem for the London model, Theorem \ref{thm:London_main} in the Appendix, shows that in the London model all distances lead to accurate mirror estimation with or without shuffling.

\begin{thm}\label{thm:Atlanta-main}
Consider an Atlanta model LPP for some $N,m,c_A,$ and $t^*=t_m^*/m$. The distances and Euclidean mirrors for this model satisfy

\begin{itemize}
\item[i.] For fixed $m$, there exists $w \in \{\pm1\}$ such that 
$\lim\limits_{N \to \infty}
\max\limits_{i \in [m]}\!\bigg| \frac{N(N-1)}{2c_A^2m}\psi_{d^2_{MV}}(t_i) -  w\psi_Z(t_i) \bigg|\!=\!0.
$

 \item[ii.] If $i\neq j$, then $\text{ind-}d^2_{MV}(X_{t_i}, X_{t_j}) = 
\frac{c_A^2}{6}\frac{N+1}{N-1}$.  If $i=j$, then $\text{ind-}d^2_{MV}(X_{t_i}, X_{t_j})=0$.
The first $m-1$ dimensions of CMDS applied to these distances may be written as
$
\mathbf{\Psi}_{\text{ind}-d_{MV}} = \sqrt{\frac{c_A(N+1)}{6(N-1)}} \bU_{m-1},
$
where $\bU_{m-1}\in \RR^{m\times m-1}$ is any matrix satisfying $\bU^{\top}_{m-1}\bU_{m-1} = I_{m-1}$ and $e^\top \bU_{m-1}=0$, for $e\in\RR^m$ the matrix of all ones.

\item[iii.] Considering $\alpha\in(0,1]$, we have that the first dimension of CMDS applied to the $\alpha$-$d_{MV}$ distance satisfies
$
\psi_{\alpha -d_{MV}} = \sqrt{1-\alpha + \frac{\alpha c_A(N+1)}{12(N-1)\lambda_1}   }      \psi_{d_{MV}},
$
where $\lambda_1$ is the largest eigenvalue of the doubly-centered matrix of squared $d_{MV}$ distances, $-\frac{1}{2}H \mathcal{D}_{d_{MV}}^{(2)}H$, $H=I-ee^\top/m$.

\item[iv.] The $W_1$ distance and expected-average-degree-distance matrices ares exactly 1-Euclidean realizable with
$
\psi_{W_1}(t_i) =\psi_{\text{deg}}(t_i)= 0 \quad\forall\; i \in [m].
$

\end{itemize}
\end{thm}

Across both London and Atlanta models, information about $t^*$ can be obtained from the observed adjacency matrices when the vertex correspondence is known.
That said, the situation is very different across models once the vertex correspondence is lost. The fundamental distinction is that in the London model, the marginal distributions $\EE[X^L_t]$ are sufficient for $t^*$, while in the Atlanta model, the marginal is constant across time, rendering accurate alignment essential for the estimation of $t^*$. Section \ref{sec:Simulations} illustrates this further.


\section{Simulations}\label{sec:Simulations}


We aim to localize $t^*$ in an observed TSG generated from the London or Atlanta model\footnote{All code available at \url{https://github.com/Graph-matched-mirrors/Vertex-alignment-and-changepoint-localization-in-network-time-series}}. To ensure sufficient density, for the London model we choose the starting point as $c_L=0.1$, resulting in $\delta_m=\frac{0.9}{m}$. For the Atlanta model, we set $u_N$ as a uniform distribution supported on discrete set $\{0.1, 0.1+\delta_N,...,0.9\}$, with $\delta_N = \frac{.9-.1}{N-1}$.
Throughout, we set $t^*$ to be $\frac{1}{2}$, select an even number of time points $m$, then fix $p$ and $q$;
in the Atlanta model we have $0 \le p ,q \le 0.5$. 
For both models, the final time series of graphs, with adjacency matrices $\{\mathbf{A}_{t_1},\cdots,\mathbf{A}_{t_m}\}$, is generated by creating a Random Dot Product Graph (RDPG) from the latent positions $\mathbf{X}_t$ at each time point; these $\mathbf{X}_t$ are estimated from the adjacency matrices via the ASE in 1-dimension, which we denote $\hat{\bX}_{t_1},\cdots,\hat{\bX}_{t_m}$.

When the matrices are aligned, we construct the distance matrix 
$
(\hat{\mathcal{D}}_{\hat{d}_{MV}})_{k,s} = \frac{1}{\sqrt{n}}\|\hat{\bX}_{t_k} - \hat{\bX}_{t_s}\|_F$ for each $k,s \in [m],$ and then apply CMDS to this matrix into $c=1$ dimension, obtaining the approximate mirror $\hat{\psi}_{\hat{d}_{MV}}$.
To study the effects of losing vertex alignment, we fix a proportion $\alpha$ of the vertices to be shuffled; then obtain an $\alpha$-shuffled TSG  $\{P_{n,\alpha}^1\bA_{t_1}(P_{n,\alpha}^1)^{\top},\cdots,P_{n,\alpha}^m\bA_{t_m}(P_{n,\alpha}^m)^{\top}\}$ as in Definition~\ref{def:alpha-shuffled-tsg}; and proceed as though the TSG were aligned. This gives a distance matrix
$
\left(\hat{\mathcal{D}}_{\hat{d}_{MV}-\alpha\text{-shuffled-TSG}}\right)_{k,s} = 
\frac{1}{\sqrt{n}} \|P^k_{n, \alpha} \hat{\bX}_{t_k} - P^s_{n, \alpha}\hat{\bX}_{t_s} \|_F$ for each  $k,s \in [m] ,$ with corresponding mirror $\hat{\psi}_{\hat{d}_{MV}\text{-}\alpha\text{-shuffled-TSG}}$.
We also consider applying graph matching and optimal transport to the shuffled TSG before computing $\hat{d}_{MV}$.  

The first of our measures does not make use of the vertex alignment.
We use the $W_1$ distance on the empirical measures formed by ASEs of the TSG with $d_{\text{ASE}}=1$. This gives a distance matrix 
$
\left(\hat{\mathcal{D}}_{W_1}\right)_{k,s} 
 {=}  \sum^n_{i=1}  \bigl|(\hat{\bX}_{t_k})_{(i)} - (\hat{\bX}_{t_s} )_{(i)}\bigr|$ for each $k,s\in[m],
$
where $(\hat{\bX})_{(i)}$ denotes the $i$-th order statistic of the one dimensional vector $\hat{\bX}$; again the estimated mirror $\hat{\psi}_{W_1}$ is obtained via 1-dimensional CMDS applied to this $\hat{\mathcal{D}}_{W_1}$.
As shown in Theorem \ref{thm:Atlanta-main} and the analogous result for the London model in the Appendix, the changepoint $t^*$ in both the London and Atlanta models can be localized as the point of slope change in a piecewise linear mirror. 
We localize $t^*$ by estimating the slope change in a one-dimensional estimate of the mirror.
To that end, we design an $l_2$ localization estimator described in Algorithm \ref{alg:2}, which is revised from the $l_{\infty}$ localizer in \cite{chen2024euclidean}; for a comparison of the $\ell_2$ and $\ell_\infty$ localizers, see Appendix \ref{sec:2inf}.
\begin{algorithm}[t!]
\caption{$l_2$ localization for piecewise linear data}
\label{alg:2}
\begin{algorithmic}[1]
\State Input: $m$ pairs of observations $\{(t_1,y_1),(t_2,y_2),\cdots,(t_m,y_m)\}$ with each $t_i\in \R$,$y_i\in \R$.
\State \textbf{for k from $2$ to $m-1$:}
\State \quad Define $S_k:=\min_{\alpha, \beta_L, \beta_R \in \R} 
\sum^m_{ i = 1 } \left( \alpha + \beta_L(t_i-t_k)+(\beta_R-\beta_L)(t_i-t_k)I_{\{t_i > t_k\}} - y_i \right)^2.$
\State Find the smallest $k_0$ such that $S_{k_0}= \min\{S_2,S_3,\cdots,S_{m-1}\}$ and set $\hat{t}=t_{k_0}$.
\State Output: $\hat{t}$. 
\end{algorithmic}
\end{algorithm}
Note that the minimization problem defining $S_k$ in Algorithm \ref{alg:2} can be solved analytically by ordinary least squares for each fixed value of $k$.
For each Monte Carlo realization, we obtain an estimated mirror, then apply this $l_2$ localization method to get an estimated changepoint $\hat{t}_{mc}$. We compute the squared error $(\hat{t}_{mc} - t^*)^2$, estimating the mean square error (MSE) by averaging these over over $nmc$ Monte Carlo simulations, recording the MSE, as well as the sample standard deviation over the realizations: 
$$
\text{MSE} := \frac{1}{nmc} \sum_{mc=1}^{nmc} \left(\hat{t}_{mc} - t^*\right)^2, \quad 
\text{std} := \sqrt{\frac{1}{nmc-1} \sum_{mc=1}^{nmc} \left( \left(\hat{t}_{mc} - t^*\right)^2 - \text{MSE} \right)^2}.
$$
Because $\hat{t}_{mc}$ are i.i.d with expectation $\EE[(\hat{t} - t^*)^2]$, we can construct an appropriate confidence interval using the usual normal approximation.
Note that $\hat{t}_{mc} \in \{\frac{2}{m},\frac{3}{m}, \cdots,\frac{m-1}{m}\}$, though in Algorithm \ref{alg:2} that the estimated changepoint cannot be a boundary point. As such, an estimator $u_m$ that randomly guesses among $\{\frac{2}{m},\frac{3}{m}, \cdots,\frac{m-1}{m}\}$ should have a chance MSE given by $\EE[ (u_m - \frac{1}{2})^2 ]$, which will be indicated as a horizontal dashed red line in all figures. 

\subsection{MSE with different shuffling ratio}\label{Simulation_MSE}

First we study how the proportion of shuffled vertices $\alpha$ and signal strength $|p-q|$ affect the changepoint estimation in both models. 
We apply the $\hat{d}_{MV}$ mirror for the London model and the $\hat{d}^2_{MV}$ mirror for the Atlanta model to the $\alpha$-shuffled TSGs, treating them as if they were aligned. We consider 100 Monte Carlo runs ($nmc=100$) for each pair of parameters, and vary $\alpha$ and $|p-q|$ in each model as shown in Figure~\ref{fig:MSEvsShuffleRatio}. 
For the Atlanta model in the first panel, MSE increases dramatically with vertex shuffling. For $p=0.4$, $q = 0.1$ or $0.2$, when $\alpha < 15\%$, the MSE for either model is better than chance, indicating that the $d^2_{MV}$ mirror is robust to a small amount of shuffling. 
This indicates that there is still a fair amount of signal in the aligned vertices, allowing for recovery in the case of the Atlanta model. 
When this alignment is not used, which is the case when we use $W_1$ or average degree, the performance is significantly worse than chance.
However, as soon as $\alpha\geq 15\%$, the signal disappears in the Atlanta model, and all methods perform worse than chance. 
The London model in Figure \ref{fig:MSEvsShuffleRatio} is robust to misalignment and the MSE is better than chance (given sufficient signal in $|p-q|$) across proportions of shuffled vertices and values of $q\neq p$.

\begin{figure}[t!]
        \includegraphics[width=\linewidth]{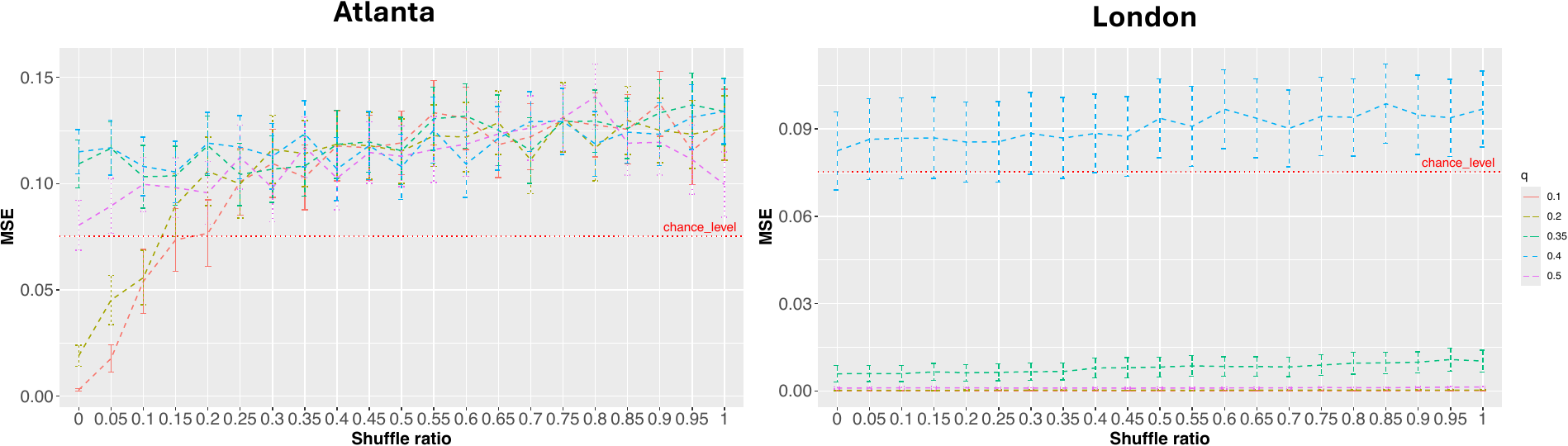}
        \caption{
    MSE versus shuffle ratio \(\alpha\) for two models: (L) Atlanta and (R) London.
    For both models, we fix \(p = 0.4\) with \(n = 300\), \(m = 40\), and \(t^* = m/2\); in London, $c_L = 0.1, \delta_L = 0.9/40$, in Atlanta, we set \(N = 50\) and $\delta_A = 0.8$; 
    \(\alpha\) ranges from 0 to 1.
    Colored curves correspond to \(q = 0.1, 0.2, 0.35, 0.4, 0.5\). Results are averaged over $ nmc = 100$.
    The red dotted horizontal line marks the chance level of $0.075$ for $m = 40$.
    }
    \label{fig:MSEvsShuffleRatio}
\end{figure}

\subsection{Simulated swarm data}\label{sec:Simulations_Swarm}

We next consider a time series of graphs derived from the behavior of interacting agents in a swarm, simulated in \citep{hindes2024swarming}. 
This data is a 1000-seconds-long video consisting of 10,001 sequential frames (1 frame per 0.1 second), representing the movements of 100 agents in a 2D plane with known agent correspondences across all frames.
For simplicity, we focus on a specific subinterval of time, spanning from 780--820 seconds, resulting in \( m = 401 \) frames with \( t_1 = 780 \, s \) and \( t_{401} = 820 \, s \).
The video for this time period can be found at \url{https://www.cis.jhu.edu/~parky//SofA/NRL-swarm-movie.mp4}. 

Our goal is to localize temporal changes in swarm dynamics. 
Upon visually reviewing the video, we identify two distinct structural changes. At 800 seconds, the agents form two major groups that gradually merge back into a single larger group. At around 810 seconds, they begin to split again into two unequal groups.
To see if our methods appropriately capture these changes at times $t=800$ and $t=810$ seconds, we convert the video into a time series of graphs, as outlined in Appendix \ref{sec:swarm}, and compute the \( d_{MV} \)-based mirror method as in Algorithm \ref{alg:iso-mirror}.
\begin{algorithm}[t!]
\caption{Iso-mirror estimation}
\label{alg:iso-mirror}
\begin{algorithmic}[1]
\State Input: TSG $\{\bA_1,\bA_2,\cdots,\bA_m\}$, ASE dimension $d_{\text{ASE}}$, CMDS dimension $d$.
\State Compute $\hat{X}_t=$ASE$(A_t)\in \RR^{n\times d_{\text{ASE}}}, 1\leq t\leq m$.
\State {Construct the distance matrix $\hat{\mathcal{D}}\in\R^{m\times m}$ where $\hat{\mathcal{D}}_{t,s}=\frac{1}{\sqrt{n}}\|\hat{\bX}_{t}-\hat{\bX}_{s}W_F\|_2$,} where $W_F=\arg \min_{W\in\mathcal{O}^{d_{\text{ASE}}\times d_{\text{ASE}}  }} \|\hat{\bX}_{t}-\hat{\bX}_{s}W\|_F$.
\State Compute CMDS$(\hat{\mathcal{D}})=\{\hat{\psi}(1),\hat{\psi}(2),\cdots,\hat{\psi}(m)\}$, with each $\hat{\psi}(t) \in \R^d$. 
\State Apply Isomap on $\{\hat{\psi}(1),\hat{\psi}(2),\cdots,\hat{\psi}(m)\}$ using $k$ nearest neighbors, where $k$ is the smallest value such that the $k$-nearest-neighbor graph is connected, obtaining $\{\hat{\psi}_{d\to1}(1),\hat{\psi}_{d\to1}(2),\cdots,\hat{\psi}_{d\to1}(m)\}$, with each $\hat{\psi}_{d\to 1}(t)\in \RR$.
\State Output: $\{\hat{\psi}_{d\to1}(1),\hat{\psi}_{d\to1}(2),\cdots,\hat{\psi}_{d\to1}(m)\}$.
\end{algorithmic}
\end{algorithm}

To investigate the impact of not knowing the true alignment, we randomly shuffle all of the vertices across all time points and then apply Algorithm \ref{alg:iso-mirror} to the shuffled TSG with $d_{ASE} =2$ and $d =2$.
We also align all networks to $\bA_1$ via graph matching,
followed by the iso-mirror algorithm applied to the graph-matched TSG with $d_{ASE} =2$ and $d =2$.
Lastly, we consider metrics that do not make use of alignment information, including the Wasserstein  distances $W_1$ and $W_2$, and the average degree distance.
To allow for the simultaneous detection and localization of first-order changepoints on the final 1-dimensional iso-mirrors above, we utilize the function ``selgmented'' in R package Segmented \citep{muggeo2017interval,segmented} to perform segmentation regression.

These results are shown in Figure \ref{fig:NRL}. We see that the iso-mirror under the true alignment and matched alignments effectively captures the dynamics of the swarm movement. Segment regression detects four slope changes: at times $787, 798,801, 811$, with 3 occurring around time points $800$ and 810—closely aligning with the changepoints identified by visual inspection. 
In contrast, utilizing a noisy alignment or ignoring the alignment fails to capture the two observed changepoints.
Unlike the London model, this example illustrates how shuffling can dramatically impact performance.
Unlike the Atlanta model, the loss in shuffling for the swarm data is recoverable, at least in part, via matching. While real data seldom aligns with all the requirements of any formal model, the swarm data's mix of characteristics observed in both London and Atlanta is evidence of the utility of these two models as theoretical extremes.

\begin{figure}[t!]
    \centering
\includegraphics[width=0.9\linewidth]{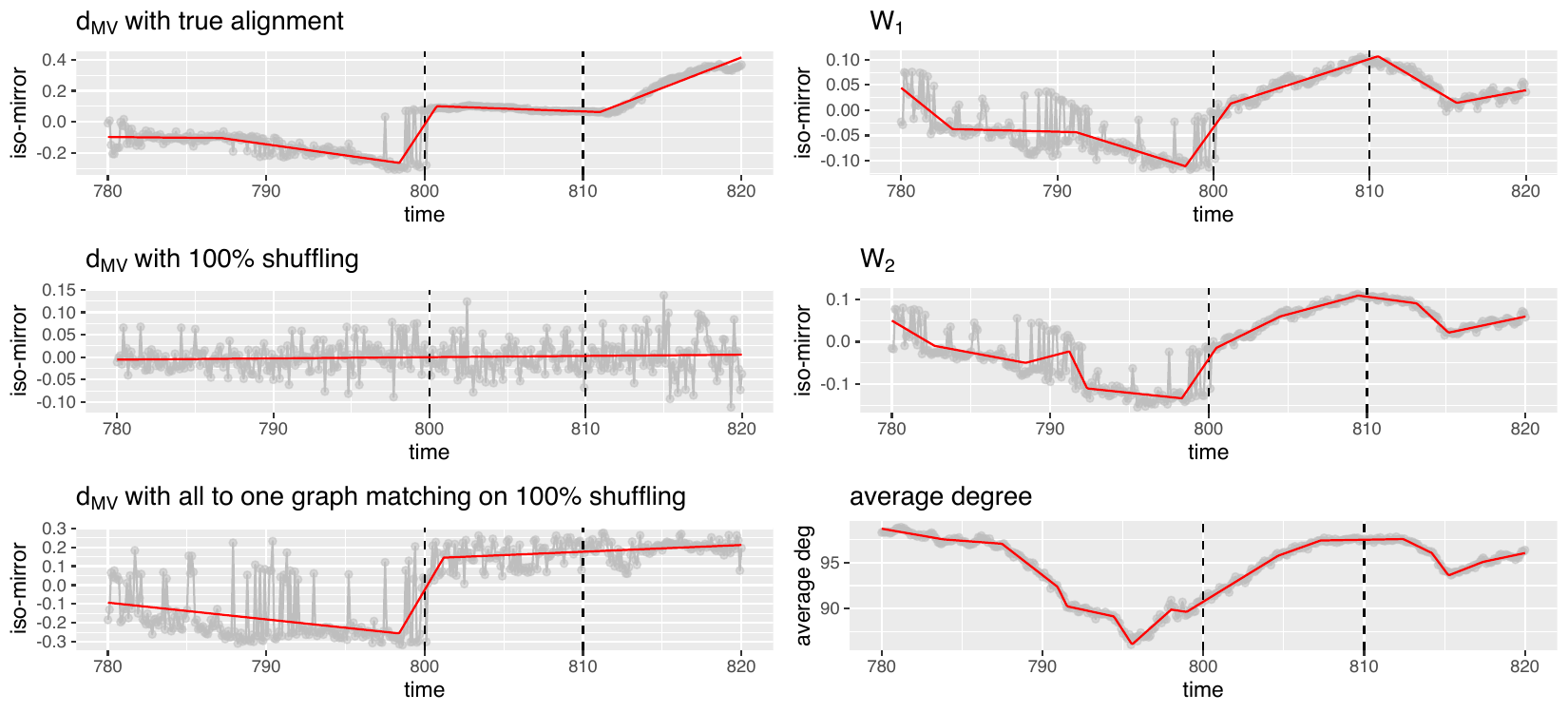}
    \caption{Gray dots represent the iso-mirror embeddings for the time series of graphs derived from simulated swarm data between 780 seconds and 820 seconds. Black dashed lines indicate two structural changes detected by eye at 800s and 810s. The red lines are the segmentation regression fits for the estimated iso-mirrors. Each labeled panel represents a different mirror estimate.
    }
    \label{fig:NRL}
\end{figure}
\section{Conclusion and discussion}\label{sec:Conclusion}

Our core takeaway is that for changepoint localization in networks, if the signal of interest can be found in the marginal distributions, as in the London model, Wasserstein-based metrics can be useful and true vertex alignment may not be essential.  When the signal lies in the joint distribution, as in the Atlanta model, alignment can be crucial for accurate changepoint localization. 
The standard $d_{MV}$ mirror is a robust option if a known and minimally errorful alignment is available.
However, when the vertex alignment is completely unknown, practitioners should be careful about using $d_{MV}$ on graph-matched data, since this approach may only capture marginal information (akin to a $W_1$ mirror), even as it appears to reveal information about the joint distribution. 
Because, in practice, the interplay between joint and marginal distributions and the impact of vertex misalignment can be difficult to determine, we recommend exploring different dissimilarities and shuffling a small portion of the vertices to see how sensitive the results are to an errorful vertex alignment. 
Our results further motivate the consideration of alternative dissimilarities and approaches to mitigating misalignment, and we leave for future work the key question of choosing a dissimilarity in a principled way to best extract information about the dynamics within a class of network time series.

\subsection*{Disclosure statement}\label{disclosure-statement}

The authors have no conflicts of interest to declare.

\subsection*{Data Availability Statement}\label{data-availability-statement}

All of our codes are available at \url{https://github.com/Graph-matched-mirrors/Vertex-alignment-and-changepoint-localization-in-network-time-series}. The video from Section~\ref{sec:Simulations_Swarm} can be found at \url{https://www.cis.jhu.edu/~parky//SofA/NRL-swarm-movie.mp4}.

\subsection*{Declaration of Generative AI Use}

No Generative AI tools were used in the preparation of this manuscript. 

\phantomsection\label{supplementary-material}
\bigskip

\begin{center}

{\large\bf SUPPLEMENTARY MATERIAL}

\end{center}

\begin{description}
\item[Additional simulations:]
Additional simulation results. (.pdf)
\item[Proofs:]
Full proofs of all results. (.pdf)
\end{description}



\end{bibunit}

\begin{bibunit}

    \appendix

\section{Euclidean mirror for the London Model}

\begin{figure}[htp]
  \centering
  \includegraphics[height=4.5cm]{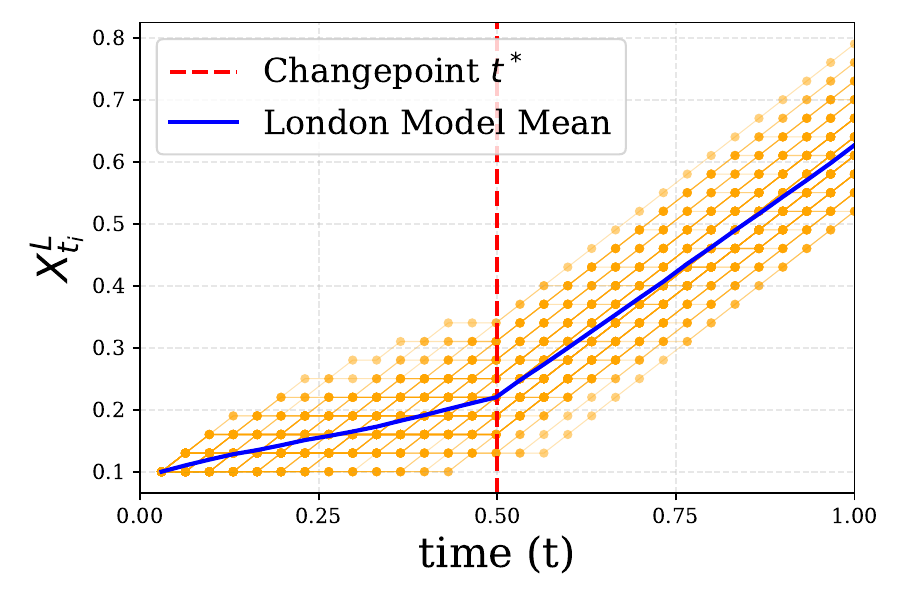}%
  \hspace{-7mm} 
  \caption{Sample trajectories from the London model with $p = 0.3$, $q = 0.9$, $m = 30$ and $t^* = 0.5$, $c_L = 0.1$, $\delta_m = 0.9/30$. Each state is represented by a dot and states are connected by orange trajectories. After the changepoint at $t^{*}=0.5$, the density of connections in the adjacency matrices increase more rapidly, indicating a structural shift both in sample paths and network connectivity.}
  \label{fig:L-illustration}
\end{figure}

\begin{figure}[htp]
  \centering
  \includegraphics[height=4.5cm]{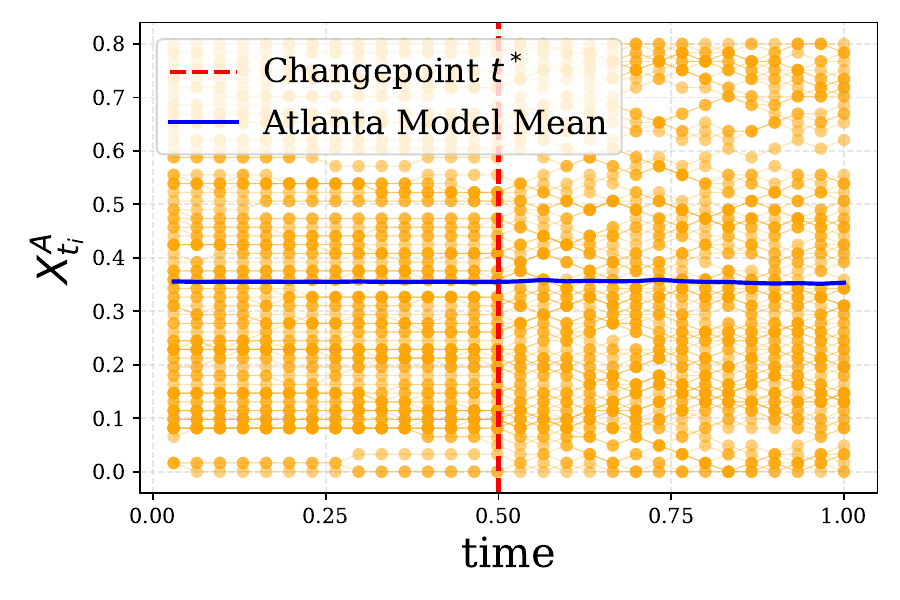}
  \hspace{-7mm}
  \caption{Sample trajectories for the Atlanta model with $p = 0.05$, $q = 0.45$, $m = 30$ and $t^* = 0.5$, $c_A = 0.8$, $N = 50$, $\delta_N=\frac{0.8}{49}$. Each state is represented by a dot and states are connected by orange trajectories. After the changepoint at $t^{*}=0.5$, the sample paths show increased oscillation because the jump probability increases. However, the density of connections in the corresponding adjacency matrices remains stable over time.}
 \label{fig:A-illustration}
\end{figure}

In this section, we provide the theoretical Euclidean mirrors for the London model, analogous to Theorem~\ref{thm:Atlanta-main}. The proof for this Theorem is in Section~\ref{sec:proofLondon}. Recall that the true 1-dimensional piecewise-linear mirror that we aim to approximate, which captures the changepoint at $t^*$ as well as the different slopes
$p$ and $q$, is given by 
$
\psi_Z(t):=pt \mathds{1}\{t<t^*\}+(pt^*+ (t-t^*)q)\mathds{1}\{t\geq t^*\}.
$  We show a collection of sample paths from both the London and Atlanta models in Figures~\ref{fig:L-illustration} and \ref{fig:A-illustration}.

\begin{thm} \label{thm:London_main}
Consider a London model LPP with $c_L=0, \delta_m=1/m$, and $t^*=t_m^*/m$. Let $\psi_{d}(t)$ denote the zero-skeleton mirror using the distance $d$, and let $\tilde{\psi}_d(t)$ be its linear interpolation between consecutive points. Then the Euclidean mirrors for this model satisfy:
\begin{itemize}

\item For $\alpha-d_{MV}$ distance with $0\le\alpha\leq 1$: As $m \to \infty$, the mirror is asymptotically Euclidean 1-realizable with mirror $\psi_Z$, namely, there is a $w\in\{\pm1\}$ such that:
$$
\sup_{t \in [0,1] } \bigg|\tilde{\psi}_{\alpha-d_{MV}}(t)-w\psi_Z(t)\bigg| \to 0;
$$

\item For $W_1$ distance: The $W_1$ distance is exactly Euclidean-1 realizable, and
$$
\psi_{W_1}(t_i) =\EE[X_{t_i}] =  \psi_Z(t_i)\quad  \forall\; i \in [m]; 
$$  

\item For expected average degree distance:
$$
\sqrt{ \frac {\psi_{\text{deg} }\left(t_i\right) }{n-1}} = \psi_Z(t_i)\quad \forall\;i\in[m].
$$

\end{itemize}
\end{thm}

With or without vertex shuffling, the Euclidean mirror for the London model is either equal to or approximates the function $\psi_Z$, meaning that the changepoint can be localized by estimating the location of the change of slope, for example by applying Algorithm~\ref{alg:2}.

\section{Additional simulation results}

In this section, we consider the details of different choices made in changepoint localization. 

\subsection{Changing CMDS dimension}

While Theorem~\ref{thm:Atlanta-main} tells us that as $m\rightarrow\infty$ only the first dimension matters, for a finite $m$, there may still be signal in additional CMDS dimensions, but there will also be additional noise in these dimensions. 
In Figures~\ref{fig:atlanta-iso-mirror_0.2} and \ref{fig:atlanta-iso-mirror_0.05}, we compare the iso-mirrors with and without shuffling for $d_{MV}$ with various choices of CMDS embedding dimension before applying ISOMAP. 
Clearly, a small number of additional CMDS dimensions beyond the first can actually improve the piecewise-linearity of the $d_{MV}$ mirror, but when $d$ is too high, the mirror becomes highly variable and exhibits false changepoints. 
Higher shuffling percentage adds more variance, as can be seen from the scattering of the estimated mirror from the shuffled networks (blue dots) compared to the unshuffled mirror (black dots) in the two figures.

\begin{figure}
    \centering
    \includegraphics[width=0.8\linewidth]{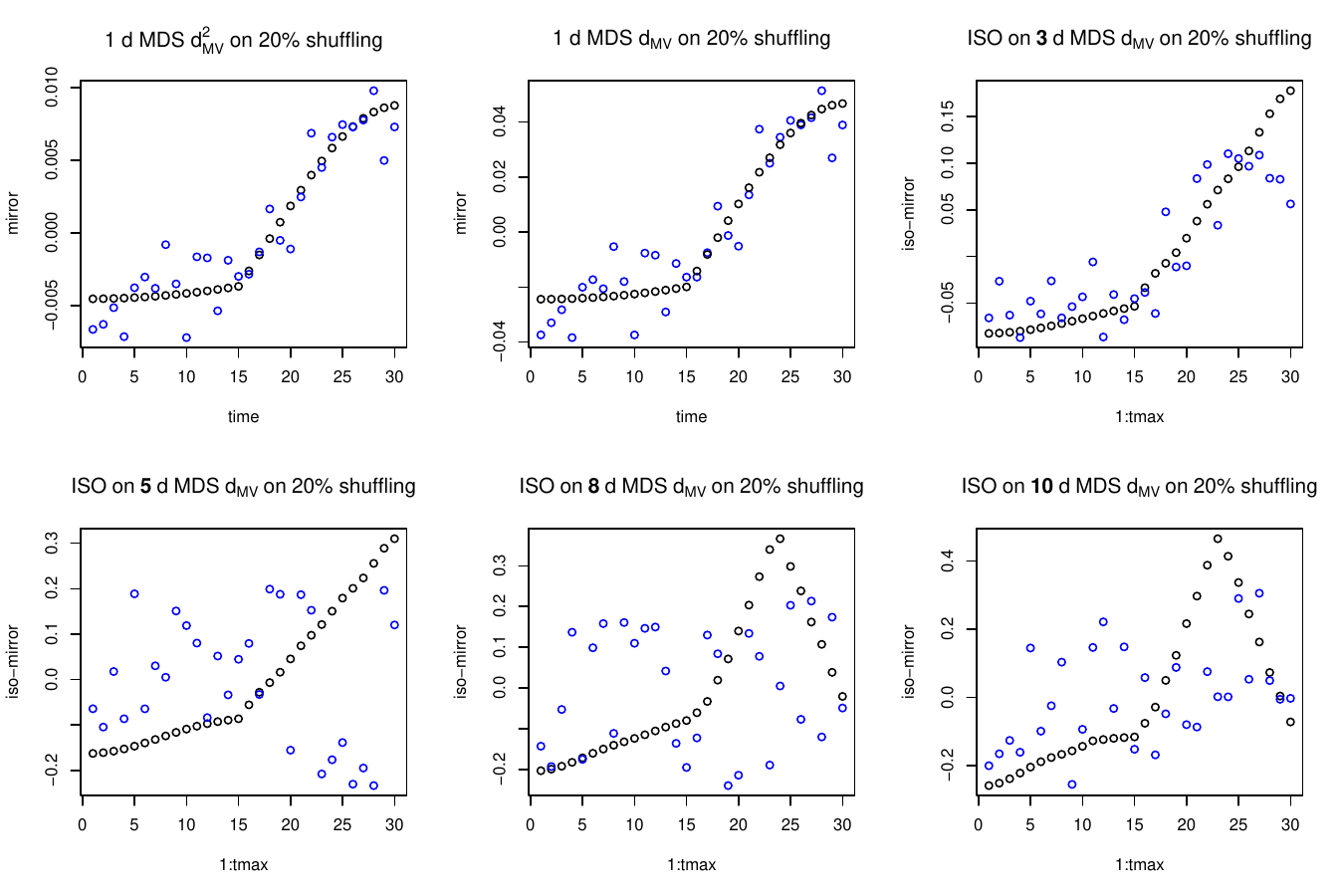}
    \caption{The mirror and iso-mirror(with CMDS embedding dimension $d$ being 3 ,5, 8, and 10) with different CMDS embedding dimension
    based on the same realized 20\% shuffled-TSG from Atlanta model in Figure \ref{fig:London-mirrors} with $n = 1000$, $p = 0.05$, $q = 0.45$, $m = 30$ and $t^* = 0.5$, $c_A = 0.8$, $N = 50$.
    }
    \label{fig:atlanta-iso-mirror_0.2}
\end{figure}

\begin{figure}
    \centering
    \includegraphics[width=0.8\linewidth]{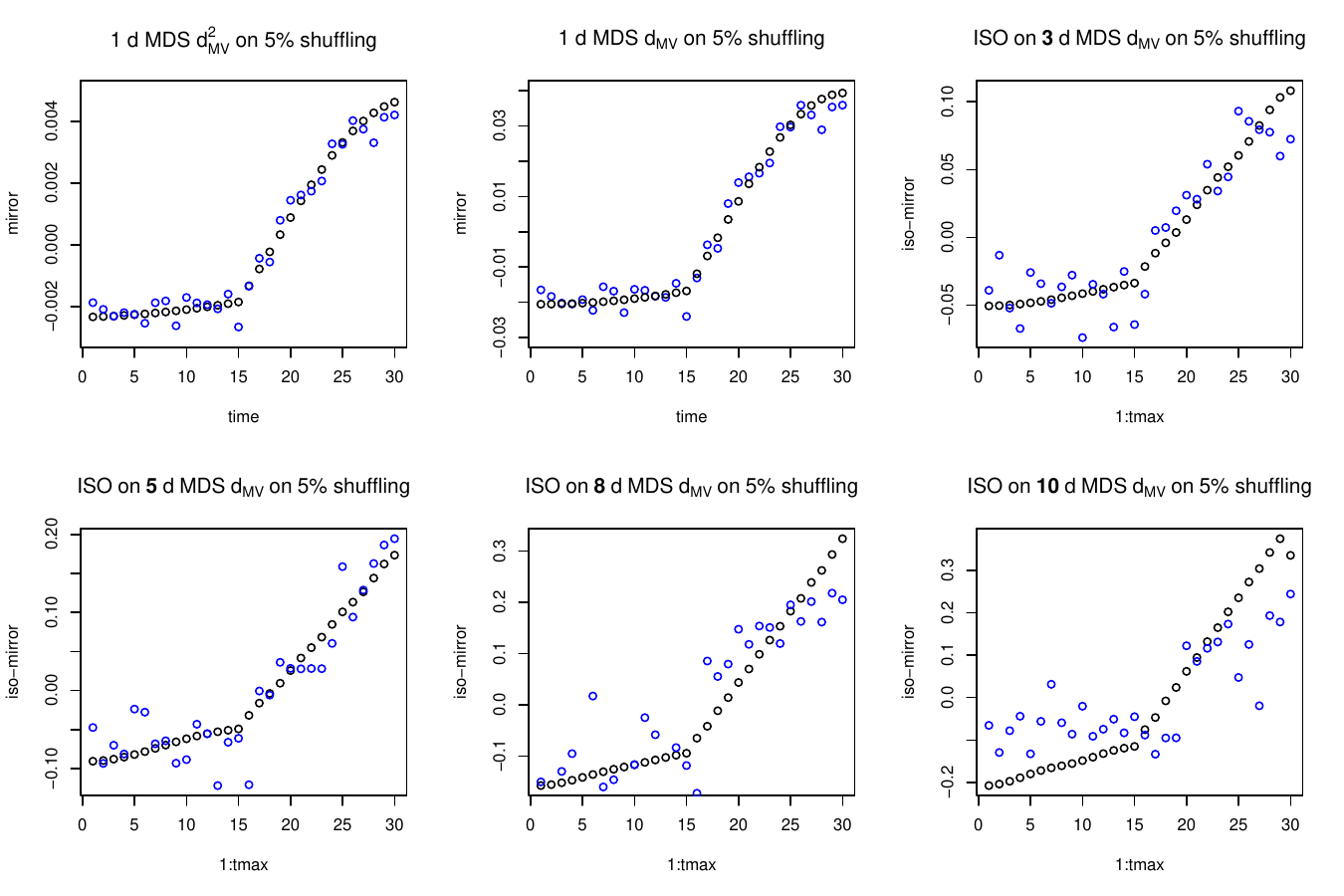}
    \caption{The mirror and iso-mirror with different CMDS embedding dimension based on the same realized 5\%-shuffled-TSG from Atlanta model in Figure \ref{fig:London-mirrors} with $n = 1000$, $p = 0.05$, $q = 0.45$, $m = 30$ and $t^* = 0.5$, $c_A = 0.8$, $N = 50$. }
    \label{fig:atlanta-iso-mirror_0.05}
\end{figure}

We note that from  Theorem~\ref{thm:Atlanta-main}, for the Atlanta model it is the distance $d_{MV}^2$ that yields approximate piecewise linearity of the Euclidean mirror, rather than $d_{MV}$, but while $\psi^A_{\alpha-d_{MV}}$ is a simple scaling of $ \psi^A_{d_{MV}}$, this is not the case for $\psi^A_{\alpha-d^2_{MV}}$. The distance matrix for $\alpha-d_{MV}^2$ can be written as
$$
\mathcal{D^A}_{\alpha-d^2_{MV}}
=\mathcal{D^A}_{\alpha-d_{MV}}^{(2)} 
=   \alpha \mathcal{D^A}_{\text{ind-}d_{MV}}^{(2)} +(1-\alpha) \mathcal{D^A}_{d_{MV}}^{(2)},
$$
so CMDS on $\mathcal{D^A}_{\alpha-d^2_{MV}}$  amounts to considering 
\begin{align*}
\mathcal{D^A}^{(2)}_{\alpha-d^2_{MV}} 
&=  
\left(\alpha \mathcal{D^A}_{\text{ind-}d_{MV}}^{(2)} +(1-\alpha) \mathcal{D^A}_{d_{MV}}^{(2)}
\right)^{(2)}\\
&= \alpha^2 \mathcal{D^A}_{\text{ind-}d^2_{MV}}^{(2)} + (1-\alpha)^2 \mathcal{D^A}_{d^2_{MV}}^{(2)} + 2\alpha(1-\alpha) \mathcal{D^A}_{\text{ind-}d_{MV}}^{(2)} \odot \mathcal{D^A}_{d_{MV}}^{(2)}, 
\end{align*}
where $\odot$ denote the entrywise product of two matrices. 
Because of the third term $\mathcal{D^A}_{\text{ind-}d_{MV}}^{(2)} \odot \mathcal{D^A}_{d_{MV}}^{(2)}$, $\psi^A_{\alpha-d^2_{MV}}$ is no longer a scaling of $\psi^A_{d^2_{MV}}$, $\psi^A_{d_{MV}}$ or $\psi^A_{\alpha-d_{MV}}$. 
However, as shown in Figure \ref{fig:atlanta-iso-mirror_0.2}, they are similar to each other.

\subsection{Changepoint localization method}
\label{sec:2inf}

We investigate the choice of changepoint localization estimator, comparing the $\ell_2$ localizer of Algorithm~\ref{alg:2} to the $\ell_{\infty}$ localizer from \cite{chen2024euclidean} for some different choices of metric, and under different shuffling ratios. In Figure~\ref{fig:atlanta_mse_vs_shuffle_appendix}, we consider this for the Atlanta model, and Figure~\ref{fig:london_mse_vs_shuffle_appendix} considers the London model. The differing impact of shuffling on the two models is immediately apparent, since no method is able to consistently localize the changepoint better than chance in the Atlanta model once there is a significant amount of vertex misalignment, but either method performs well for the London model as long as a changepoint is available to be detected. As a function of the shuffling ratio, there is not much of a difference apparent between these localizers in either model, however.

\begin{figure}[htbp]
  \centering

  \begin{subfigure}[b]{0.48\textwidth}
    \includegraphics[width=\textwidth]{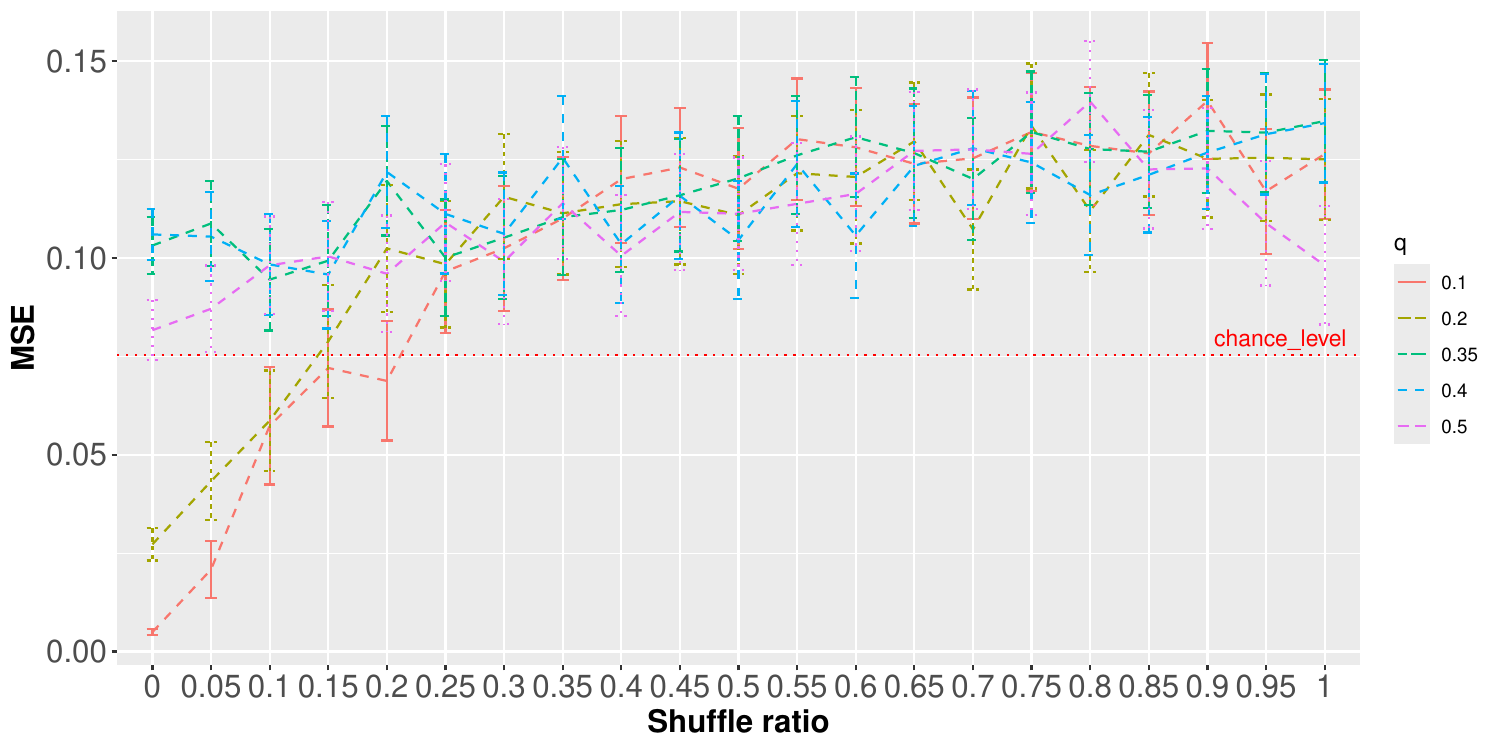}
    \caption{$l_2$ localizer on $d_{MV}$ mirror}
  \end{subfigure}
  \hfill\begin{subfigure}[b]{0.48\textwidth}
    \includegraphics[width=\textwidth]{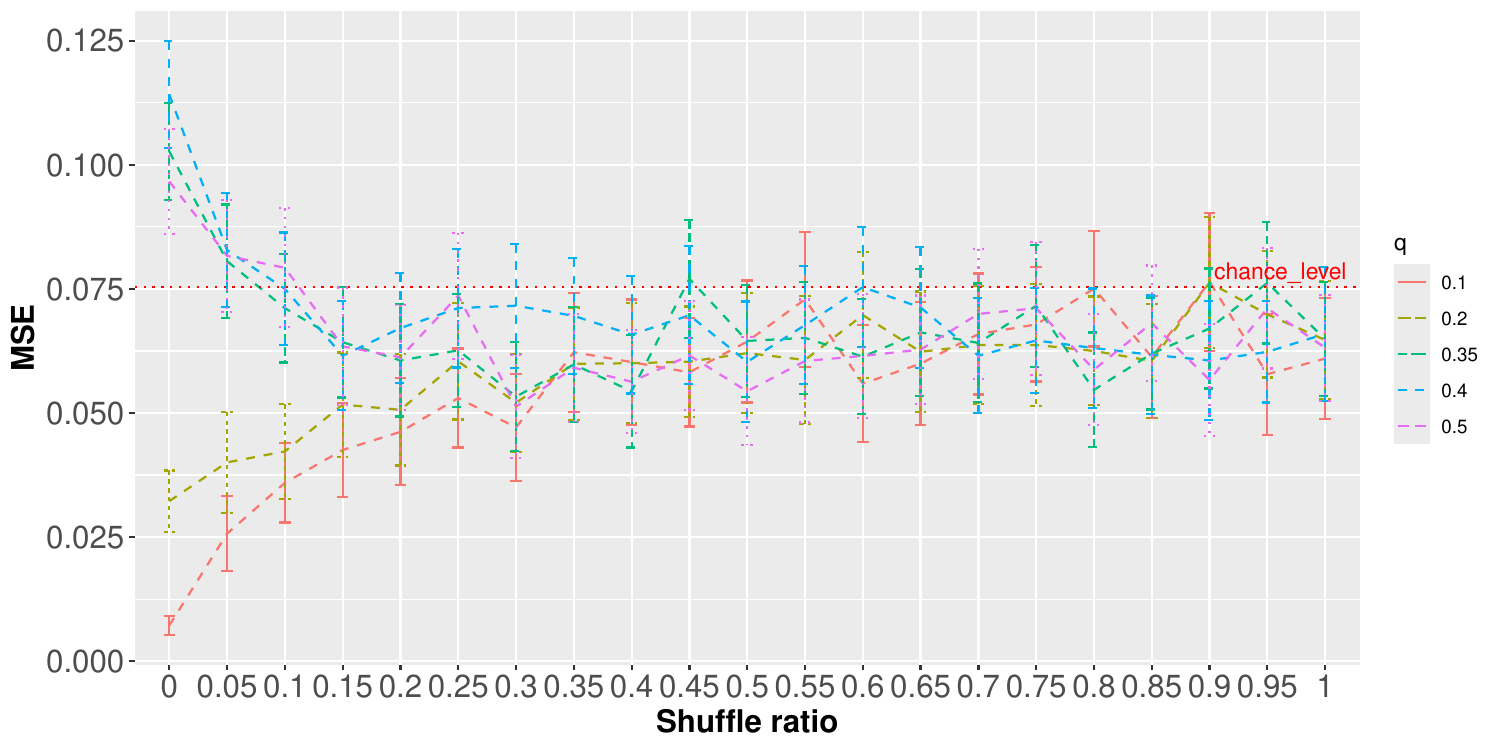}
    \caption{$l_\infty$ localizer on $d_{MV}$ mirror}
  \end{subfigure}
  
  \vspace{1em}

  \begin{subfigure}[b]{0.48\textwidth}
    \includegraphics[width=\textwidth]{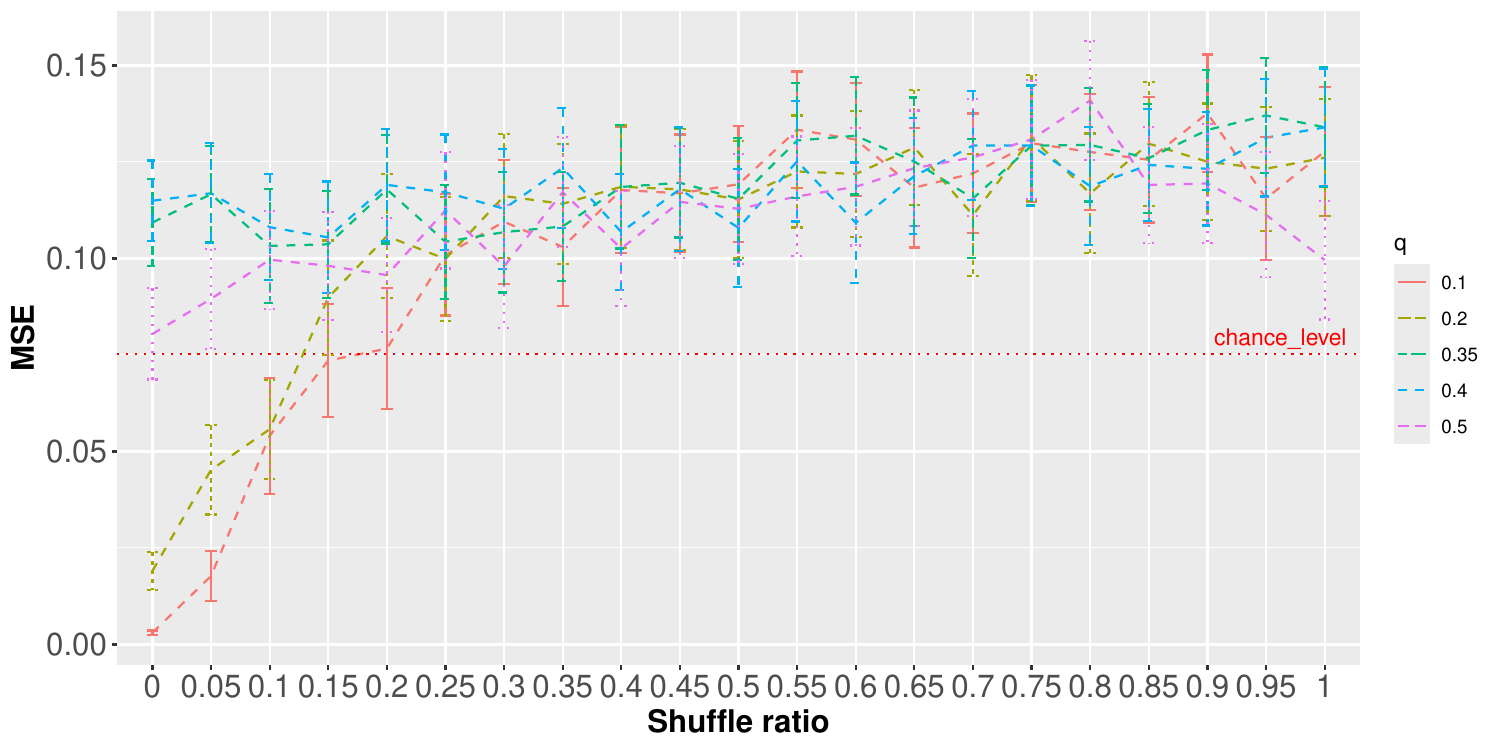}
    \caption{$l_2$ localizer on $d_{MV}^2$ mirror}
  \end{subfigure}
  \hfill
  \begin{subfigure}[b]{0.48\textwidth}
    \includegraphics[width=\textwidth]{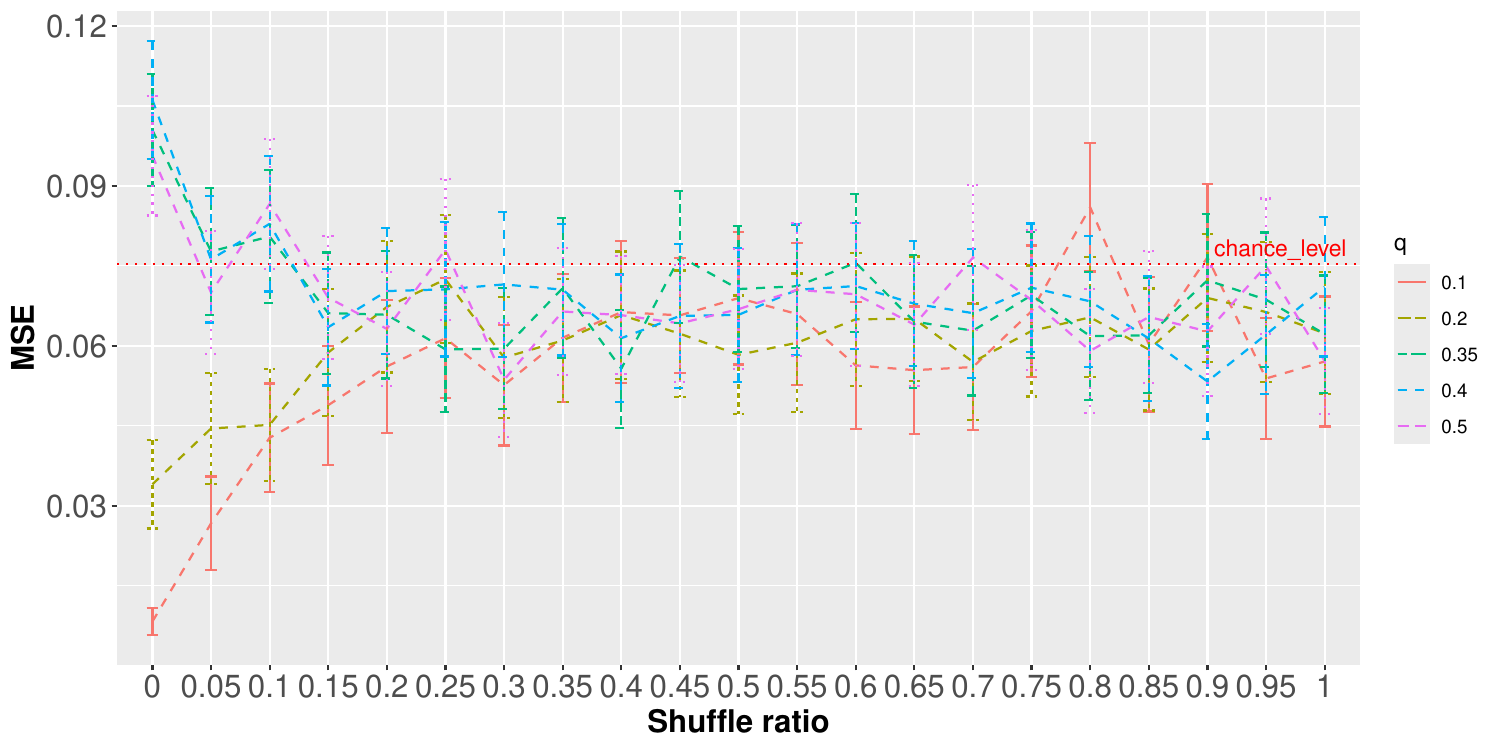}
    \caption{$l_\infty$ localizer on $d_{MV}^2$ mirror}
  \end{subfigure}

  \vspace{1em}
  \begin{subfigure}[b]{0.48\textwidth}
    \includegraphics[width=\textwidth]{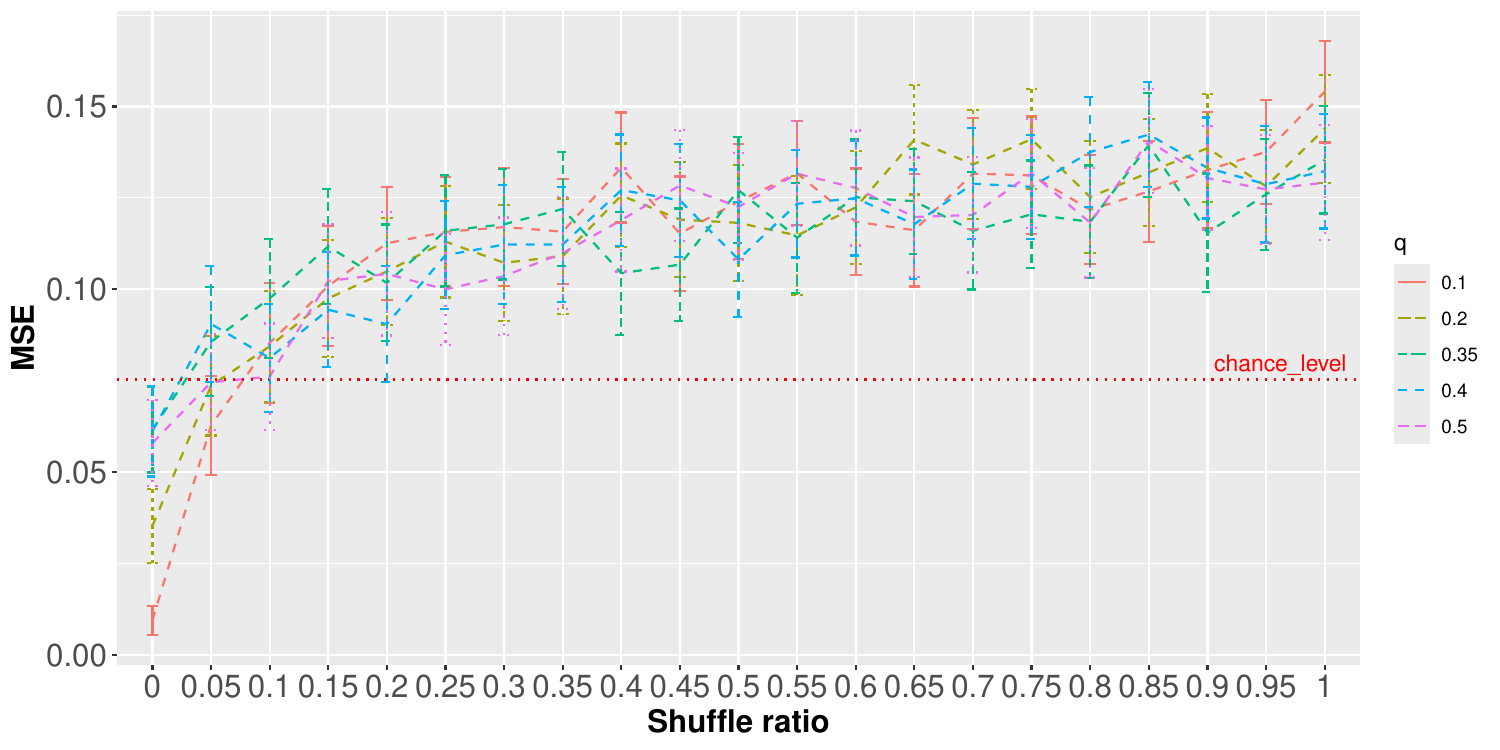}
    \caption{$l_2$ localizer on iso mirror with $d = 10$}
  \end{subfigure}
  \hfill
  \begin{subfigure}[b]{0.48\textwidth}
    \includegraphics[width=\textwidth]{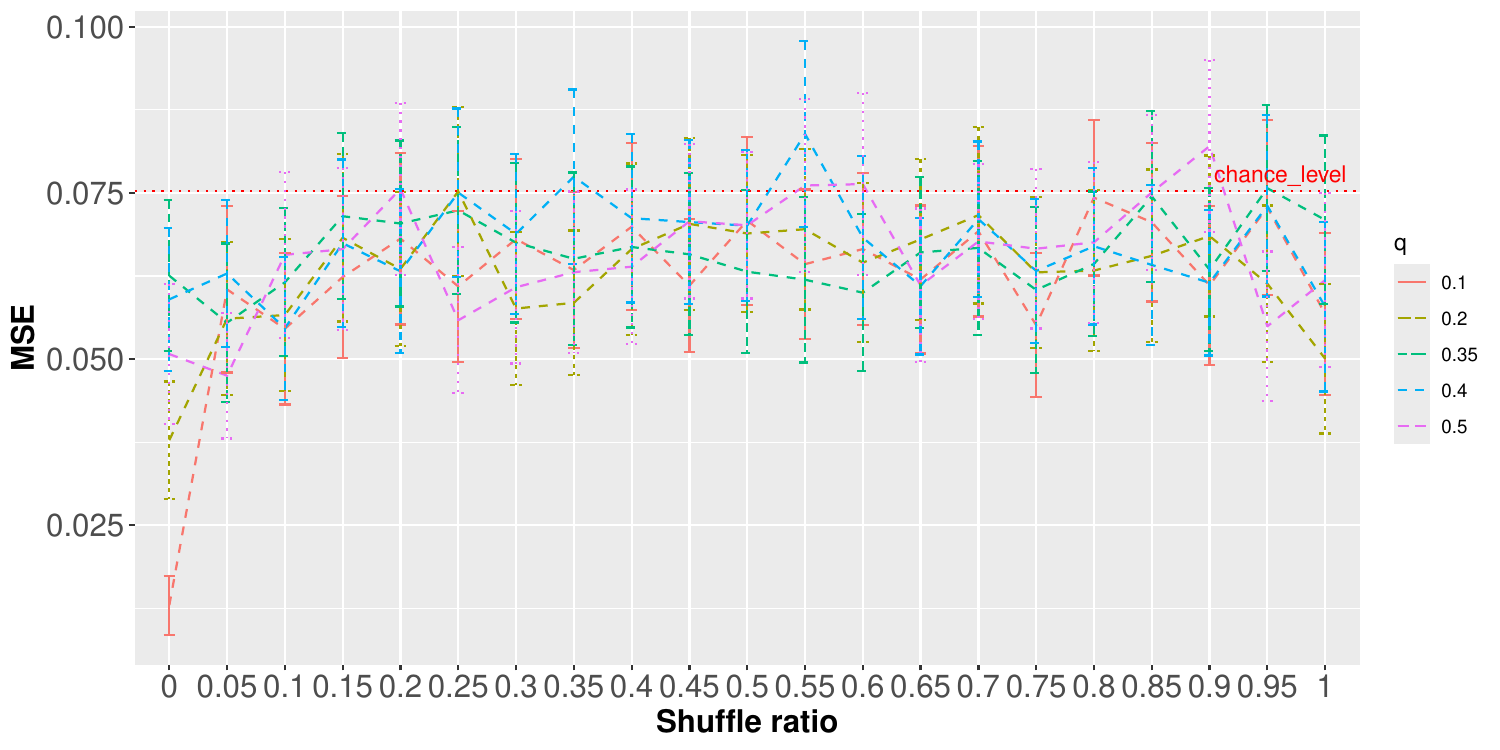}    
    \caption{$l_\infty$ localizer on iso mirror with $d = 10$}
  \end{subfigure}

  \caption{MSE versus shuffle ratio $\alpha$ under the Atlanta model for various combinations of mirror distances and localization methods. Each panel shows results using either the $d_{MV}$, $d_{MV}^2$, or iso-mirror with CMDS embedding dimension $d = 10$, paired with an $l_2$ or $l_\infty$ localizer. For fixed  $p = 0.4$ and each $q \in \{0.1, 0.2, 0.35, 0.4, 0.5\}$, results are averaged over $nmc = 100$ Monte Carlo replicates. Vertical bars denote 1.96 standard deviations, and the red dotted line marks the chance MSE level of 0.075 for $m = 40$. Across all configurations, none of the methods successfully detect the true changepoint location for all but a small fraction of misaligned vertices. MSE under the $l_2$ localizer often lies above the chance level, primarily due to bias toward boundary estimates, as the procedure tends to overfit noise near the endpoints of the time series.
  }
  \label{fig:atlanta_mse_vs_shuffle_appendix}
\end{figure}

\begin{figure}[htbp]
  \centering

  \begin{subfigure}[b]{0.48\textwidth}
    \includegraphics[width=\textwidth]{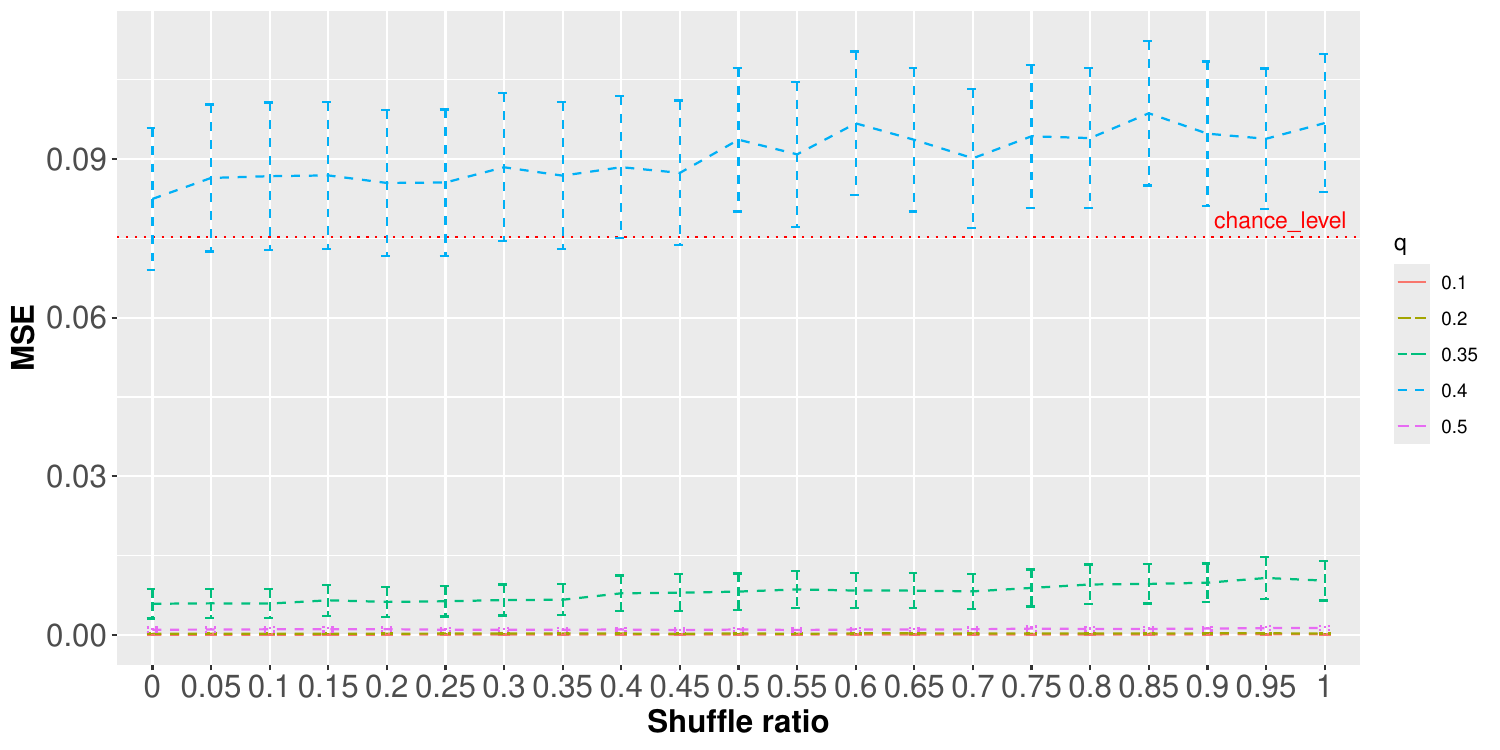}
    \caption{$l_2$ localizer on $d_{MV}$ mirror}
  \end{subfigure}
  \hfill\begin{subfigure}[b]{0.48\textwidth}
    \includegraphics[width=\textwidth]{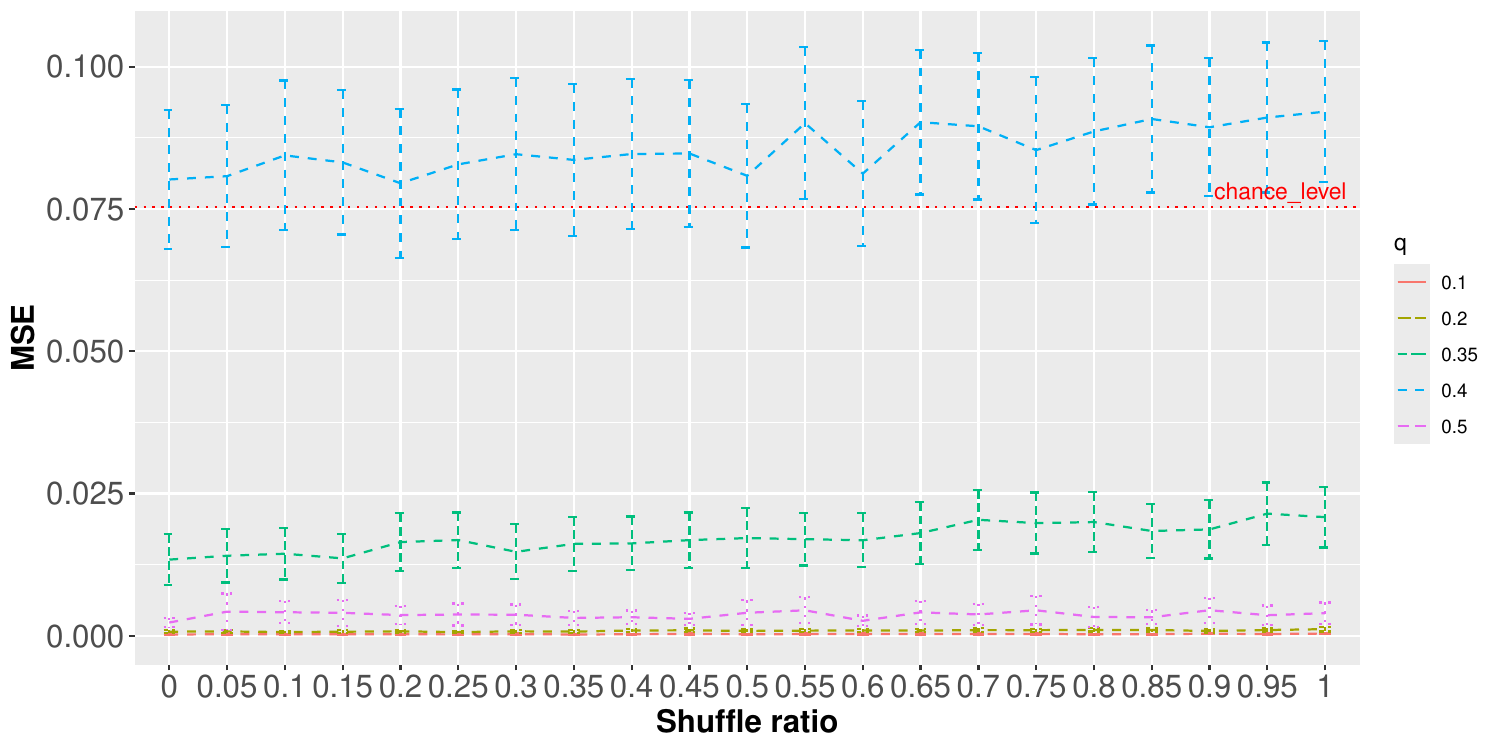}
    \caption{$l_\infty$ localizer on $d_{MV}$ mirror}
  \end{subfigure}

  \caption{
    MSE versus proportion of shuffled vertices $\alpha$ under the London model for two variants of the mirror localization procedure. The plots show performance using the $d_{MV}$ mirror combined with either an $l_2$ or $l_\infty$ localizer. For $p = 0.4$ and each $q \in \{0.1, 0.2, 0.35, 0.4, 0.5\}$, results are averaged over $nmc = 100$ Monte Carlo replicates. Error bars represent 1.96 standard deviations, and the red dotted line denotes the chance MSE level of 0.075 for $m = 40$. Across all shuffle ratios, the London model demonstrates robust changepoint recovery, with MSE consistently below chance and stable across increasing levels of vertex misalignment.
    Both localization methods exhibit strong performance, and no significant difference is observed between the $l_2$ and $l_\infty$ localizers.}
  \label{fig:london_mse_vs_shuffle_appendix}
\end{figure}

\section{Details of Swarm construction}
\label{sec:swarm}

For each frame \( t \), we observe the locations in \( \RR^2 \) of \( n = 100 \) agents.
We use the locations as the latent positions for each node, represented as a matrix \( \mathbf{X}_t \in \RR^{100 \times 2} \) by stacking the locations row-wise.
Next, we generate an RDPG using \( \mathbf{X}_t \) as latent position matrix, thresholding the \( \mathbf{P}_t = \mathbf{X}_t \mathbf{X}^{\top}_t \) matrix such that entries below 0 are set to 0 and entries above 1 are set to 1.
This process yields a time series of undirected graphs without self-loops, characterized by \( n = 100 \) nodes and \( m = 401 \) time points.
We can then implement our \( d_{MV} \)-based mirror method as outlined in Algorithm \ref{alg:iso-mirror}.

\section{Proof of theoretical results}
In this section, we provide proofs for our theoretical results.

\subsection{Proof of Theorem \ref{thm:independentdMV_convergence}}
We begin with an elementary bound that addresses the impact of derangement; we include it for reader reference.

\begin{proposition}\label{prop:derangement-2d}
For any $n>2$, denote $\sigma$ as a fixed derangement in $S_n$, that is $\forall i \in [n], \sigma(i) \neq i$. Consider two random variables in $X,Y\in\RR^1$ with a joint measure $\mu_{X,Y}$ on $\RR^2$. For a function $g: \RR^d \to \RR$, denote $\EE_{\mu}[g] : =\EE_{(x_1,x_2,\cdots,x_d) \sim \mu}[g(x_1,x_2,\cdots,x_d)]$.
Denote $\{(x_i,y_i)^n_{i=1}\}$ as $n$ i.i.d samples from $\mu_{X,Y}$ defined on the probability space $((\RR^2)^{\times n}, (\mathcal{B}(\RR^2))^{\times n}) = : (\Omega, \mathcal{B} ).$ The measure on $\Omega$ is denoted as $\PP_{\Omega}$. Denote $
\{ \left(x_i, y _{\sigma(i)} \right); i \in [n]\}$ as the sample after the derangement. Then we have for any function bounded on support of $\mathcal{S}: = \text{support of } \mu_X \otimes \mu_Y $, $f: \RR^2 \to \RR$, denote 
$
B_{f,\mathcal{S}} = \sup_{(x,y) \in \mathcal{S}} |f(x,y)|,
$
and for any $\epsilon > 0$, 
$$
\PP_{\Omega} \left( \left| \frac{1}{n} \sum^n_{i=1} f\left(x_i, y_{\sigma(i)}\right)  -  \EE_{ \mu_X \otimes \mu_Y}[f]   \right| \ge \epsilon \right) \le \frac{ 3 B_{f,\mathcal{S}}^2 }{n \epsilon ^2}.
$$
\end{proposition}


\begin{proof}
First we note for any $i$ recall $\sigma(i) \neq i$, then $(x_i, y_i) \ind (x_{\sigma(i)}, y_{\sigma(i)}) $. Now consider two projections $\Pi_x(x,y) = x$ and $\Pi_y(x, y)= y$. By independence, $\Pi_x(x_i, y_i) \ind \Pi_y(x_{\sigma(i)}, y_{\sigma(i)})$, that is, $x_i \ind y_{\sigma(i)}$. Further $x_i \sim \mu_X$, $y_{\sigma(i)} \sim \mu_Y$, thus for any $i$, $(x_i, y_{\sigma(i)}) \sim \mu_X \otimes \mu_Y $. Then 
$$
\EE \biggl[ \frac{1}{n} \sum^n_{i=1} f\left(x_i, y_{\sigma(i)}\right) \biggr] = \frac{1}{n} \sum^n_{i=1} \EE \biggl[ f\left(x_i, y_{\sigma(i)}\right) \biggr] =  \EE_{ \mu_X \otimes \mu_Y}[f].
$$
Next we show the variance is shrinking with $n$. Note
$$
\Var\biggl[ \frac{1}{n} \sum^n_{i=1} f\left(x_i, y_{\sigma(i)}\right) \biggr] = \frac{1}{n^2} \left( \sum^n_{i=1} \Var \biggr[ f\left(x_i, y_{\sigma(i)}\right) \biggr]  + \sum^n_{k=1} \sum^n_{t \neq k } \Cov \biggl[f\left(x_k, y_{\sigma(k)}\right) , f\left(x_t, y_{\sigma(t)}\right)   \biggr] \right),
$$
Note $f$ is bounded on $\mathcal{S}$, i.e., the support of $\mu_X \otimes \mu_Y$, so there exists a $c'_{f,\mathcal{S}}$ such that for all $k \in [n]$ we have
$$\Var [f\left(x_k, y_{\sigma(k)}\right)] = \Var_{(x,y)\sim\mu_X \otimes \mu_Y} [f(x,y)] \leq \EE_{(x,y)\sim\mu_X \otimes \mu_Y} [f(x,y)^2] \leq B_{f,\mathcal{S}}^2.$$ 
Thus the first term is bounded:
$$
\sum^n_{i=1} \Var \biggr[ f\left(x_i, y_{\sigma(i)}\right) \biggr] \leq n B_{f,\mathcal{S}}^2.
$$
Now consider the second term. For a given $k \in [n]$, there are at most two terms in the summand that are nonzero, which are both bounded by a constant. Given the derangement $\sigma$ and a fixed $k$, denote $t_1 :=\sigma(k)$; $t_2:=\sigma^{-1}(k)$. As long as $t \neq k, t_1, t_2$, then $\{k,\sigma(k)\}\cap\{t,\sigma(t)\}=\varnothing$ and 
$$
f\left(x_k, y_{\sigma(k)}\right) \ind f\left(x_t, y_{\sigma(t)}\right) \Rightarrow \Cov \biggl[f\left(x_k, y_{\sigma(k)}\right) , f\left(x_t, y_{\sigma(t)}\right)\biggr] = 0.
$$
For the at most two nonzero terms in the sum, Cauchy-Schwarz yields
$$\mathrm{Cov}\left[f(x_k,y_{\sigma(k)}),f(x_t,y_{\sigma(t)})\right]\leq \sqrt{\mathrm{Var}\left[f(x_k,y_{\sigma(k)})\right]\mathrm{Var}\left[f(x_t,y_{\sigma(t)})\right]}\leq B_{f,\mathcal{S}}^2,$$
so
$$
\sum^n_{k=1} \sum_{t \neq k } \Cov \biggl[f\left(x_k, y_{\sigma(k)}\right) , f\left(x_t, y_{\sigma(t)}\right) \biggr]\leq 2nB_{f,\mathcal{S}}^2.
$$
Thus
$$
\Var \biggl[ \frac{1}{n} \sum^n_{i=1} f\left( x_{i} - y_{\sigma(i)} \right) \biggl] \leq \frac{3B_{f,\mathcal{S}}^2}{n}.
$$
and the result now follows from Chebyshev's inequality.

\end{proof}

We now prove Theorem~\ref{thm:independentdMV_convergence}. 
\begin{proof}

Denote $\PP_{S_n}$ as the uniform distribution on $S_n$.
Because $\sigma$ is independent of the samples, the probability measure on $\Omega \times S_n$ is the product measure $\PP_{\Omega} \otimes \PP_{S_n}$, denoted as $\PP_{\Omega\times S_n}$.
Let $A$ denote the event
$$
A  :=\left\{\omega \times \sigma \in \Omega \times S_n: \left| \frac{1}{n} \sum^n_{i=1} f\left(x_i, y_{\sigma(i)}\right)  -  \EE_{ \mu_X \otimes \mu_Y}[f]   \right|> \epsilon\right\}. 
$$ Note $A\subseteq \Omega \times S_n $. We control $\PP_{\Omega\times S_n}(A)$. 
For any permutation $\sigma$, we divide $[n]$ into two disjoint sets based on the points fixed by $\sigma$, namely $\{i : \sigma(i) = i\}$ and its complement $\{i : \sigma(i) \neq i\}$. Denote the cardinality of the first set by $N^{\sigma}_{\text{fix}}$. We deduce that 
\begin{align*}
\frac{1}{n} &\sum^n_{i=1} f\left(x_i, y_{\sigma(i)}\right)  -  \EE_{ \mu_X \otimes \mu_Y}[f]\\
&= 
\underbrace{\frac{1}{n} \sum_{i:i=\sigma(i)}  f\left(x_i, y_{\sigma(i)}\right)}_{=:Y_n^{\sigma}} - \frac{N^{\sigma}_{\text{fix}}}{n}\EE_{ \mu_X \otimes \mu_Y}\\
&\quad +  \underbrace{\frac{n-N^{\sigma}_{\text{fix}}}{n} \left(\frac{1}{n-N^{\sigma}_{\text{fix}}} \sum_{i:i \neq \sigma(i)}  f\left(x_i, y_{\sigma(i)}\right) - \EE_{ \mu_X \otimes \mu_Y}  \right) }_{=:W_n^{\sigma}}.
\end{align*}
where $Y_n^{\sigma}$ and $W_n^{\sigma}$ are random variables defined above. 

Assume $f$ is upper bounded by $B_{f,\mathcal{S}}>\frac{\epsilon}{4}$ on $\mathcal{S}$. Denote $\alpha :  = \frac{\epsilon}{4B_{f,\mathcal{S}}}$, $ 0< \alpha <1 $. 
Denote by $E$ the event
$$
E : = \Omega \times \{\sigma \in S_n:N^{\sigma}_{\text{fix}} \leq \alpha n \}. $$
Since $|\sup_{(x,y) \in \mathcal{S}} f(x,y)| \leq B_{f,\mathcal{S}}$, we know that
$$
|Y_n^{\sigma}| \leq \frac{N^{\sigma}_{\text{fix}}B_{f,\mathcal{S}}}{n},
$$
and hence on $E$, we have that 
$|Y_n^{\sigma}| \leq \alpha B_{f,\mathcal{S}}=\frac{\epsilon}{4}$ and
$$\bigg|\frac{N^{\sigma}_{\text{fix}}}{n} \EE_{ \mu_X \otimes \mu_Y}[f]\bigg| \leq \frac{N^{\sigma}_{\text{fix}}}{n} \EE_{ \mu_X \otimes \mu_Y}[|f|] \leq  \frac{B_{f,\mathcal{S}} N^{\sigma}_{\text{fix}}}{n} \leq \alpha B_{f,\mathcal{S}}=\frac{\epsilon}{4}.
$$
Now, on $A$, we see that   
$|Y_n^{\sigma}+W_n^{\sigma} - \frac{N^{\sigma}_{\text{fix}}}{n} \EE_{ \mu_X \otimes \mu_Y}[f] | > \epsilon
$, so 
$$
|W_n^{\sigma} | > \epsilon - |Y_n^{\sigma}| - \bigg|\frac{N^{\sigma}_{\text{fix}}}{n} \EE_{ \mu_X \otimes \mu_Y}[f]\bigg|,
$$ 
Thus, on $E$, we see that $|Y_n^{\sigma}| \leq \frac{\epsilon}{4}$ and $\bigg|\frac{N^{\sigma}_{\text{fix}}}{n} \EE_{ \mu_X \otimes \mu_Y}[f]\bigg| \leq \frac{\epsilon}{4} $, and therefore on $A \cap E$, we have $|W_n^{\sigma}| > \frac{\epsilon}{2}$.
If we define $A'$ by
$$
A' := \left\{ \omega \times \sigma \in \Omega \times S_n:  |W_n^{\sigma}| > \frac{\epsilon}{2} \right\}
$$
then we conclude from the above that
$A \cap E \subseteq A' \cap E.$
We bound $\PP_{\Omega \times S_n}(A)$ by noting that 
\begin{align*}
\PP_{\Omega \times S_n}(A) = & \PP_{\Omega \times S_n}(A\cap E) + \PP_{\Omega \times S_n}(A \cap E^c) \\
\le & \PP_{\Omega \times S_n}(A' \cap E) + \PP_{\Omega \times S_n}(E^c)
\end{align*}
Consider the first term:
\[
\begin{aligned}
&\PP_{\Omega \times S_n}(A' \cap E) \\
& =\sum_{\sigma_0: N^{\sigma_0}_{\text{fix}} \leq \alpha n }  
\PP_{\Omega}\!\left(|W_n^{\sigma_0}|>\frac{\epsilon}{2}\right) \, \PP_{S_n}(\sigma_0) \quad 
\\
& = \sum_{\sigma_0: N^{\sigma_0}_{\text{fix}} \leq \alpha n }  \PP_{\Omega}\left( \bigg|\frac{n-N^{\sigma_0}_{\text{fix}}}{n} \left(\frac{1}{n-N^{\sigma_0}_{\text{fix}}} \sum_{i:i \neq \sigma_0(i)}  f\left(x_i, y_{\sigma_0(i)}\right) - \EE_{ \mu_X \otimes \mu_Y}[f]  \right) \bigg|> \frac{\epsilon}{2}  ~  \right)\PP_{S_n}(\sigma_0) \\
&=   \sum_{\sigma_0: N^{\sigma_0}_{\text{fix}} \leq \alpha n }  \PP_{\Omega}\left( \bigg| \frac{1}{n-N^{\sigma_0}_{\text{fix}}} \sum_{i:i \neq \sigma_0(i)}  f\left(x_i, y_{\sigma_0(i)}\right) - \EE_{ \mu_X \otimes \mu_Y}[f] \bigg|>\frac{n}{n-N^{\sigma_0}_{\text{fix}}} \frac{\epsilon}{2}  \right) \PP_{S_n}(\sigma_0) \\
& \leq \sum_{\sigma_0: N^{\sigma_0}_{\text{fix}} \leq \alpha n } 
\frac{ (12B_{f,\mathcal{S}}^2)(n-N^{\sigma_0}_{\text{fix}})^2 }{(n - N^{\sigma_0}_{\text{fix}} ) n^2 \epsilon ^2}  \PP_{S_n}(\sigma_0) \quad (\text{Proposition }\ref{prop:derangement-2d}) \\
& = \sum_{\sigma_0: N^{\sigma_0}_{\text{fix}} \leq \alpha n }  \frac{ ( 12B_{f,\mathcal{S}}^2)(n-N^{\sigma_0}_{\text{fix}}) }{n^2 \epsilon ^2} \PP_{S_n}(\sigma_0) \\
& \leq \sum_{\sigma_0: N^{\sigma_0}_{\text{fix}} \leq \alpha n } \frac{  12B_{f,\mathcal{S}}^2}{n \epsilon ^2} \PP_{S_n}(\sigma_0)\\
& = \frac{  12B_{f,\mathcal{S}}^2}{n \epsilon ^2} \PP_{S_n}\left( N^{\sigma_0}_{\text{fix}} \leq \alpha n \right)  \\
& \le \frac{  12B_{f,\mathcal{S}}^2}{n \epsilon ^2}.
\end{aligned}
\]

An easy computation shows $\EE_{S_n}[N^{\sigma}_{\text{fix}}] = 1$, so we bound $\PP(E^c)$ by Markov's inequality:
$$
\PP_{\Omega \times S_n}(E^c ) = \PP_{S_n}(N^{\sigma}_{\text{fix}} > \alpha n) \leq \frac{1}{\alpha n}.
$$
Combining these completes the proof.
\end{proof}

\subsection{Proof of Corollary \ref{Cor:partial shuffling}}
Let $A \subseteq \Omega \times S_{\alpha_n}$ denote the event
$$
\resizebox{\textwidth}{!}{$\displaystyle A : = \{ \omega \times \sigma_{\alpha_n} \in \Omega \times S_{\alpha_n} , \left| \hat{d}^2_{MV}(\mathbf{X}^{t_i}, P_{\alpha_n} \mathbf{X}^{t_j})-  
\left( \left(1-\frac{\alpha_n}{n}\right) d_{MV}(X_{t_i},X_{t_j}) +  \frac{\alpha_n}{n}\, \text{ind-}d^2_{MV}(X_{t_i}, X_{t_j})\right)
\right| \geq \epsilon  \}.$}
$$
We derive that 
\begin{align*}
\hat{d}^2_{MV}(\mathbf{X}_{t}, P_{n,\alpha} \mathbf{X}_{t'})  = & \frac{1}{n}\| \bX_t - P_{\alpha_n}  \bX_{t'}\|^2_F   \\
= & \frac{1}{n}\| \bX_t - (I_{n - \alpha_n} \oplus P_{\alpha_n} )   \bX_{t'}\|^2_F \\
= & \frac{1}{n} \left(  \sum^{n - \alpha_n}_{i=1} \left( \left( \bX_t \right)_i - (\bX_{t'})_i \right)^2  + \sum_{i = n - \alpha_n + 1}^n \left( \left( \bX_t \right)_i - (\bX_{t'})_{\sigma^{\alpha_n}(i)} \right)^2     \right) \\
= & \frac{n - \alpha_n}{n} ~ \frac{1}{n - \alpha_n}  \sum^{n - \alpha_n}_{i=1} \left( \left( \bX_t \right)_i - (\bX_{t'})_i \right)^2  + \frac{\alpha_n}{n} \frac{1}{\alpha_n}\sum_{i = n - \alpha_n + 1}^n \left( \left( \bX_t \right)_i - (\bX_{t'})_{\sigma^{\alpha_n}(i)} \right)^2. 
\end{align*}

Recall that $\{\left(\left( \bX_t \right)_i - (\bX_{t'})_i \right)^2 \} ^{n - \alpha_n}_{i=1}$ are i.i.d samples with mean $\EE[(X_t - X_{t'})^2] = d^2_{MV}(X_{t_i}, X_{t_j})$ and variance $\Var[(X_{t}-X_{t'})^2]$. Further, $X_t,X_{t'}$ are both supported on a subset of $[0,1]$, so $\Var[(X_{t}-X_{t'})^2] \leq 1$. Chebyshev's inequality guarantees that
\begin{align*}
& \PP_{\Omega \times S_{\alpha_n}} \left( \underbrace{  \bigg| \frac{1}{n - \alpha_n}  \sum^{n - \alpha_n}_{i=1} \left( \left( \bX_t \right)_i - (\bX_{t'})_i \right)^2  - d^2_{MV}(X_{t_i}, X_{t_j}) \bigg| \ge \epsilon }_{ = : B} \right) \\
= & \PP_{\Omega} \left(  \bigg| \frac{1}{n - \alpha_n}  \sum^{n - \alpha_n}_{i=1} \left( \left( \bX_t \right)_i - (\bX_{t'})_i \right)^2  - d^2_{MV}(X_{t_i}, X_{t_j}) \bigg| \ge \epsilon  \right)  \\
\le &  \frac{1}{(n-\alpha_n) \epsilon^2}.    
\end{align*}
For the second term we apply Theorem \ref{thm:independentdMV_convergence} with $f(x,y)=(x-y)^2$ and conclude that

$$
\PP_{\Omega \times S_{\alpha_n}} \left(  \underbrace{ \bigg| \frac{1}{ \alpha_n}  \sum^{n }_{i = n - \alpha_n +1 } \left( \left( \bX_t \right)_i - (\bX_{t'})_{\sigma^{\alpha_n}(i)} \right)^2  - \text{ind-}d^2_{MV}(X_{t_i}, X_{t_j}) \bigg| \ge \epsilon }_{ = : C} \right)  
\le   \frac{12}{\alpha_n \epsilon^2} + \frac{4}{\alpha_n \epsilon}  .  
$$
Finally, $A \subseteq B \cup C$, and hence 
$$
\PP_{\Omega \times S_{\alpha_n}} (A) \le \PP_{\Omega \times S_{\alpha_n}}( B \cup C) \leq \PP_{\Omega \times S_{\alpha_n}}(B) + \PP_{\Omega \times S_{\alpha_n}}(C) \leq 
\frac{1}{(n-\alpha_n) \epsilon^2} + \frac{12}{\alpha_n \epsilon^2} + \frac{4}{\alpha_n \epsilon}. 
$$


\subsection{Proof of Theorem \ref{thm:London_main}}
\label{sec:proofLondon}

The first bullet point in the unshuffled case is a property of $\tilde{\psi}^1_{d_{MV}}$, and it is proved in \cite{chen2024euclidean} Theorem 2.26. 
For the completely shuffled case $\alpha=1$, we calculate the ind-$d_{MV}$ for $X^L_{t_i}$ and $X^L_{t_j}$ from London Model:

When $t_i,t_j < t^*$, note $X^L_{t_i} \sim \frac{1}{m} \text{Bin}(i,p)$, $X^L_{t_j} \sim \frac{1}{m} \text{Bin}(j,p)$. 
\begin{align*}
\text{ind-}d_{MV}(X^L_{t_i},X^L_{t_j})^2 & =  \EE[(X^L_{t_i}-X^L_{t_j})^2] \quad \text{ ($w=1$ is the minimizer.)}  \\
& =  \Var[X^L_{t_i}-X^L_{t_j}]   + \left(\EE[X^L_{t_i}]-\EE[X^L_{t_j}]\right)^2 \\
& = \Var[X^L_{t_i}]+\Var[X^L_{t_j}] + \left(\EE[X^L_{t_i}]-\EE[X^L_{t_j}]\right)^2  \quad (X^L_{t_i} \ind X^L_{t_j})\\
& = \frac{1}{m^2}  (i+j)(p-p^2) + \frac{1}{m^2}\left(p(i-j)\right)^2. 
\end{align*}
Similarly, when $t_i < t^* < t_j$, $X^L_{t_i} \sim \frac{1}{m} \text{Bin}(i,p)$, $X^L_{t_j} = B_j+B_2$ with $ B_j\sim \frac{1}{m}\text{Bin}(j - t_m^*,q), B_2 \sim \frac{1}{m}\text{Bin}(t_m^*,p), B_j \ind B_2$.
\begin{align*}
\text{ind-}d_{MV}(X^L_{t_i},X^L_{t_j})^2
& = \Var[X^L_{t_i}]+\Var[X^L_{t_j}] + \left(\EE[X^L_{t_i}]-\EE[X^L_{t_j}]\right)^2  \\
& = \Var[X^L_{t_i}]+\Var[B_j] + \Var[B_2]+ \left(\EE[X^L_{t_i}]-\EE[X^L_{t_j}]\right)^2  \\
& = \frac{1}{m^2}\left( \left(p-p^2\right)i + \left(p-p^2\right)t^*_m + (j-t^*_m)(q-q^2) + \left(\left( j-t^*_m   \right)q +\left(t^*_m - i \right)p\right)^2   \right).
\end{align*}
Finally, when $t_i, t_j > t^*$, $X^L_{t_i} = B_i + B_2, X^L_{t_j} = B_j + B_2$, then 
\begin{align*}
\text{ind-}d_{MV}(X^L_{t_i},X^L_{t_j})^2
& = \Var[X^L_{t_i}]+\Var[X^L_{t_j}] + \left(\EE[X^L_{t_i}]-\EE[X^L_{t_j}]\right)^2  \\
& = \Var[B_i]+\Var[B_2]+\Var[B_j] + \Var[B_2]+ \left(\EE[X^L_{t_i}]-\EE[X^L_{t_j}]\right)^2  \\
& = \frac{1}{m^2}\left( \left(q-q^2\right)(i-t^*_m) + 2\left(p-p^2\right)t^*_m + (j-t^*_m)(q-q^2) + q(i-j)^2  \right) \\
& = \frac{1}{m^2}\left(  (q-q^2)(i+j)+2t^*_m(p-p^2-q+q^2)   + q(i-j)^2  \right).
\end{align*}
The above show the form of $\mathcal{D^L}^{(2)}_{\text{ind-}d_{MV}}$. Note that for any $m$, for all $i,j\in[m]$,
$$
(\mathcal{D^L}^{(2)}_{\text{ind-}d_{MV}})_{i,j} - (\mathcal{D}^{(2)}_Z)_{i,j} \sim O\left(\frac{1}{m}\right).
$$
And in the proof of Theorem 2.26 in \cite{chen2024euclidean}, all other properties are consequences of this property. Thus the same proof can go through for $\mathcal{D^L}^{(2)}_{\text{ind-}d_{MV}}$ too. Thus $\tilde{\psi}^L_{\text{ind-}d_{MV}}$ has the same property. 

For $0<\alpha<1$, we combine these two calculations to obtain for all $i,j\in[m]$
$$
(\mathcal{D^L}^{(2)}_{\alpha-d_{MV}})_{i,j} - (\mathcal{D}^{(2)}_Z)_{i,j} = \alpha (\mathcal{D^L}^{(2)}_{\text{ind-}d_{MV}})_{i,j} + (1-\alpha)(\mathcal{D^L}^{(2)}_{d_{MV}})_{i,j} - (\mathcal{D}_Z^{(2)})_{i,j} \sim O\left(\frac{1}{m}\right).
$$
By the same argument as above we have the result. 

For the second bullet point, we observe that whenever $s\leq t$, $X_s$ is stochastically dominated by $X_t$, since $\PP[X_s\leq X_t]=1$, so by \cite{de20211}, 
$$
W_1(X^L_{t_i},X^L_{t_j}) = |\EE [X^L_{t_i}] - \EE[X^L_{t_j}]|.
$$

For the third bullet point, consider a graph $\bA$ with $n$ vertices that is generated from a random dot product model with the distribution of random variable $X$. That is, latent positions $\{x_1, x_2, ... x_n\}$ are drawn i.i.d from $X \in \R^1$, then each edge is a Bernoulli sampled with probability equal to the product of the latent positions. The average degree is 
\begin{align*}
\EE \biggl[ \frac{1}{n} \sum^n_{i =1} \sum^n_{j \neq i} \bA_{i,j} \biggr] & = 
\EE \biggl[  \EE \biggl[ \frac{1}{n} \sum^n_{i =1} \sum^n_{j \neq i} \bA_{i,j} | \{x_1, x_2, ... x_n\}  \biggr] \biggr]  \\
& =
\EE \biggl[   \frac{1}{n} \sum^n_{i =1} \sum^n_{j \neq i} \EE[\bA_{i,j} | \{x_1, x_2, ... x_n\}]   \biggr] \\ 
& = 
\EE \biggl[   \frac{1}{n} \sum^n_{i =1} \sum^n_{j \neq i} x_i x_j  \biggr]  \\
& = 
\frac{1}{n} \sum^n_{i =1} \sum^n_{j \neq i} \EE[x_i] \EE[x_j]  \\
& = (n-1) \left(\EE[X]\right)^2.
\end{align*}
For the London model, assuming $c=0$,$t^*_m/m = t^*$,
$$
\EE[X^L_{t_i}] = \begin{cases}
p t_i & t_i \le t^*, \\
p \frac{t^*_m}{m} + q(t_i - \frac{t^*_m}{m}) = p t^* + q(t_i - t^*) & t_i > t^*.
\end{cases}
$$
\begin{align*}
d_{\mathrm{deg}}(X_t,X_s)&=(n-1)|\EE[X_t]^2-\EE[X_s]^2|\\
    &=(n-1)|\psi_Z(t)^2-\psi_Z(s)^2|\\
    &=|\psi_{\mathrm{deg}}(t)-\psi_{\mathrm{deg}}(s)|
\end{align*}
when we set $\psi_{\mathrm{deg}}(t)=(n-1)\psi_Z(t)^2$.

\subsection{Proof of Theorem~\ref{thm:London_MLE}}

Consider $n$ sample trajectories from the London model, $\{X_i(t):i\in[n], t\in[m]\}$. The likelihood is given by
\begin{align*}
    \mathrm{lik}(\{X_i(t)\}_{i,t}|p,q,t^*) &=\prod_{i=1}^n\prod_{t=1}^{t^*-1}p^{\mathbf{1}_{>0}(X_i(t)-X_i(t-1))}(1-p)^{1-\mathbf{1}_{>0}(X_i(t)-X_i(t-1))}\\
    &\qquad \times\prod_{t=t^*}^m q^{\mathbf{1}_{>0}(X_i(t)-X_i(t-1))}(1-q)^{\mathbf{1}_{>0}(X_i(t)-X_i(t-1))}\\
    &=\prod_{t=1}^{t^*-1}p^{\sum_i \mathbf{1}_{>0}(X_i(t)-X_{i}(t-1))}(1-p)^{n-\sum_i \mathbf{1}_{>0}(X_i(t)-X_i(t-1))}\\
    &\qquad\times\prod_{t=t^*}^m q^{\sum_i \mathbf{1}_{>0}(X_i(t)-X_i(t-1))}(1-q)^{n-\sum_i \mathbf{1}_{>0}(X_i(t)-X_i(t-1))}.
\end{align*}
To complete the proof of sufficiency, we observe that
$$\sum_{i=1}^n \mathbf{1}_{>0}(X_i(t)-X_i(t-1))=\sum_{i=1}^n (X_i(t)-X_i(t-1))/\delta=\sum_{j=0}^m j(c_t(j)-c_{t-1}(j))/\delta.$$

Now we consider maximizing the likelihood. If we fix the supposed changepoint at time $t$, the log-likelihood becomes
\begin{align*}
\ell(p,q)&= \sum_{s=1}^{t-1} \sum_{i=1}^n \left(\mathbf{1}_{>0}(X_i(s)-X_i(s-1))\log\left(\frac{p}{1-p}\right)\right)+n(t-1)\log(1-p)\\
&\qquad +\sum_{s=t}^m \sum_{i=1}^n\left(\mathbf{1}_{>0}(X_i(s)-X_i(s-1))\log\left(\frac{q}{1-q}\right)\right)+n(m-t+1)\log(1-q).
\end{align*}
For a given $i$, we have
\begin{align*}
\sum_{s=1}^{t-1}\mathbf{1}_{>0}(X_i(s)-X_i(s-1))&=\sum_{s=1}^{t-1}(X_i(s)-X_i(s-1))/\delta=(X_i(t-1)-c)/\delta,\\
\sum_{s=t}^m \mathbf{1}_{>0}(X_i(s)-X_i(s-1))&=\sum_{s=t}^m (X_i(s)-X_i(s-1))/\delta=(X_i(m)-X_i(t-1))/\delta.
\end{align*}
So summing over $i$ gives
\begin{align*}
\ell(p,q)&= \sum_{j=0}^{t-1}jc_{t-1}(j)\log\left(\frac{p}{1-p}\right)+n(t-1)\log(1-p)\\
&\qquad +\sum_{j=0}^m j(c_m(j)-c_{t-1}(j))\log\left(\frac{q}{1-q}\right)+n(m-t+1)\log(1-q).
\end{align*}
The remaining optimization now follows the standard argument for the MLE of a Bernoulli distribution.

\subsection{Proof of Theorem \ref{thm:Atlanta-main}}
\label{sec:AtlPf}

We want to determine the behavior of $d_{MV}^2(X_s^A,X_t^A)$ when $N$ is large. First we show $d_{MV}^2$ can be written as a trace. Let $T_p$ denote the transition probability matrix for the Markov chain defining one step of the Atlanta model when the transition probability is $p$, and let $M=[(i-j)^2]_{i,j=1}^N$. 
Since all random variables in Atlanta model are positive supported, we know $w=1.$ Assume $i<j<t^*_m$:
\begin{align*}
d^2_{MV}\left(X^A_{t_i} , X^A_{t_j}\right) & = \EE[  \left(X^A_{t_i} - X^A_{t_j}\right)^2 ] \\
&=\sum_{n=1}^N\sum_{l=1}^N\PP\bigl( X_{t_i}^A=S_n, X_{t_j}^A=S_l\bigr)(S_l-S_n)^2\\
& =  \sum^N_{n=1} \frac{1}{N}\sum^N_{l=1} \PP\left(X^A_{t_j} = S_l | X^A_{t_i} = S_n \right)  \left(S_l - S_n \right)^2 (X^A_{t_i}\text{ is uniform}) \\
&= \frac{1}{N} \sum^N_{n=1}\sum^N_{l=1} \left(T^{j-i}_{p}\right)_{n,l}(\delta_N(n-l))^2 \\
& = \frac{\delta^2_N}{N} \tr\left(T_p^{j-i} M \right).
\end{align*}
When $i<t^*_m<j$, 
$$
\PP(X^A_{t_i} = S_l | X^A_{t_j} = S_n ) = \left(T_p^{t_m^*-i}T_q^{j-t_m^*}\right)_{n,l},
$$ and the case with $t_m^*<i<j$ looks like the first case but with $T_q$ replacing $T_p$. Now we simplify the trace expression. 
We know from \cite{kouachi2006eigenvalues} the eigenvalues of $T_p$ are: 
\[
\lambda_k^{(p)}
\;=\;
1 \;-\; 2p
\;+\;
2p\;\cos\Bigl(\frac{(k-1)\,\pi}{N}\Bigr),
\quad
k \;= \;1,\;\dots,\;N.
\]
and corresponding eigenvectors are:

$$
(v_k)_j = \cos\left( \frac{ (k-1) \pi}{2 N} (2j-1) \right) , k=1,2,\cdots ,N , \quad j = 1,2, \cdots, N.
$$

Also note $$
||v_k||^2 = \sum_{j=1}^N \cos ^2\left(\frac{k \pi(2 j-1)}{2 N}\right)=\frac{N}{2}, \quad \text { for } \quad k \neq 0
$$ then we have the orthonormal eigenvectors $u_k = \sqrt{\frac{2}{N}}v_k$ for $k>1$ and $u_1 = \sqrt{\frac{1}{N}}v_1$. Note that $u_k$s are not functions of $p$ and are common eigenvectors of all $T_p$ matrices.
\begin{lemma}
    Sum of cosines and sines of an arithmetic sequence.
    \[
\sum_{k=0}^{n-1} \cos(a + k \cdot d) = \frac{\sin\left(\frac{n \cdot d}{2}\right)}{\sin\left(\frac{d}{2}\right)} \cdot \cos\left(\frac{2a + (n-1) \cdot d}{2}\right)
\]
\[
\sum_{k=0}^{n-1} \sin(a + k \cdot d) = \frac{\sin\left(\frac{n \cdot d}{2}\right)}{\sin\left(\frac{d}{2}\right)} \cdot \sin\left(\frac{2a + (n-1) \cdot d}{2}\right)
\]
\end{lemma}
\begin{proof}
Note that
\[
\sum_{k=0}^{n-1} e^{(a + k \cdot d)i} = e^{ai}\frac{e^{ndi} - 1}{e^{di}-1}= e^{ai}\frac{e^{ndi/2} - e^{-ndi/2}}{e^{di/2}-e^{-di/2}}\frac{e^{ndi/2}}{e^{di/2}} =  \frac{\sin\left(\frac{n \cdot d}{2}\right)}{\sin\left(\frac{d}{2}\right)}e^{ai + ndi/2-di/2}
\]
By taking the real part of the above equation we get
\[
\sum_{k=0}^{n-1} \cos(a + k \cdot d) = \frac{\sin\left(\frac{n \cdot d}{2}\right)}{\sin\left(\frac{d}{2}\right)} \cdot \cos\left(\frac{2a + (n-1) \cdot d}{2}\right),
\]
and by taking the imaginary part we get
\[
\sum_{k=0}^{n-1} \sin(a + k \cdot d) = \frac{\sin\left(\frac{n \cdot d}{2}\right)}{\sin\left(\frac{d}{2}\right)} \cdot \sin\left(\frac{2a + (n-1) \cdot d}{2}\right).
\]
\end{proof}
\begin{lemma}
    Define $\alpha_k = u_k^\top M u_k$, then with $\omega_k=(k-1)\pi/2,$
    $$
\alpha_k = 
\begin{cases}
    \frac{N(N+1)(N-1)}{6} &\quad k = 1\\
    -\frac{\cos^2\left(\omega_k/N\right)}{N\sin^4\left(\omega_k/N\right)}&\quad \text{even } k\\
    0 &\quad \text{odd } k>1\\
\end{cases}.
$$

\end{lemma}
\begin{proof}
    For $k = 1$, 
    \[
    \alpha_1 = u_1^\top M u_1 = \frac{1}{N} \sum_{i=1}^N\sum_{j=1}^N M_{ij} = \frac{1}{N} \sum_{i=1}^N\sum_{j=1}^N (i-j)^2 = \frac{1}{N} \sum_{i=1}^N\sum_{j=1}^N (i^2-2ij+j^2)
    \]
    \[ 
    =\frac{1}{N} \left(\frac{N(N+1)(2N+1)}{6}\cdot N - 2\left(\frac{N(N+1)}{2}\right)^2+N\cdot \frac{N(N+1)(2N+1)}{6}\right)
    \]
    \[ 
    =\frac{N(N+1)(N-1)}{6}
    \]
    If $k>1$, $\left(v_k\right)_i = \cos\left(\frac{\omega_k}{N}(2i-1)\right)$
    \[
    \alpha_k = u_k^\top M u_k = \frac{2}{N} \sum_{i=1}^N\sum_{j=1}^N \left(v_{k}\right)_i\,M_{ij}\, \left(v_{k}\right)_j = \frac{2}{N} \sum_{i=1}^N\sum_{j=1}^N \cos\left( \frac{\omega_k}{N} (2i-1) \right)(i-j)^2\cos\left( \frac{\omega_k}{N} (2j-1) \right)
    \]
    \[= \frac{2}{N} \left(\sum_{i=1}^N i^2 \cos\left( \frac{\omega_k}{N} (2i-1) \right)\right)\left(\sum_{j=1}^N\cos\left( \frac{\omega_k}{N} (2j-1) \right)\right)
    \]
    \[- \frac{2}{N} \cdot 2\left(\sum_{i=1}^N i \cos\left( \frac{\omega_k}{N} (2i-1) \right)\right)\left(\sum_{j=1}^N j\cos\left( \frac{\omega_k}{N} (2j-1) \right)\right)
    \]
    \[+ \frac{2}{N} \left(\sum_{i=1}^N \cos\left( \frac{\omega_k}{N} (2i-1) \right)\right)\left(\sum_{j=1}^N j^2 \cos\left( \frac{\omega_k}{N} (2j-1) \right)\right)
    \]
    Let $f_d(k) = \sum_{i=1}^N i^d \cos\left(\frac{\omega_k}{N}(2i-1)\right)$, then we have \[
    \alpha_k = \frac{2}{N}\left(f_2(k)\cdot f_0(k)-2f_1(k)^2 + f_0(k)f_2(k)\right).
    \]
    Note that $f_0(k) = \sum_{i=1}^N \cos\left(\frac{\omega_k}{N}(2i-1)\right)$, which is a sum of cosines of an arithmetic sequence, so
    \[
    f_0(k) = \frac{\sin\left(\omega_k\right)}{\sin\left(\omega_k/N\right)}\cdot\cos\left(\omega_k\right).
    \]
    Therefore, for integer $k$ we have $f_0(k) = 0$.
    Now $\alpha_k$ simplifies to $-\frac{4}{N}f_1(k)^2$. In order to find $f_1(k)$, first, we calculate sum of sines of the same arithmetic sequence
    \[
    \sum_{i=1}^N \sin\left(\frac{\omega_k}{N}(2i-1)\right) = \frac{\sin\left(\omega_k\right)}{\sin\left(\omega_k/N\right)}\cdot\sin\left(\omega_k\right) = \frac{\sin^2\left(\omega_k\right)}{\sin\left(\omega_k/N\right)}.
    \]
    Now if we take derivative of both sides with respect to $\frac{\omega_k}{N}$, we get
    \[
    \sum_{i=1}^N \left(2i-1\right)\cdot\cos\left(\frac{\omega_k}{N}(2i-1)\right)  = \frac{2N \cos\left(\omega_k\right) \sin\left(\omega_k\right)}{\sin\left(\omega_k/N\right)} - \frac{\cos\left(\omega_k/N\right) \sin^{2}\left(\omega_k\right)}{\sin^{2}\left(\omega_k/N\right)}.
    \]
    Therefore,
    \[
    f_1(k) = \frac{N \cos\left(\omega_k\right) \sin\left(\omega_k\right)}{\sin\left(\omega_k/N\right)} - \frac{\cos\left(\omega_k/N\right) \sin^{2}\left(\omega_k\right)}{2\sin^{2}\left(\omega_k/N\right)}+ \frac{f_0(k)}{2}
    \]
    which for integer $k$ simplifies to
    \[
    f_1(k) = - \frac{\cos\left(\omega_k/N\right) \sin^{2}\left(\omega_k\right)}{2\sin^{2}\left(\omega_k/N\right)}
    \]
    If $k$ is odd, then $\sin^{2}\left(\omega_k\right) = 0$, and if $k$ is even, $\sin^{2}\left(\omega_k\right) = 1$, so we have 
    \[
f_1(k) = 
\begin{cases}
    - \frac{\cos\left(\omega_k/N\right)}{2\sin^{2}\left(\omega_k/N\right)}&\quad \text{even } k\\
    0 &\quad \text{odd } k
\end{cases}
\]
Recall that $\alpha_k = -\frac{4}{N}f_1(k)^2$, therefore, for $k > 1$ we have 
\[
\alpha_k = 
\begin{cases}
    -\frac{\cos^2\left(\omega_k/N\right)}{N\sin^4\left(\omega_k/N\right)}&\quad \text{even } k\\
    0 &\quad \text{odd },
\end{cases}
\]
or in general,
\[
\alpha_k = 
\begin{cases}
    \frac{N(N+1)(N-1)}{6} &\quad k = 1\\
    -\frac{\cos^2\left(\omega_k/N\right)}{N\sin^4\left(\omega_k/N\right)}&\quad \text{even } k\\
    0 &\quad \text{odd } k>1\\
\end{cases}
\]
\end{proof}
\begin{lemma}
    For non-negative integer $k < N$, we have 
    \[
\mathrm{tr }(T_p^k\,M) =  2(N-1)kp -4\sum_{t=2}^k\frac{(-1)^t}{t-1}\binom{k}{t} \binom{2t-4}{t-2}p^t
\] 
\end{lemma}
\begin{proof}
    Let $T_p = U\Sigma_pU^\top$ be the eigendecomposition of $T_p$. Note that since the eigenvectors do not depend on $p$, we do not index the matrix $U$ with $p$. 
    \[
\mathrm{tr }(T_p^k\,M)
= \mathrm{tr} \left(  U \Sigma_p^k U^{\top} M \right)
= \mathrm{tr} \left(   \Sigma_p^k U^{\top} M U \right)
= \sum_{m=1}^{N} \alpha_m\,\bigl(\lambda_m^{(p)}\bigr)^k \text{ with  } \alpha_m = u^{\top}_m M u_m.
\] 
Before simplifying the last summation, note that $\sum_{m = 1}^N \alpha_m = \tr\left(M\right) = 0$, we will use this later. Now plugging in the values for $\alpha_m$ and $\lambda_m^{(p)}$ from the previous lemmas,
\begin{align*}
\mathrm{tr }(T_p^k\,M) &= \sum_{m=1}^{N} \alpha_m\,\bigl(\lambda_m^{(p)}\bigr)^k 
= \alpha_1 +\sum_{\text{even }m} \alpha_m\,\bigl(\lambda_m^{(p)}\bigr)^k\\
&= \alpha_1 -\sum_{\text{even }m} \frac{\cos^2\left(\omega_m/N\right)}{N\sin^4\left(\omega_m/N\right)}\,\left(1-2p+2p\cos\left(\frac{(m-1)\pi}{N}\right)\right)^k\\
&= \alpha_1 -\sum_{\text{even }m} \frac{\cos^2\left(\omega_m/N\right)}{N\sin^4\left(\omega_m/N\right)}\,\left(1-4p\sin^2\left(\frac{\omega_m}{N}\right)\right)^k \\
&=\alpha_1 - \sum_{\text{even }m} \frac{\cos^2\left(\omega_m/N\right)}{N\sin^4\left(\omega_m/N\right)} + \frac{4kp}{N}\sum_{\text{even }m} \frac{\cos^2\left(\omega_m/N\right)}{\sin^2\left(\omega_m/N\right)}\\
&\quad -\sum_{t=2}^k\binom{k}{t}\frac{(-4p)^t}{N}\sum_{\text{even }m}\cos^2\left(\omega_m/N\right)\sin^{2t-4}\left(\omega_m/N\right)
\end{align*}
Note that 
\[
\alpha_1 - \sum_{\text{even }m} \frac{\cos^2\left(\omega_m/N\right)}{N\sin^4\left(\omega_m/N\right)} = \sum_{m=1}^N\alpha_m = 0,
\]
Now we calculate the third term's summation,  we follow the idea of Cauchy's proof of the Basel problem.  

\[
\frac{\cos\left(Nx\right) + i\sin\left(Nx\right)}{\sin^{N}\left(x\right)} = \frac{\left(\cos\left(x\right) + i\sin\left(x\right)\right)^{N}}{\sin^{N}\left(x\right)} =\left( \cot(x) + i\right)^{N}=\sum_{j = 0}^{N}\binom{N}{j}\cot^{N-j}\left(x\right) i^{j}
\]
By taking the real part of the above equations we get
\[
\frac{\cos\left(Nx\right)}{\sin^{N}\left(x\right)} = \sum_{j = 0}^{\lfloor N/2\rfloor}\binom{N}{2j}\cot^{N-2j}\left(x\right) (-1)^{j}
\]
Note that for even $m\leq N$, $\cos\left(N\left(\omega_m/N\right)\right) = \cos\left((m-1)\pi\right/2) = 0$ and $\sin\left(\omega_m/N\right)\neq 0$, so $\omega_m/N$ is a root of the left hand side of the above equation. If $N$ is an even number, then $\cot^2\left(\omega_m/N\right)$ are roots of the polynomial
\[
P(t) = \sum_{j = 0}^{N/2}\binom{N}{2j}(-1)^{j} t ^{N/2 - j}.
\]
All of the numbers $\cot^2\left(\omega_m/N\right)$ for even $0<m\leq N$ are distinct and there are $N/2$ of them. Therefore, they are all of the roots of the above polynomial. From the Vieta formula for the sum of the roots of a polynomial, we get
\[
\sum_{\text{even } m}\cot^2\left(\omega_m/N\right) =  -\frac{-\binom{N}{2}}{\binom{N}{0}}  = \frac{N(N-1)}{2}
\]
For odd $N$, we have 
\[
\frac{\cos\left(Nx\right)}{\cot\left(x\right)\sin^{N}\left(x\right)} = \sum_{j = 0}^{ (N-1)/2}\binom{N}{2j}\cot^{N-1-2j}\left(x\right) (-1)^{j},
\]
and $\omega_m/N$ is again a root of left hand side, therefore, $\cot^2\left(\omega_m/N\right)$ are roots of the polynomial
\[
P(t) = \sum_{j = 0}^{(N-1)/2}\binom{N}{2j}(-1)^{j} t ^{(N-1)/2 - j}.
\]
As $P$ is a $(N-1)/2$ degree polynomial, $\cot^2\left(\omega_m/N\right)$ are all the roots of $P$, and again by the Vieta formula
\[
\sum_{\text{even } m}\cot^2\left(\omega_m/N\right)=  -\frac{-\binom{N}{2}}{\binom{N}{0}}  = \frac{N(N-1)}{2}.
\]
Therefore, in general,
\[
\sum_{\text{even }m} \frac{\cos^2\left(\omega_m/N\right)}{\sin^2\left(\omega_m/N\right)} = \sum_{\text{even } m} \cot^2\left(\omega_m/N\right) = \frac{N(N-1)}{2}
\]

Now we show that for $N > k\geq t$
\[
\sum_{\text{even }m}\cos^2\left(\omega_m/N\right)\sin^{2t}\left(\omega_m/N\right) = c_t N,
\]
where \[
c_t = \frac{1}{2^{2t+2}(t+1)} \binom{2t}{t}
\]
First, note that $\sin^{2t}\left(x\right)$ is an even function with period $2\pi$, so we can write a Fourier series using $\cos\left(rx\right)$ functions for nonnegative integer $r$. 
\[
\sin^{2t}\left(x\right) = \sum_{r= 0 }^\infty A_{r,t} \cos\left(rx\right),
\]
where,
\newcommand{\diff}{\mathop{}\!\mathrm{d}}
\[
A_{0,t} = \frac{1}{2\pi}\int_0^{2\pi} \sin^{2t}\left(x\right)\diff x
\]
and for $r > 0$
\[
A_{r,t} = \frac{1}{\pi}\int_0^{2\pi}\sin^{2t}\left(x\right)\cos\left(rx\right)\diff x.
\]
Let $z = e^{ix}$, then
\[
I_{r,t} = \int_0^{2\pi}\sin^{2t}\left(x\right)\cos\left(rx\right)\diff x = \int_{|z| = 1}\left(\frac{z-z^{-1}}{2i}\right)^{2t}\left(\frac{z^r+z^{-r}}{2}\right)\frac{1}{iz}\diff z 
\]
\[
= \frac{1}{2^{2t+1}i\,(-1)^{t}}\int_{|z| = 1}\frac{\left(z^2-1\right)^{2t}\left(z^{2r}+1\right)}{z^{2t+r+1}}\diff z .
\]
The final integral has a singularity at $z = 0$, therefore, $I_{r,t}$ simplifies to 
\[
I_{r,t} = \frac{1}{2^{2t+1}i\,(-1)^{t}}\,\,2\pi i \,\,\mathop{\mathrm{Res}}_{z = 0}\left(\frac{\left(z^2-1\right)^{2t}\left(z^{2r}+1\right)}{z^{2t+r+1}}\right) =  \frac{\pi}{2^{2t}\,(-1)^t} \,\,\mathop{\mathrm{Res}}_{z = 0}\left(\frac{\left(z^2-1\right)^{2t}\left(z^{2r}+1\right)}{z^{2t+r+1}}\right)
\]
Note that the residue term is equal to the coefficient of $z^{-1}$ in the Laurent series of the integrand around the singularity. Which here is equal to coefficient of $z^{2t+r}$ in expansion of $\left(z^2-1\right)^{2t}\left(z^{2r}+1\right)$. 
\[
\left(z^2-1\right)^{2t}\left(z^{2r}+1\right) = \sum_{j = 0}^{2t}\binom{2t}{j} z^{2j+2r}\left(-1\right)^{2t-j}+\sum_{j = 0}^{2t}\binom{2t}{j} z^{2j}\left(-1\right)^{2t-j}
\]
\[
= \sum_{j = r}^{2t+r}\binom{2t}{j-r} \left(-1\right)^{j-r}z^{2j}+\sum_{j = 0}^{2t}\binom{2t}{j} \left(-1\right)^{j}z^{2j}
\]
Now note that if $r$ is odd, $2t+r$ is odd, but the above series has no non-zero odd power of $z$. Therefore, the residue term would be zero, and $I_{r,t} = 0$ for odd $r$. For even $r$, the terms where $j = t + r/2$ are the desired ones in both summations. If $r > 2t$, we have $t + r/2 > 2t$ and $t + r/2 < r$, so there is not a $z^{2t+r}$ in any of the summations. But if $r \leq 2t$, the coefficient would be
\[
\binom{2t}{t-r/2} \left(-1\right)^{t -r/2}+\binom{2t}{t + r/2} \left(-1\right)^{t + r/2} = 2\binom{2t}{t-r/2} \left(-1\right)^{t +r/2}
\]
Therefore,
\[
I_{r,t} =  \begin{cases}
    \frac{\pi(-1)^{r/2}}{2^{2t-1}}\binom{2t}{t-r/2}\quad &\text{even} \; r \leq 2t\\
    0 \quad &\text{otherwise},
\end{cases}
\]
which results in
\[
A_{r,t} =  \begin{cases}
    \frac{1}{2^{2t}}\binom{2t}{t}\quad &r = 0\\
    \frac{(-1)^{r/2}}{2^{2t-1}}\binom{2t}{t-r/2}\quad &\text{even} \; 0<r \leq 2t\\
    0 \quad &\text{otherwise},
\end{cases}
\]

By using the above Fourier series we have 
\[\sum_{\text{even } m} \sin^{2t}\left(\omega_m/N\right) =\sum_{r= 0 }^t A_{2r,t} \sum_{\text{even } m} \cos\left(2r\omega_m/N\right). \]
When $N>k\geq r$, for $r>0$, $N$ does not divide $r$ and we can use the same lemma on sum of cosines of an arithmetic sequence to obtain
\[
\sum_{\text{even } m} \cos\left(2r\omega_m/N\right) = 
\begin{cases}
    \frac{\sin\left((N-1)r\pi/(2N)\right)}{\sin\left(r\pi/N\right)}\cdot\cos\left((N-1)r\pi/(2N)\right) \quad & \text{odd } N\\
    \frac{\sin\left(r\pi/2\right)}{\sin\left(r\pi/N\right)}\cdot\cos\left(r\pi/2\right) \quad & \text{even } N,
\end{cases}
\]
which simplifies to 
\[
\sum_{\text{even } m} \cos\left(2r\omega_m/N\right) = 
\begin{cases}
    -\frac{(-1)^{r}}{2} \quad & \text{odd } N\\
    0 \quad & \text{even } N
\end{cases}
\]
after using the double-angle and angle addition formulas for $\sin(x)$. For $r = 0$, we have \[
\sum_{\text{even } m} \cos\left(r\omega_m/N\right) = \left\lfloor\frac{N}{2}\right\rfloor
\]
For even $N>2t$,
\[
\sum_{\text{even } m} \sin^{2t}\left(\omega_m/N\right) = A_{0,t}\frac{N}{2} =  \frac{N}{2^{2t+1}}\binom{2t}{t}.
\]
For odd $N>2t$,
\begin{align*}
\sum_{\text{even } m} \sin^{2t}\left(\omega_m/N\right) &= A_{0,t}\frac{N-1}{2} + \sum_{r=1}^{t} A_{2r}\frac{(-1)^{r}}{2} =  \frac{N-1}{2^{2t+1}}\binom{2t}{t} - \sum_{r=1}^{t} \frac{1}{2^{2t}} \binom{2t}{t-r}\\
&= \frac{N-1}{2^{2t+1}}\binom{2t}{t} - \frac{1}{2^{2t}}\sum_{j = 0 }^{t-1}  \binom{2t}{j}\\
&= \frac{N-1}{2^{2t+1}}\binom{2t}{t} - \frac{1}{2^{2t+1}}\left(2^{2t}-\binom{2t}{t}\right)\\
&= \frac{N}{2^{2t+1}}\binom{2t}{t} -\frac{1}{2}
\end{align*}

Now note that the original summation
\(
\sum_{\text{even }m}\cos^2\left(\omega_m/N\right)\sin^{2t}\left(\omega_m/N\right)
\)
can be written as 
\[
\sum_{\text{even }m}\left(1-\sin^2\left(\omega_m/N\right)\right)\sin^{2t}\left(\omega_m/N\right)
= \sum_{\text{even }m}\sin^{2t}\left(\omega_m/N\right) - 
\sum_{\text{even }m}\sin^{2t+2}\left(\omega_m/N\right)
\] 
For even $N$,
\[
\sum_{\text{even }m}\cos^2\left(\omega_m/N\right)\sin^{2t}\left(\omega_m/N\right) = \frac{N}{2^{2t+1}}\binom{2t}{t} - \frac{N}{2^{2t+3}}\binom{2t+2}{t+1} = \frac{N}{2^{2t+2}(t+1)} \binom{2t}{t}
\]
If $N$ is odd,
\[
\sum_{\text{even }m}\cos^2\left(\omega_m/N\right)\sin^{2t}\left(\omega_m/N\right) = \frac{N}{2^{2t+1}}\binom{2t}{t}-\frac{1}{2} - \frac{N}{2^{2t+3}}\binom{2t+2}{t+1}+\frac{1}{2} = \frac{N}{2^{2t+2}(t+1)} \binom{2t}{t}
\]
So we proved that
\[
\sum_{\text{even }m}\cos^2\left(\omega_m/N\right)\sin^{2t}\left(\omega_m/N\right) = c_t N
\]
where
\[
c_t = \frac{1}{2^{2t+2}(t+1)} \binom{2t}{t}
\]

Combining all of these terms, we obtain
\begin{align*}
\mathrm{tr }(T_p^k\,M) &= \alpha_1 - \sum_{\text{even }m} \frac{\cos^2\left(\omega_m/N\right)}{N\sin^4\left(\omega_m/N\right)} + \frac{4kp}{N}\sum_{\text{even }m} \frac{\cos^2\left(\omega_m/N\right)}{\sin^2\left(\omega_m/N\right)}\\
&\quad -\sum_{t=2}^k\binom{k}{t}\frac{(-4p)^t}{N}\sum_{\text{even }m}\cos^2\left(\omega_m/N\right)\sin^{2t-4}\left(\omega_m/N\right)\\
&= 0 + \frac{4kp}{N}\cdot\frac{N(N-1)}{2}-\sum_{t=2}^k\binom{k}{t}\frac{(-4p)^t}{N}\cdot c_{t-2} N\\
&= 2(N-1)kp -4\sum_{t=2}^k\frac{(-1)^t}{t-1}\binom{k}{t} \binom{2t-4}{t-2}p^t 
\end{align*}

\begin{lemma}
    For non-negative integers $k$, $l$, and $N> k + l$
    \[
\mathrm{tr }(T_p^kT_q^l\,M) =2(N-1)(kp+lq)  -\sum_{d=2}^{k+l}\sum_{t=\max(0,d-l)}^{\min(d,k)}(-1)^{d}\frac{4}{d-1} \binom{2d-4}{d-2}\binom{k}{t}\binom{l}{d-t}p^tq^{d-t}
\]
\end{lemma}
\begin{proof}
Proof of this lemma is similar to the previous one. 
\begin{align*}
\mathrm{tr }(T_p^kT_q^l\,M) &= 
\mathrm{tr} \left(  U \Sigma_p^k U^{\top}U \Sigma_q^l U^{\top} M \right)
= \mathrm{tr} \left(   \Sigma_p^k \Sigma_q^l U^{\top} M U \right)\\
&= \sum_{m=1}^{N} \alpha_m\,\bigl(\lambda^{(p)}_m\bigr)^k\bigl(\lambda^{(q)}_m\bigr)^l\\
&=\alpha_1 +\sum_{\text{even }m} \alpha_m\,\bigl(\lambda^{(p)}_m\bigr)^k\bigl(\lambda^{(q)}_m\bigr)^l\\
&= \alpha_1 -\sum_{\text{even }m} \frac{\cos^2\left(\omega_m/N\right)}{N\sin^4\left(\omega_m/N\right)}\,\left(1-4p\sin^2\left(\frac{\omega_m}{N}\right)\right)^k \left(1-4q\sin^2\left(\frac{\omega_m}{N}\right)\right)^l\\
&= \alpha_1 -\sum_{t=0}^k\sum_{s=0}^l\sum_{\text{even }m} \frac{\cos^2\left(\omega_m/N\right)}{N\sin^4\left(\omega_m/N\right)}\binom{k}{t}\binom{l}{s}\,\left(-4p\sin^2\left(\frac{\omega_m}{N}\right)\right)^t \left(-4q\sin^2\left(\frac{\omega_m}{N}\right)\right)^s\\
&= \alpha_1 -\sum_{t=0}^k\sum_{s=0}^l\frac{(-4p)^t(-4q)^s}{N}\binom{k}{t}\binom{l}{s}\sum_{\text{even }m}\cos^2\left(\frac{\omega_m}{N}\right)\sin^{2(t+s)-4}\left(\frac{\omega_m}{N}\right)\\
&= \alpha_1 -\sum_{d=0}^{k+l}\sum_{t=\max(0,d-l)}^{\min(d,k)}\frac{(-4p)^t(-4q)^{d-t}}{N}\binom{k}{t}\binom{l}{d-t}\sum_{\text{even }m}\cos^2\left(\frac{\omega_m}{N}\right)\sin^{2d-4}\left(\frac{\omega_m}{N}\right)\\
&= \underbrace{\alpha_1 -\frac{1}{N}\sum_{\text{even }m}\cos^2\left(\frac{\omega_m}{N}\right)\sin^{-4}\left(\frac{\omega_m}{N}\right)}_{=0}\\
&\quad +\frac{4pk+4ql}{N}\underbrace{\sum_{\text{even }m}\cos^2\left(\frac{\omega_m}{N}\right)\sin^{-2}\left(\frac{\omega_m}{N}\right)}_{=N(N-1)/2}\\
&\quad -\sum_{d=2}^{k+l}\sum_{t=\max(0,d-l)}^{\min(d,k)}\frac{(-4)^{d}p^tq^{d-t}}{N}\binom{k}{t}\binom{l}{d-t}\underbrace{\sum_{\text{even }m}\cos^2\left(\frac{\omega_m}{N}\right)\sin^{2d-4}\left(\frac{\omega_m}{N}\right)}_{=c_{d-2}N}.
\end{align*}
Therefore,
\begin{align*}
\mathrm{tr }(T_p^kT_q^l\,M) &= 2(N-1)(kp+lq) -\sum_{d=2}^{k+l}\sum_{t=\max(0,d-l)}^{\min(d,k)}(-4)^{d}c_{d-2}\binom{k}{t}\binom{l}{d-t}p^tq^{d-t}\\
&= 2(N-1)(kp+lq)  -\sum_{d=2}^{k+l}\sum_{t=\max(0,d-l)}^{\min(d,k)}(-1)^{d}\frac{4}{d-1} \binom{2d-4}{d-2}\binom{k}{t}\binom{l}{d-t}p^tq^{d-t}
\end{align*}
\end{proof}

Let us consider the behavior as $N \to \infty$. For fixed $m$, since $k = |i-j| \leq m$ the term $$
-4\sum_{t=2}^k\frac{(-1)^t}{t-1}\binom{k}{t} \binom{2t-4}{t-2}p^t $$
is upper bounded by a constant determined by $m$, denoted as $C_m$. Thus as $N$ grows this term is negligible. The same holds for $ -\sum_{d=2}^{k+l}\sum_{t=\max(0,d-l)}^{\min(d,k)}(-1)^{d}\frac{4}{d-1} \binom{2d-4}{d-2}\binom{k}{t}\binom{l}{d-t}p^tq^{d-t}$ since $k + l \leq 2m$. Thus we have when $t^*_m/m = t^*$, for all $i,j\in [m]$:
$$
\frac{N(N-1)}{2c_A^2m}\left(\mathcal{D}^{(2)}_{d_{MV}}\right)_{i,j} - \left(\mathcal{D}_Z\right)_{i,j} = 
\frac{N(N-1)}{2c_A^2m}\left(\mathcal{D}_{d^2_{MV}}\right)_{i,j} - \left(\mathcal{D}_Z\right)_{i,j} 
\leq \frac{C_m}{2m(N-1)} \sim O\left(\frac{1}{N}\right), 
$$ simply denote $D' := \frac{N(N-1)}{2c_A^2m}\left(\mathcal{D^A}^{(2)}_{d_{MV}}\right) = \frac{N(N-1)}{2c_A^2m}\left(\mathcal{D^A}_{d^2_{MV}}\right)$, we have
$$
\left(D'\right)_{i,j} - \left(\mathcal{D}_Z\right)_{i,j} \sim O\left(\frac{1}{N}\right),
$$
since every entry of both matrices are bounded by constant we also have
$$
\left(D'^{(2)}\right)_{i,j} - \left(\mathcal{D}^{(2)}_Z\right)_{i,j} \sim O\left(\frac{1}{N}\right).
$$

Let us denote the eigenvalues and the corresponding orthonormal eigenvectors of $-\frac{1}{2}H\mathcal{D}^{(2)}_ZH$ and $-\frac{1}{2}HD'^{(2)}H$ respectively by $\lambda_1, \lambda_2, \cdots, \lambda_m$ and $U_1,U_2,\cdots,U_m$; $\tilde{\lambda}_1, \tilde{\lambda}_2,\cdots, \tilde{\lambda}_m$ and $\tilde{U}_1,\tilde{U}_2,\cdots,\tilde{U}_m$.  
Note that for fixed $m$, both $\psi^1_{d^2_{MV}}$ and $\psi_Z$ are $m \times 1$ vectors as the first dimension of CMDS result of $\mathcal{D}_{d^2_{MV}}$ and $\mathcal{D}_Z$, and because $D' = \frac{N(N-1)}{2c_A^2m} \mathcal{D}_{d^2_{MV}}$,
we have:
$$
\frac{N(N-1)}{2c_A^2m}\psi^1_{d^2_{MV}} = \sqrt{\tilde{\lambda}_1} \tilde{U}_1, ~~~ \psi_Z = \sqrt{\lambda_1} U_1.
$$
Now we only need to show that: 
$$
\max_{i \in [m]}\bigg| \frac{N(N-1)}{2c_A^2m}\psi^1_{d^2_{MV}}(t_i) -  \psi_Z(t_i) \bigg| = \bigg\| \sqrt{\tilde{\lambda}_1} \tilde{U}_1 - \sqrt{\lambda_1} U_1\bigg\|_{2 \to \infty} \sim O\left(\frac{1}{N}\right).
$$
To show this we follow the proof in \cite{chen2024euclidean} Theorem 2.26. Note here $m$ is not growing as in \cite{chen2024euclidean} and it is treated as a constant. Denote  double centering matrix $H \in \RR^{m \times m} : = I -\frac{J}{m}$ where $I$ is the order $m$ identity matrix and $J$ is $m\times m$ all one matrix. Here we state several properties about $D'$ and $\mathcal{D}_Z$:
\begin{itemize}
    \item $\lambda_1 \sim \Theta(1)$ and $\lambda_i =0$ for $i \geq 2$; $ \left\Vert U_1 \right\Vert_{2 \to \infty} \sim \Theta(1)$.
    \item $\|D' -\mathcal{D}_Z\|_{\infty} \sim O\left(\frac{1}{N}\right)$, and $\|D'^{(2)} -\mathcal{D}_Z^{(2)}\|_{\infty} \sim O\left(\frac{1}{N}\right)$, $\|HD'^{(2)}H -H\mathcal{D}_Z^{(2)}H\|_{\infty} \sim O\left(\frac{1}{N}\right)$.
    \item $\left|\tilde{\lambda}_1-\lambda_1\right| \sim O\left(\frac{1}{N}\right)$,$\tilde{\lambda}_1 \sim \Theta(1)$, $\left|\sqrt{\tilde{\lambda}_1} -\sqrt{\lambda_1}\right| = \frac{\left|\tilde{\lambda}_1-\lambda_1\right|}{\sqrt{\tilde{\lambda}_1}+\sqrt{\lambda_1}} \sim \frac{O\left(\frac{1}{N}\right)}{\Theta\left(1\right)}  \sim O\left(\frac{1}{N}\right).$
\end{itemize}

Then because $\lambda_1 \sim \Theta(1)$ and  $\|HD'^{(2)}H -H\mathcal{D}_Z^{(2)}H\|_{\infty} \sim O\left(\frac{1}{N}\right)$, once $N$ is large enough, we will have $|\lambda_1|>4\left\Vert H\left(\mathcal{D}^{(2)}_{\varphi_m}-\mathcal{D}^{(2)}_Z\right)H\right\Vert_{\infty}$ and we can use Theorem 4.2 in \cite{cape2019two}. Thus there exists a $w=1$ or $w=-1$ such that:
\begin{align*}
&\left\Vert \sqrt{\tilde{\lambda}_1}\tilde{U}_1-w\sqrt{\lambda_1}U_1  \right\Vert_{2 \to \infty}  \leq \sqrt{ \tilde{\lambda}_1} \left\Vert \tilde{U}_1-wU_1\right\Vert_{2 \to \infty} + \left|\sqrt{\tilde{\lambda}_1} -\sqrt{\lambda_1}\right|\left\Vert U_1 \right\Vert_{2 \to \infty}  \\
&\leq  \sqrt{ \tilde{\lambda}_1}\frac{c\|HD'^{(2)}H -H\mathcal{D}_Z^{(2)}H\|_{\infty} \left\Vert U_1 \right\Vert_{2 \to \infty} }{\lambda_1}
+\left|\sqrt{\tilde{\lambda}_1} -\sqrt{\lambda_1}\right|\left\Vert U_1 \right\Vert_{2 \to \infty} \\
& \leq \left( \frac{ c \sqrt{ \tilde{\lambda}_1} \|HD'^{(2)}H -H\mathcal{D}_Z^{(2)}H\|_{\infty}}{\lambda_1}+\left|\sqrt{\tilde{\lambda}_1} -\sqrt{\lambda_1}\right| \right) \left\Vert U_1 \right\Vert_{2 \to \infty} \sim O\left(\frac{1}{N}\right). 
\end{align*}
Thus the result follows.

Now we consider CMDS on $\mathcal{D^A}_{\text{ind-}d_{MV}}$. Denote $a : = \frac{c_A^2}{6}\frac{N+1}{N-1}$, we note that $\mathcal{D^A}^{(2)}_{\text{ind-}d_{MV}} = a J - aI$ where $I$ is $m \times m$ identity matrix and $J$ is $m \times m$ all one matrix, then 
$$
-\frac{1}{2}H\mathcal{D^A}^{(2)}_{\text{ind-}d_{MV}}H = \frac{a}{2}H(I - J  )H = \frac{a}{2}\left( I - \frac{J}{m} \right)\left(I - J  \right) \left( I - \frac{J}{m} \right)    = \frac{a}{2}H.
$$
Note matrix $H= I - \frac{J}{m}$ is idempotent and has eigenvalues $1$ and $0$ with geometric multiplicity $m-1$ and $1$ respectively. Thus eigenvalues of $\frac{a}{2}H$ are $a/2 > 0$ and $0$ and the first $m-1$ eigenvectors correspond to the eigenvalue $a/2$. Then any orthonormal basis of $\RR^{m-1}$ can serve as these eigenvectors, denoted as $U_{m-1}$. The results follow.

For $ \mathcal{D^A}_{\alpha -d_{MV}}^{(2)}$. Denote $B^{\mathcal{A}}_{d_{MV}} = -\frac{1}{2} H \mathcal{D^A}_{d_{MV}}^{(2)}  H$. Thus 
$$
-\frac{1}{2} H \mathcal{D^A}_{\alpha -d_{MV}}^{(2)} H = 
-\frac{1}{2} H \left( \alpha \mathcal{D^A}_{\text{ind-}d_{MV}}^{(2)} +(1-\alpha) \mathcal{D^A}_{d_{MV}}^{(2)}  \right) H
= (1-\alpha)B^{\mathcal{A}}_{d_{MV}}  + \frac{\alpha a }{2}H.
$$ 
Now denote the eigenvalues of $B^{\mathcal{A}}_{d_{MV}}$ as $\lambda_1>\lambda_2>\cdots>\lambda_m = 0$ and $U_1$,\ldots,$U_m$. Then for any orthonormal eigenvector of $B^{\mathcal{A}}_{d_{MV}}$ that is orthogonal to $\mathbf{1}$: $U_i$ such that $U^{\top}_i\mathbf{1} = 0$ because $B^{\mathcal{A}}_{d_{MV}} U_i  = \lambda_i U_i$ then
$$
\left(-\frac{1}{2} H \mathcal{D^A}_{\alpha -d_{MV}}^{(2)} H \right)U_i = \left( (1- \alpha) B^{\mathcal{A}}_{d_{MV}} + \frac{a\alpha}{2}H \right) U_i = \left((1-\alpha)\lambda_i + \frac{a \alpha}{2}\right)U_i.
$$ 
Those $U_i$ that are  orthogonal to $\mathbf{1}$ are also eigenvectors of $-\frac{1}{2} H \mathcal{D^A}_{\alpha -d_{MV}}^{(2)} H $ and correspond to the eigenvalue $(1-\alpha)\lambda_i + a \alpha/2$. So when $\lambda_1$ is the most positive eigenvalue of $B_{d_{\mathrm{MV}}}^{\mathcal{A}}$, $(1-\alpha)\lambda_1 + \frac{a \alpha}{2}$ will be the most positive eigenvalue of $-\frac{1}{2}H\mathcal{D}^{\mathcal{A}^{(2)}}_{\alpha-d_{\mathrm{MV}}}H.$ Thus 
$$
\psi^A_{d_{MV}} = \sqrt{\lambda_1} U_1, ~~ \psi^A_{\alpha -d_{MV}} = \sqrt{ (1-\alpha)\lambda_1 + \frac{a \alpha}{2} } U_1.
$$

For the $W_1$ distance, we know $X^A_{t_i} \sim u$ for all $i \in[m]$, thus $W_1(X^A_{t_i},X^A_{t_j}) = 0$ for all $i,j \in [m]$. 

For the expected average degree, since in Atlanta model, for any $i\in[m]$: $\EE[X_{t_i}] = \EE [u] = \frac{c_A}{2}$, the results follow. 

\end{proof}

\subsection{Proof of Theorem \ref{thm:W1_continuous}}

\begin{proof}
Suppose 
\begin{equation}
\label{eq:wmin2inf}
\min_W \|\bU-\bX W\|_{2\rightarrow\infty}\leq \epsilon, \min_W \|\bV-\bY W\|_{2\rightarrow\infty}\leq \epsilon,
\end{equation}
and recall that 
$$\procwphat[1](\bX,\bY)=\min_W \min_\sigma \frac{1}{n}\sum_{i=1}^n \|\bX^i-W \bY^{\sigma(i)}\|.$$
We will show that
$$ |\procwphat[1](\bU,\bV)-\procwphat[1](\bX,\bY)|\leq 2\epsilon.$$
Let $\sigma:[n]\rightarrow[n]$ be the optimal permutation and $W_{XY}$ be the optimal orthogonal matrix for $\bX,\bY$, so that
$$\procwphat[1](\bX,\bY)=\frac{1}{n}\sum_{i=1}^n \|\bX^i-W_{XY}\bY^{\sigma(i)}\|.$$
Then since this permutation is one possible choice of alignment for $\bU,\bV$, but may not be optimal, we have
\begin{align*}
\procwphat[1](\bU,\bV)&=\min_{\sigma'}\min_{W} \frac{1}{n}\sum_{i=1}^n \|\bU^i-W\bV^{\sigma'(i)}\|\\
&\leq \min_W \frac{1}{n}\sum_{i=1}^n \|\bU^i-W\bV^{\sigma(i)}\|\\
&\leq \min_{W,W_1,W_2} \frac{1}{n}\sum_{i=1}^n \left(\|\bU^i-W_1\bX^i\|+\|W_1 \bX^i-W_2\bY^{\sigma(i)}\|+\|W_2\bY^{\sigma(i)}-W\bV^{\sigma(i)}\|\right).
\end{align*}
Let $W_1'$ achieve $\|\bU-\bX W_1'\|_{2\rightarrow\infty}\leq \epsilon$, as guaranteed to exist by Equation~\ref{eq:wmin2inf}. Then we select $W_2'=(W_1')^\top W_{XY}$, and $W'=W_2' (W_3')^\top $, where $W_3'$ achieves $\|\bV-\bY W_3'\|_{2\rightarrow\infty}\leq \epsilon$ (again by Equation~\ref{eq:wmin2inf}). Since these are three particular choices of orthogonal matrices, we have
\begin{align*}
\procwphat[1](\bU,\bV)&\leq \frac{1}{n}\sum_{i=1}^n \left(\|\bU-\bX W_1'\|_{2\rightarrow\infty}+\|(W_1')^\top(\bX^i-W_{XY} \bY^{\sigma(i)})\|+\|(\bY-\bV W_3')(W_2')^\top\|_{2\rightarrow\infty}\right)\\
&\leq \procwphat[1](\bX,\bY)+2\epsilon.
\end{align*}
For any $\sigma$ and $W$, after choosing $W_1$, $W_2$ in light of Equation~\ref{eq:wmin2inf}, we have
\begin{align*}
\frac{1}{n}\sum_{i=1}^n \|\bU^i-W\bV^{\sigma(i)}\|&\geq \frac{1}{n}\sum_{i=1}^n \left(\|W_1 \bX^i-WW_2\bY^{\sigma(i)}\|-\|\bU^i-W_1\bX^i\|-\|W(W_2\bY^{\sigma(i)}-\bV^{\sigma(i)})\|\right)\\
&\geq \frac{1}{n}\sum_{i=1}^n \|W_1\bX^i-WW_2\bY^{\sigma(i)}\|-2\epsilon\\
&\geq \procwphat[1](\bX,\bY)-2\epsilon.
\end{align*}
Now minimizing over $\sigma$ and $W$ gives 
$$
\procwphat[1](\bU,\bV)\geq \procwphat[1](\bX,\bY)-2\epsilon.
$$
\end{proof}

\subsection{Proof of Theorem \ref{thm:W1_as_converge}}
We will make use of the following lemma: 
\begin{lemma}
\label{lem:supsclose}
For a metric space $(\mathcal{M},d)$ and functions $L_1,L_2: \mathcal{M}\rightarrow\RR, $ if $|L_1(x)-L_2(x)|\leq \epsilon$ for all $x\in\mathcal{M}$, then 
$$\left| \sup_{x\in\mathcal{M}} L_1(x) - \sup_{y\in\mathcal{M}} L_2(y)\right|\leq \epsilon.$$
\end{lemma}
\begin{proof}
Indeed, let $\delta>0$, and let $x^*\in\mathcal{M}$ satisfy $L_1(x^*) \geq \sup_{x\in\mathcal{M}} L_1(x)-\delta$. Then 
\begin{align*}
\sup_{y\in \mathcal{M}} L_2(y) & \geq L_2(x^*)\\
&\geq L_1(x^*)-\epsilon\\
&\geq \sup_{x\in \mathcal{M}} L_1(x) - \epsilon - \delta.
\end{align*}
Since the choice of $\delta>0$ was arbitrary, we have
$$ \sup_{y\in \mathcal{M}} L_2(y) \geq \sup_{x\in \mathcal{M}}L_1(x)-\epsilon.$$ Now observe that the assumptions on $L_1$ and $L_2$ are symmetric, so we must also have 
$$ \sup_{x\in \mathcal{M}} L_1(x) \geq \sup_{y\in \mathcal{M}} L_2(y) - \epsilon,$$ which completes the proof of the lemma.
\end{proof}

In \cite{chewi2025statistical}, the following is Lemma 2.7:
\begin{lemma}
\label{lem:f1coveringnumber}
    Let $\mathcal{F}_1$ be the collection of 1-Lipschitz functions from $[0,1]^d\rightarrow\RR$ with $f(0)=0$. Then $$\mathcal{N}(\epsilon,\mathcal{F}_1,\|\cdot\|_{\infty})\leq \exp\left[C\left(\frac{4\sqrt{d}}{\epsilon}\right)^d\right].$$
\end{lemma}

\begin{proof}
For $n\geq 1$, $f\in \mathcal{F}_1$, and some real orthogonal matrix $W$, let
\begin{align*}
L_{1,n}(f,W)&= \EE_{x\sim \mu_\bX} f(x) - \EE_{y\sim \mu_\bY} f(Wy),\\
L_2(f,W) &= \EE_{x\sim \mu} f(x) - \EE_{y\sim \nu} f(Wy).
\end{align*}
Recall that by the Kantorovich dual formulation of the Wasserstein-1 problem, 
\begin{align*}
\procwphat[1](\bX,\bY)&=\procwp[1](\mu_{\bX},\mu_{\bY})=\min_W W_1(\mu_\bX,\mu_{\bY W^\top})= \min_W \sup_{f\in\mathcal{F}_1} L_{1,n}(f,W),\\
\procwp[1](\mu_X,\mu_Y)&=\min_W W_1(\mu_X,\mu_{Y}\circ W^\top)=\min_W\sup_{f\in\mathcal{F}_1}L_2(f,W).
\end{align*}
Since $L_{1,n}(f,W)=L_{1,n}(f+c,W)$ and $L_2(f,W)=L_2(f+c,W)$ for any constant $c$, there is no loss of generality in taking $f(0)=0$. On a set of full measure $\Omega_{f,W}$, $L_{1,n}(f,W)\rightarrow L_2(f,W)$ as $n\rightarrow\infty$ by the ordinary strong law of large numbers, since all $X_i\overset{\text{iid}}{\sim} \mu_X$, $Y_j\overset{\text{iid}}{\sim} \mu_Y$, and the Lipschitz condition on $f$, plus the compactness of $B_d$, are sufficient to guarantee that $\EE_{x\sim \mu} f(x), \EE_{y\sim \nu} f(Wy)$ are finite. We want to show that there is a set of full measure on which this convergence takes place uniformly in $(f,W)$, which we will accomplish using a covering for $\mathcal{F}_1\times \mathcal{O}^{d}$. 

Consider a covering for $\mathcal{F}_1\times \mathcal{O}^d$ of the form $\bigcup_{i=1}^N B((f_i,W_i),\epsilon)$, where $f_i \in\mathcal{F}_1$, $W_i\in \mathcal{O}^d$, and $B((f_i,W_i), \epsilon) = \{(g,W)\in \mathcal{F}_1\times \mathcal{O}^d: 2\|g-f_i\|_{\infty}+\|W-W_i\| \leq \epsilon\}.$ The existence of such a covering is guaranteed by compactness of $\mathcal{F}_1\times \mathcal{O}^d$, where the compactness of $\mathcal{F}_1$ is supplied by Lemma~\ref{lem:f1coveringnumber}. The open cover $\{B((f,W),\epsilon); (f,W)\in\mathcal{F}_1\times \mathcal{O}^d\}$ then has a finite subcover by compactness, which gives a covering of the desired form. Then for any $n$, $(g,R)\in B((f,W),\epsilon)$, and $\{x_i\}_{i=1}^n, \{y_i\}_{i=1}^n\subset B_d,$ we have
\begin{align*}
|L_{1,n}(g,R)-L_{1,n}(f,W)|&=\left|\frac{1}{n}\sum_{i=1}^n [g(x_i)-g(Ry_i)] - \frac{1}{n}\sum_{j=1}^n [f(x_j)-f(Wy_j)]\right|\\
&\leq \frac{1}{n}\sum_{i=1}^n|g(x_i)-f(x_i)|+|g(Ry_i)-g(Wy_i)|+|g(Wy_i)-f(Wy_i)|\\
&\leq \frac{1}{n}\sum_{i=1}^n(2\|g-f\|_{\infty}+\|Ry_i-Wy_i\|)\\
&\leq 2\|g-f\|_{\infty}+\|R-W\|\leq \epsilon,
\end{align*}
and $|L_2(g,R)-L_2(f,W)|\leq \epsilon$ by a similar argument. So for a given $\epsilon>0$, select a covering for $\mathcal{F}_1\times \mathcal{O}^d$ of the given type. For each $1\leq i\leq N(\epsilon)$, let $\Omega_i$ be the set of full measure for which $L_{1,n}(f_i,W_i)\rightarrow L_2(f_i,W_i)$ as $n\rightarrow\infty$. Then for each $i$, for any $\omega\in \Omega_i$, there is a sufficiently large $K_i=K_i(\epsilon,\omega)$ such that for all $k\geq K_i$, $|L_{1,k}(f_i,W_i)-L_2(f_i,W_i)|\leq \epsilon.$ So for any $\omega\in \Omega(\epsilon):=\bigcap_{i=1}^N \Omega_i$ and $k\geq K(\epsilon):=\max_{1\leq i\leq N} K_i(\epsilon,\omega)$, we have for $(g,R)\in B((f_i,W_i),\epsilon)$:
\begin{align*}
    |L_{1,k}(g,R)-L_2(g,R)| &\leq |L_{1,k}(g,R)-L_{1,k}(f_i,W_i)| \\
    &\quad+|L_{1,k}(f_i,W_i)-L_2(f_i,W_i)|+|L_2(f_i,W_i)-L_2(g,R)| \leq 3\epsilon.
\end{align*} 
But since $\{B((f_i,W_i),\epsilon)\}$ covers $\mathcal{F}_1\times \mathcal{O}^d$, this bound holds for all $(g,R)\in \mathcal{F}_1\times \mathcal{O}^d$. Put briefly, for $\omega\in\Omega(\epsilon)$, once $k\geq K(\epsilon)$, every $(g,R)\in\mathcal{F}_1\times \mathcal{O}^d$ satisfies $$|L_{1,k}(g,R)-L_2(g,R)| \leq 3\epsilon.$$ Consider the sequence $\epsilon_m=1/m$, and define $\Omega:= \bigcap_{m=1}^{\infty}\Omega(1/m)$, which is also a set of full measure. Then for any $\omega\in \Omega$ and $\epsilon>0$, letting $m>1/\epsilon$, $\omega \in \Omega(1/m)$, so there exists $K(1/m)$ so large that for any $k\geq K(1/m)$, every $(g,R)\in\mathcal{F}_1\times \mathcal{O}^d$ satisfies $$|L_{1,k}(g,R)-L_2(g,R)|\leq 3/m \leq 3\epsilon.$$ We apply Lemma~\ref{lem:supsclose} to get 
$$ \left|\sup_{f\in\mathcal{F}_1}L_{1,n}(f,W)-\sup_{g\in\mathcal{F}_1}L_2(g,W)\right|\leq 3\epsilon,$$ and apply Lemma~\ref{lem:supsclose} again to get
$$|\procwphat[1](\bX,\bY)-\procwp[1](\mu_X,\mu_Y)|=\left|\inf_W \sup_{f\in \mathcal{F}_1} f_{1,n}(f,W) - \inf_R \sup_{g\in\mathcal{F}_1} L_2(g,R)\right|\leq 3\epsilon. $$ Since $\epsilon>0$ was arbitrary, on $\Omega$ we have $$\procwphat[1](\bX,\bY)\rightarrow \procwp[1](\mu_X,\mu_Y),$$
as required.

Now we wish to prove the concentration result. Since $\mu_X, \mu_Y$ are distributions on $B_d^+$, $W\in\mathcal{O}^d$, and $f$ is 1-Lipschitz, we may apply Hoeffding's inequality:
\begin{align*}
\PP\left[\left|\frac{1}{n}\sum_{i=1}^n f(X_i) - \EE_{X\sim\mu_X} f(X)\right|>t\right]&\leq 2\exp\left(\frac{-nt^2}{2}\right),\\
\PP\left[\left|\frac{1}{n}\sum_{j=1}^n f(WY_j) - \EE_{Y\sim\mu_Y} f(WY)\right|>t\right]&\leq 2\exp\left(\frac{-nt^2}{2}\right).\\
\end{align*}
As such, for a given $(f,W)\in\mathcal{F}_1\times \mathcal{O}^d$, 
$$ \PP[|L_{1,n}(f,W)-L_2(f,W)|>t]\leq 4\exp(-nt^2/2).$$ Consider the $\epsilon$-covering for $\mathcal{F}_1\times \mathcal{O}^d$ with $\epsilon=t$. Then from the consistency argument, we have that
\begin{equation}
\label{eq:procwassprobbound}
\PP\left[\forall\;(g,R)\in\mathcal{F}_1\times\mathcal{O}^d\; |L_{1,n}(g,R)-L_2(g,R)|\leq 3t\right]\geq 1- \mathcal{N}(t,\mathcal{F}_1\times\mathcal{O}^d,d)[4\exp(-nt^2/2)].\end{equation}

Using Lemma~\ref{lem:f1coveringnumber}, we may bound 
\begin{align*}
\mathcal{N}(t,\mathcal{F}_1\times\mathcal{O}^d,d) &\leq \mathcal{N}(t/4,\mathcal{F}_1,\|\cdot\|_{\infty})\mathcal{N}(t/2,\mathcal{O}^d,\|\cdot\|)\\
&\leq \exp\left[C\left(\frac{4\sqrt{d}}{t}\right)^d\right] \left(1+\frac{4\sqrt{d}}{t}\right)^{d^2},
\end{align*}
where we bounded the latter covering number as follows:
$$ \mathcal{N}(\epsilon,\mathcal{O}^d,\|\cdot\|)\leq \mathcal{N}(\epsilon,\mathcal{O}^d,\|\cdot\|_F)\leq \mathcal{N}(\epsilon/\sqrt{d},B_d,\|\cdot\|)^d,$$ since any real orthogonal matrix has $d$ columns in $B_d$.

So the log of the complementary probability on the right hand side of Equation~\ref{eq:procwassprobbound} when $t=C_d n^{-1/(d+2)}$ for a sufficiently large constant $C_d$ is  
\begin{align*}
\log&\left[\mathcal{N}(t,\mathcal{F}_1\times\mathcal{O}^d,d)[4\exp(-nt^2/2)]\right]\\
&\leq C\left(\frac{4\sqrt{d}}{t}\right)^d + d^2\log\left(1+\frac{4\sqrt{d}}{t}\right)-\frac{nt^2}{2}+\log(4)\\
&\sim (C(4\sqrt{d}/C_d)^d+1) n^{d/(d+2)} - C_d n^{d/(d+2)}/2\leq -n^{d/(d+2)}.
\end{align*}

Theorem 1 from \cite{rubin2022statistical} proves convergence of the ASE estimates of the latent position matrices to the true ones with respect to the $\|\cdot\|_{2\rightarrow\infty}$ norm. They suppose that scaled latent positions $\{\xi_i\}_{i=1}^n$ are an i.i.d. sample drawn from some latent position distribution $F$ on $\RR^d$, where $F$ has a second moment matrix $\EE_F[\xi\xi^\top]$ that is positive definite with rank $d$, and that the latent positions are defined as $X_i=\sqrt{\rho_n}\xi_i$ for some quantity $\rho_n\in(0,1]$. We will consider the case that $\rho_n\equiv 1$. They further suppose that $A_X\in\{0,1\}^{n\times n}$ is drawn from a generalized random dot product graph, so that $\EE[A_X]= \bX I_{p,q} \bX^\top$; we will consider the special case that $p=d$ and $q=0$. They define a universal constant $c>1$ which gives a power of the logarithmic factor in their assumptions, as well as in the upper bound for the $\|\cdot\|_{2\rightarrow\infty}$ norm. With our simplifying assumptions, the theorem then reads that for any $a>0$, there is a constant $C_a>0$ such that for sufficiently large $n$, with probability at least $1-n^{-a}$,
$$
\min_{W\in\mathcal{O}^d}\|\hat{X}-XW\|_{2,\infty} \leq C_a \frac{\log^c(n)}{\sqrt{n}}.
$$
We note that a more stringent bound might be obtained from Corollary 4.1 of \cite{xie2024entrywise}, but either bound reveals that this is a lower-order term in the present setting.

As such, given $A\sim\mathrm{RDPG}(\bX), B\sim\mathrm{RDPG}(\bY)$ with ASEs $\hat{\bX}, \hat{\bY}$, there are events $\mathcal{A}_n,\mathcal{B}_n$ each with probability at least $1-n^{-a}$ such that
\begin{align*}
\text{On }\mathcal{A}_n: &\min_{W\in\mathcal{O}^d}\|\hat{\bX}-\bX W\|_{2\rightarrow\infty} \leq C_a \frac{\log^c(n)}{\sqrt{n}},\\
\text{On }\mathcal{B}_n: &\min_{W\in\mathcal{O}^d}\|\hat{\bY}-\bY W\|_{2\rightarrow\infty} \leq C_a' \frac{\log^c(n)}{\sqrt{n}}.\\
\end{align*}
So on $\mathcal{A}_n\cap \mathcal{B}_n$, Theorem~\ref{thm:W1_continuous} gives us
$$|\procwp[1](\hat{\mu}_n,\hat{\nu}_n)-\procwp[1](\mu_n,\nu_n)|\leq C_a'' \frac{\log^c(n)}{\sqrt{n}}.$$

From the first part of this theorem, there is an event $\mathcal{C}_n$ with probability at least $1-\exp(-n^{d/(d+2)})$ such that
$$
|\procwp[1](\mu_n,\nu_n)-\procwp[1](\mu,\nu)|\leq C_d n^{-1/(d+2)}.
$$
So with probability at least $1-3n^{-a}$,
$$
|\procwp[1](\hat{\mu}_n,\hat{\nu}_n)-\procwp[1](\mu,\nu)|\leq (C_d+1)n^{-1/(d+2)}.
$$


\end{proof}



\end{bibunit}

\end{document}